\documentclass[a4paper,twoside]{article}  

\usepackage{amssymb,amsmath,amsthm,dsfont,amsfonts,color,latexsym}   

\usepackage{hyperref}
\usepackage{mathrsfs} 

\definecolor{blue}{rgb}{0.00,0.00,1.00}
\definecolor{red}{rgb}{1.00,0.00,0.00}

\renewcommand{\baselinestretch}{1.2}
\def\bq{\begin{equation}}
\def\eq{\end{equation}}
\def\ba{\begin{array}{ccc}}
\def\bal{\begin{array}{lll}}
\def\ea{\end{array}}

 \def\lt#1{\left#1}\def\rt#1{\right#1}
\def\({\left(}\def\){\right)}
\def\[{\left[}\def\]{\right]}

    \def \U   {\mathcal{U}}
    \def \W   {\mathcal{W}}
    \def \RR   {\mathcal{R}}
    \def \C   {\mathbb{C}}
    \def \O   {\mathbb{O}}
    
    \def \R   {\mathbb{R}}

    \def\M    {\mathbb{M}}
    \def\Q    {\mathbb{Q}}
    
    \def\y    {\mathbf{y}}
    \def\z    {\mathbf{z}}
    \def\e    {\mathbf{e}}

    \def\S    {\mathbb{S}}
    
    \def\eps  {\epsilon}
    
    \def\intr {\int_{\R^3}}
    \def\intra {\int_{\R}}
    \def\ints {\int_{\S^2}}
    \def\intt {\int^t_0}

    \def \Z    {\mathrm{Z}}
    \def \Y    {\mathrm{Y}}
    \def \X    {\mathrm{X}}
    \def \V    {\mathrm{V}}
    \def \N    {\mathbb{N}}

    \def \Dt   {\frac{\rm d}{{\rm d}t}}
    
    \def \dt    {\partial_t}

    \def \dx    {\partial_x}

    \def \dxa   {\partial^{\alpha}_x}
    
    \def \dv    {\partial_v}
    
    \def \dvb   {\partial^{\beta}_v}

    \def\Tdx   {\nabla_x}
    \def\Tdv   {\nabla_v}

       \def\be{\begin{equation}}
       \def\ee{\end{equation}}
       \def\bma#1\ema{{\allowdisplaybreaks\begin{align}#1\end{align}}}
       \def\bmas#1\emas{{\allowdisplaybreaks\begin{align*}#1\end{align*}}}
       \def\bln#1\eln{{\allowdisplaybreaks\begin{aligned}#1\end{aligned}}}
       \def\nnm{\notag}
       \def\bgr#1\egr{\allowdisplaybreaks\begin{gather}#1\end{gather}}
       \def\bgrs#1\egrs{\allowdisplaybreaks\begin{gather*}#1\end{gather*}}

       \theoremstyle{plain}
       \newtheorem{lem}{\bf Lemma}[section]
       \newtheorem{thm}[lem]{\textbf{Theorem}}

       \newtheorem{cor}[lem]{\textbf{Corollary}}
       
\begin{document}


\title{Green's Function and Pointwise Behavior of the One-Dimensional Vlasov-Maxwell-Boltzmann System}

\author{ Hai-Liang Li$^{1,2}$,\, Tong Yang$^3$,\, Mingying Zhong$^4$\\[2mm]
 \emph
    {\small\it  $^1$School of  Mathematical Sciences, Capital Normal University, } \\
    {\small\it $^2$Academy For Multidisciplinary Studies, Capital Normal University, } \\
    {\small\it Beijing 100048, P.R. China.}\\
    {\small\it E-mail:\ hailiang.li.math@gmail.com}\\
    {\small\it $^3$Department of Applied Mathematics, The Hong Kong Polytechnic University,} \\
    {\small\it Hung Hom, Hong Kong, P.R.  China.}\\
    {\small\it E-mail: t.yang@polyu.edu.hk} \\
    {\small\it  $^4$School of  Mathematics and Information Sciences, Guangxi University,} \\
    {\small\it Nanning 530004, P.R. China.}\\
    {\small\it E-mail:\ zhongmingying@sina.com}\\[5mm]
    }
\date{ }

\pagestyle{myheadings}
\markboth{Vlasov-Maxwell-Boltzmann System}%
{H.-L. Li, T. Yang,  M.-Y. Zhong}

 \maketitle

 \thispagestyle{empty}

\begin{abstract}\noindent
 The pointwise space-time behavior of the Green's function of the one-dimensional Vlasov-Maxwell-Boltzmann (VMB) system  is studied in this paper. It is shown that  the Green's function  consists of  the macroscopic  diffusive waves and Huygens waves with the speed $\pm \sqrt{\frac53}$ at low-frequency, the  hyperbolic waves  with the speed $\pm 1$ at high-frequency, the singular kinetic  and  leading short waves, and the remaining term decaying exponentially in space and time.  Note that these high-frequency hyperbolic waves are completely new and can not be observed for the Boltzmann equation and the Vlasov-Poisson-Boltzmann system. In addition, we establish the pointwise space-time estimate of the global solution to the nonlinear VMB system based on the Green's function. Compared to the Boltzmann equation and the Vlasov-Poisson-Boltzmann system, some new ideas  are introduced to overcome the difficulties caused by the coupling effects of the transport of particles and the rotating of electro-magnetic fields, and investigate the new hyperbolic waves and singular leading short waves.


\end{abstract}

%
\tableofcontents

\section{Introduction}
\label{sect1}
\setcounter{equation}{0}

The Vlasov-Maxwell-Boltzmann (VMB) system is a fundamental model in plasma physics for the describing the time evolution of dilute charged particles,
such as electrons and ions,  under the influence of the self-induced Lorentz forces governed by Maxwell equations, cf. \cite{ChapmanCowling,Markowich}
for derivation and the physical background. In this paper, we consider the two-species Vlasov-Maxwell-Boltzmann system that describes both the time evolution of ions and electrons
\be
 \left\{\bln     \label{VMB1z}
 &  \dt F_++v_1\dx F_++(E+v\times B)\cdot\Tdv F_+ =Q(F_+,F_+)+Q(F_+,F_-), \\
 &  \dt F_-+v_1\dx F_--(E+v\times B)\cdot\Tdv F_- =Q(F_-,F_-)+Q(F_-,F_+), \\
 & \dt E_1=-\intr (F_+-F_-)v_1dv,\quad \dx E_1=\intr (F_+-F_-)dv,\\
 &\dt E_2=-\dx B_3-\intr (F_+-F_-)v_2dv,\\
 & \dt E_3=\dx B_2-\intr (F_+-F_-)v_3dv,\\
 & \dt B_2=\dx E_3,\quad  \dt B_3=-\dx E_2,\quad B_1\equiv0,
 \eln\right.
\ee
where $F_{\pm}=F_{\pm}(t,x,v)$ are number
density distribution functions of ions $(+)$ and electrons $(-)$ having position $x\in \R$ and velocity $v=(v_1,v_2,v_3)\in \R^3$ at time $t\in \R^+$, and $E(t,x)=(E_1,E_2,E_3)$, $B(t,x)=(B_1,B_2,B_3)$ denote the electro and magnetic fields, respectively. The operator $Q(f,g)$ describing
 the binary elastic collisions  for
the hard-sphere model is given by
 \bq
 Q(f,g)=\intr\ints  |(v-v_*)\cdot\omega|(f(v'_*)g(v')-f(v_*)g(v))dv_*d\omega,
 \eq
where
$$
 v'=v-[(v-v_*)\cdot\omega]\omega,\quad v'_*=v_*+[(v-v_*)\cdot\omega]\omega,\quad \omega\in\S^2.
$$

The Vlasov-Maxwell-Boltzmann system has been intensively studied and important progress has been made, cf.  \cite{Duan4,Duan5,Guo4,Jang,Jiang,Li1,Strain} and the references therein. For instance, the global existence of unique strong solution with initial data near the normalized global Maxwellian was obtained in spatial period domain~\cite{Guo4} and  in spatial 3D whole space  \cite{Strain} for hard sphere collision, and then in \cite{Duan1,Duan2} for general collision kernels with or without angular cut-off assumption.
The  long time behavior of the global solution near a global Maxwellian was studied  in \cite{Duan4,Duan5}.
The spectrum analysis and the optimal decay rate of the global solution to the VMB systems for both one-spices and two-spices were studied in \cite{Li1}. The fluid dynamic limit for the VMB system  was investigated in \cite{Jang,Jiang,Guo1}. And the pointwise estimates of the Green's functions  for the Boltzmann equation were given in the important papers \cite{Liu1,Liu2}. In addition, the
Green's function   for the Vlasov-Poisson-Boltzmann (VPB) system was studied in our previous work \cite{Li4}.

In this paper, we study the pointwise space-time behaviors of the Green's function  and the global solution to the VMB system~\eqref{VMB1z}  based on the spectral analysis~\cite{Li1}.

Let $ F_1=F_++F_-$ and $ F_2=F_+-F_-$. Then the system \eqref{VMB1z} becomes
\be
 \left\{\bln     \label{VMB2}
&\dt F_1+v_1\dx F_1+(E+v\times B)\cdot\Tdv F_2
=Q(F_1,F_1),\\
&\dt F_2+v_1\dx F_2+(E+v\times B)\cdot\Tdv F_1
=Q(F_2,F_1),\\
& \dt E=\dx \mathbb{O} B-\intr F_2v dv,\quad \dx E_1=\intr F_2dv,\\
 & \dt B=-\dx \mathbb{O} E,\quad B_1\equiv0,
\eln\right.
 \ee
 where $\O$ is a $3\times 3$ matrix defined by
 \be
 \O=\left(\ba 0 & 0 & 0\\ 0 & 0 & -1 \\ 0 & 1 & 0 \ea\right).
 \ee
 To study the pointwise space-time behavior of both the Green's function and the global solution to  the VMB system~\eqref{VMB2}, let us consider the Cauchy problem for the system  with  initial data
\be
F_i(0,x,v)=F_{i,0}(x,v),\,\,\, i=1,2, \quad  E(0,x)=E_0(x),\quad B(0,x)=B_0(x), \label{VMB3}
\ee
with the compatibility  condition
\be \dx E^1_0(x)=\intr F_{2,0}dv,\quad B^1_0(x)=0. \label{com}\ee
Note that the VMB system \eqref{VMB2} has a steady state $(F^*_1,F^*_2,E^*,B^*)=(M(v),0,0,0)$, where the normalized Maxwellian $M(v)$ is given by
$$
 M=M(v)=\frac1{\sqrt{2\pi}}e^{-\frac{|v|^2}2},\quad v\in\R^3.
$$

Set
$$F_1=M+\sqrt M f_1,\quad F_2=\sqrt Mf_2.$$
Then the  system \eqref{VMB2}, \eqref{VMB3} and \eqref{com} for $(F_1,F_2,E,B)$ can be written as the following system for $(f_1,f_2,E,B)$:
\bma
&\dt f_1+v_1\dx f_1-Lf_1
  =\frac12(v\cdot E)f_2-(E+v\times B)\cdot \Tdv f_2+\Gamma(f_1,f_1),\label{VMB3a}
  \\
&\dt f_2+v_1\dx f_2-L_1f_2-v\sqrt M\cdot E
  =\frac12(v\cdot E)f_1-(E+v\times B)\cdot \Tdv f_1+\Gamma(f_2,f_1),  \label{VMB4a}
  \\
&\dt E=\dx \O B-\intr f_2v\sqrt Mdv,\quad \dx E_1=\intr f_2\sqrt Mdv,   \label{VMB3b}
\\
&\dt B=-\dx \O E,\quad B_1\equiv0,    \label{VMB3d}
  \ema
 with the  initial data
\be
\left\{\bln
&f_i(0,x,v)=f_{i,0}(x,v),\,\,\,  i=1,2, \quad E(0,x)=E_0(x),\quad  B(0,x)=B_0(x),\label{VMB3e}\\
&\dx E^1_0(x)=\intr f_{2,0}\sqrt{M}dv,\quad B^1_0(x)=0,   
\eln\right.
\ee
where the linearized collision operators $L$ and $L_1$  and the nonlinear operator $\Gamma $ are defined by
 \bma
&Lf=\frac1{\sqrt M}[Q(M,\sqrt{M}f)+Q(\sqrt{M}f,M)], \label{VMB3x}         \\
&L_1f=\frac1{\sqrt M}Q(M,\sqrt{M}f),                \label{VMB3y}
\\
&\Gamma(f,g)=\frac1{\sqrt M}Q(\sqrt{M}f,\sqrt{M}g).  \label{VMB3z}
 \ema
It is well-known that (cf. \cite{Yu,Cercignani})
 \be \left\{\bln
 (Lf)(v)&=(Kf)(v)-\nu(v) f(v), \quad
 (Kf)(v)=\intr k(v,u)f(u)du,          \label{VMB3v}             \\
 (L_1f)(v)&=(K_1f)(v)-\nu(v) f(v), \quad
 (K_1f)(v)=\intr k_1(v,u)f(u)du,  \\
\nu(v)&= \sqrt{2\pi} \bigg(e^{-\frac{|v|^2}2}+ \(|v|+\frac1{|v|}\)\int^{|v|}_0e^{-\frac{|u|^2}2}du\bigg),\\
k(v,u)&=\frac2{\sqrt{2\pi}|v-u|}e^{-\frac{(|v|^2-|u|^2)^2}{8|v-u|^2}-\frac{|v-u|^2}8}-\frac{|v-u|}{2\sqrt{2\pi}}e^{-\frac{|v|^2+|u|^2}4},\\
k_1(v,u)&=\frac2{\sqrt{2\pi}|v-u|}e^{-\frac{(|v|^2-|u|^2)^2}{8|v-u|^2}-\frac{|v-u|^2}8},
\eln\right.
 \ee
where $\nu(v)$, the collision frequency, is a real function, $K$ and $K_1$ are self-adjoint compact operators on $L^2(\R^3_v)$ with real symmetric integral kernels $k(v,u)$ and $k_1(v,u)$.
The null space of the operator $L$, denoted by $N_0$, is a subspace
spanned by the orthonoraml basis $\{\chi_j,\ j=0,1,\cdots,4\}$  given by
\bq \chi_0=\sqrt{M},\quad \chi_j=v_j\sqrt{M} \ (j=1,2,3), \quad
\chi_4=\frac{(|v|^2-3)\sqrt{M}}{\sqrt{6}};\label{basis}\eq
and the null space of the operator $L_1$, denoted by $N_1$, is
spanned only by $\sqrt{M}$.

In the followings,
 $L^2(\R^3)$ denotes the Hilbert space of complex valued functions
on $\R^3$ with the inner product and the norm given by
$$
(f,g)=\intr f(v)\overline{g(v)}dv,\quad \|f\|=\(\intr |f(v)|^2dv\)^{1/2}.
$$
And denote by $P_0,P_d$ the projection operators from $L^2(\R^3_v)$ to the subspace $N_0, N_1$ with
\bma
 &P_0f=\sum_{j=0}^4(f,\chi_j)\chi_j,\quad P_1=I-P_0, \label{P10}
 \\
 &P_df=(f,\chi_0)\chi_0,   \quad P_r=I-P_d. \label{Pdr}
 \ema

From the Boltzmann's H-theorem, the linearized collision operators $L$ and $L_1$ are non-positive, precisely,  there are  constants $\mu_1, \,\,\mu_2>0$ such that \bma
 (Lf,f)&\leq -\mu_1 \| P_1f\|^2, \quad  \ f\in D(L),\\
 (L_1f,f)&\leq -\mu_2 \|P_rf\|^2, \quad  \ f\in D(L_1),\label{L_4}
 \ema
where $D(L)$ and $D(L_1)$ are the domains of $L$ and $L_1$ given by
$$ D(L)=D(L_1) =\left\{f\in L^2(\R^3)\,|\,\nu(v)f\in L^2(\R^3)\right\}.$$

In addition, for the hard sphere model, $\nu(v)$ satisfies
 \be
\nu_0(1+|v|)\leq\nu(v)\leq \nu_1(1+|v|).  \label{nuv}
 \ee
Without loss of generality, we assume $\mu_1=\mu_2=\mu$ and $\nu(0)\ge \nu_0\ge \mu>0$ in this paper.

For any vectors $U=(f,E,B), V=(g,X,Y)\in L^3(\R^3_v)\times \R^3\times \R^3$, define the  $L^2$ inner product and norm  by
$$ (U,V)=(f,g)+(E,X)+(B,Y),\quad \|U\|=\sqrt{(U,U)}. $$
Define the solution space $\mathcal{X}_1$ as
\be
\mathcal{X}_1=\{U=(f,E,B)\in L^2(\R^3_v)\times \R^3\times \R^3\,|\,\dx E_1=(f,\chi_0),\, B_1=0 \},
\ee
with the norm
\be  \|U\|^2_{\mathcal{X}_1} =\|P_rf\|^2+|\dx E_1|^2+|E|^2+|B|^2=\|U\|^2. \ee

From the system \eqref{VMB3a}--\eqref{VMB3d},
we have the following linearized  system for $(f_1,f_2, E, B)$:
\bma
&\dt f_1+v_1\dx f_1-Lf_1=0,\label{VMB5a}\\
&\dt f_2+v_1\dx f_2-L_1f_2-v\sqrt M\cdot E=0,\label{VMB6a}\\
&\dt E=\dx \O B-( f_2,v\chi_0),\quad \dx E_1=( f_2,\chi_0),\label{VMB7a}\\
& \dt B=-\dx \O E,\quad B_1\equiv0. \label{VMB8a}
\ema
\def\BB{\mathbb{B}}
\def\AA{\mathbb{A}}
%
%
%
For convenience of notations, we rewrite the linearized Boltzmann equation \eqref{VMB5a} for $f_1 $ as
\bq   \label{LVMB0}
 \left\{\bln
 &\dt f_1=\BB_0f_1, \quad t>0,\\
 &f_1(0,x,v)=f_{1,0},
\eln\right.
 \eq
where the operator $\BB_0$ defined on $L^2(\R^3_x\times\R^3_v)$ is given by
\begin{equation}
	\BB_0=L-v_1\dx,     \label{LVMB0z}
\end{equation}
and represent the linear  modified  Vlasov-Maxwell-Boltzmann system \eqref{VMB6a}--\eqref{VMB8a} for $U=(f_2,E,B)^T $ as
 \bq
 \left\{\bln            \label{LVMB1b}
 &\dt U=\AA_0U, \quad t>0,\\
 & \dx E_1=( f_2,\chi_0),\quad B_1=0,
 \\
 &U(0,x,v)=U_0 =(f_{2,0},E_0,B_0),
 \eln\right.
 \eq
where the operator  $\AA_0$ defined on $L^2(\R^3_x\times\R^3_v)\times L^2(\R^3_x)\times L^2(\R^3_x)$ is given by
 \begin{align}
&\AA_0 =\left( \ba
   L_1-v_1\dx
 &v\chi_0 &0\\
-P_m &0 &\dx \O\\
0 &-\dx \O &0
\ea\right),  \label{A_0}\\
& P_mf=(f,v\chi_0),\quad \forall f\in L^2(\R^3_v).
 \end{align}

Since we investigate the pointwise behavior with respect to the space-time variable $(t,x)$, it is convenient to regard the Green's functions $G_b(t,x)$ and $G(t,x)=(G^{ij}(t,x))_{1\le i,j\le 3}$ as the operators in $L^2(\R^3_v)$ and $\mathcal{X}_1$ defined by
\be   \label{LB}
 \left\{\bln
 &\dt G_b =\BB_0 G_b ,\quad t>0, \\
 &G_b(0,x)= \delta(x)I_v,
 \eln\right.
 \ee
 and
\be   \label{LVMB}
 \left\{\bln
 &\dt G =\AA_0 G ,\quad t>0, \\
 & \dx G^{2j}_1=( G^{1j},\chi_0),\quad G^{3j}_1=0,\\
 &G(0,x)=   G_0(x),
 \eln\right.
 \ee
where $I_v$ is the identity in $L^2(\R^3_v)$, $(G^{2j}_k$, $G^{3j}_k)$, $k=1,2,3$ are the $k^{th}$ line of $(G^{2j}$, $G^{3j})$, and
\be
 G_0(x)=\left( \ba
\delta(x)P_r &\dx\delta(x) \mathbb{I}_0 &0\\
0 & \delta(x)\mathbb{I}_1 & 0\\
0 & 0 &\delta(x)\mathbb{I}_2
\ea\right).
\ee
Here, $\mathbb{I}_0=(\chi_0,0,0)$  and $\mathbb{I}_1,\mathbb{I}_2$ are $3\times 3$ matrices given by
\be
\mathbb{I}_1=\left(\ba 1 & 0 & 0\\ 0 & 1 & 0 \\ 0 & 0 & 1 \ea\right),\quad \mathbb{I}_2=\left(\ba 0 & 0 & 0\\ 0 & 1 & 0 \\ 0 & 0 & 1 \ea\right).
 \ee
In addition, it holds for any $U=(g,E,B)\in \mathcal{X}_1$  that
 $$G_0(x)\ast U=(P_rg+\dx E_1\chi_0,E,B)=U,$$
namely, $G_0(x)=\delta(x)I_{\mathcal{X}_1}$.

Then, the solution $f_1$ to the Boltzmann equation \eqref{LVMB0} and $U=(f_2,E,B)$ to  the Vlasov-Maxwell-Boltzmann system \eqref{LVMB1b} can be represented by the Green's functions $G_b(t,x)$ and $G(t,x)$ respectively as
\bma
f_1(t,x)&=G_b(t)\ast f_{1,0}=\intra G_b(t,x-y)f_{1,0}(y)dy, \\
U(t,x)&=G(t)\ast U_0=\intra G(t,x-y)U_0(y)dy,
\ema
where $f_{1,0}(y)=f_{1,0}(y,v)$ and $U_0(y)=(f_{2,0}(y,v),E_0(y),B_0(y)).$

For any $(t,x)$, $f\in L^2(\R^3_v)$ and $U=(g,E,B)\in \mathcal{X}_1$, we define the $L^2$ norms of $G_b(t,x)$ and $G(t,x)$ by
\be
\|G_b(t,x)\|=\sup_{f\in L^2_v}\frac{\|G_b(t,x)f\|}{\|f\|},\quad
\|G(t,x)\|=\sup_{U\in \mathcal{X}_1}\frac{\|G(t,x)U\|}{\|U\|}.
\ee


\noindent\textbf{Notations:} \ \ Before state the main results in this paper, we list some notations.
The Fourier transform of $f=f(x,v)$
is denoted by
$$\hat{f}(\xi,v)=(\mathcal{F}f)(\xi,v)=\frac1{\sqrt{2\pi}}\int_{\R} f(x,v)e^{-  i x\xi}dx.$$

Set a weight function $w(v)$ by
$$
w(v)=(1+|v|^2)^{1/2},
$$
so that the weighted Sobolev space $ H^N_k$ $(H^N=H^N_0)$ is given by
$
 H^N_k=\{\,f\in L^2(\R_x\times \R^3_v)\,|\,\|f\|_{H^N_k}<\infty\,\}
$
equipped with the norm
$$
 \|f\|_{H^N_k}=\sum_{ |\alpha|+|\beta|\le N}\|w^k\dxa\dvb f\|_{L^2(\R_x\times \R^3_v)}.
$$
For $q\ge1$, denote
$$
L^{2,q}=L^2(\R^3_v,L^q(\R_x)),\quad
\|f\|_{L^{2,q}}=\bigg(\intr\bigg(\intra|f(x,v)|^q dx\bigg)^{2/q}dv\bigg)^{1/2}.
$$
For $\gamma\ge 0$, define
$$ \|f(x)\|_{L^\infty_{v,\gamma}}=\sup_{v\in \R^3}(1+|v|)^{\gamma}|f(x,v)|.
$$
For any vector $U=(f,E,B)\in L^2_v\times \mathbb{R}^3\times \mathbb{R}^3$,
define $Y^\infty_{\gamma}$, $\gamma\ge 0$ norm  of $U$ by
$$
\|U(x)\|_{Y^\infty_{\gamma}}=\|f(x)\|_{L^\infty_{v,\gamma}}+|E(x)|+|B(x)|,
$$
and  define $Z^q$, $q\ge 1$ norm of $U$ by
$$ \|U\|_{Z^q}=\|f\|_{L^{2,q}}+\|E\|_{L^q_x}+\|B\|_{L^q_x}.$$

In the following,  denote by $\|\cdot\|_{L^2_{x,v}}$  the norm of the function spaces $L^2(\R_x\times \R^3_v)$, and by $\|\cdot\|_{L^2_x}$ and $\|\cdot\|_{L^2_v}$  the norms of the function spaces $L^2(\R_x)$  and $L^2(\R^3_v)$ respectively.
For any integer $m\ge1$, denote by $\|\cdot\|_{H^m_x}$ and $\|\cdot\|_{L^2_v(H^m_x)}$ the norms in the spaces $H^m(\R_x)$ and $L^2(\R^3_v,H^m(\R_x))$ respectively.


In addition, denote the wave speeds $\sigma_j$, $-2\le j\le 2$ as
\be \sigma_{\pm 2}=\pm \mathbf{c}=\pm\sqrt{\frac53},\quad \sigma_{\pm 1}=\pm 1, \quad \sigma_0=0. \label{speed}\ee

With the above preparation, we are ready to state  the  main results in this paper.
Firstly, we have the pointwise space-time estimates on the Green's function for the linearized VMB system~\eqref{LVMB}.

\begin{thm}\label{green1}
Let $G(t,x)=(G^{ij}(t,x))_{3\times 3}$ be the Green's function for the VMB system defined by \eqref{LVMB}.
Then,  the Green's function $G(t,x)$  has the following decomposition:
$$ G(t,x)= W_{0}(t,x)+  G_{F,0}(t,x)+  G_{F,1}(t,x)+ G_{R}(t,x),$$
where  $W_{0}(t,x)$ is the singular  part, $ G_{F,0}(t,x)$ and $G_{F,1}(t,x)$ are the fluid parts corresponding to  low frequency and  high frequency respectively, and  $G_R(t,x)$ is the remaining part.
In particular, there exist  positive constants $C,D$ such that   $G_{F,0}(t,x)$ is smooth and contains only diffusion waves as
\be \label{in1}
\left\{\bln
&\|\dxa G^{1j}_{F,0}(t,x)\|+\|\dxa G^{2j}_{F,0}(t,x)\|
\le C\((1+t)^{-\frac{3+ \alpha}2}e^{-\frac{|x|^2}{D(1+t)}}  + e^{-\frac{|x|+t}{D}}\),
\\
&\|\dxa G^{3j}_{F,0}(t,x)\|+\|\dxa G^{j3}_{F,0}(t,x)\|
\le C\((1+t)^{-\frac{2+ \alpha}2}e^{-\frac{|x|^2}{D(1+t)}} + e^{-\frac{|x|+t}{D}}\),
\\
&\|\dxa G^{33}_{F,0}(t,x)\| \le C\((1+t)^{-\frac{1+ \alpha}2}e^{-\frac{|x|^2}{D(1+t)}} + e^{-\frac{|x|+t}{D}}\),
\eln\right.
\ee
for $\alpha\ge 0$ and $j=1,2$.
The high-frequency fluid part $G_{F,1}(t,x)$ contains hyperbolic waves as
\be
  G_{F,1}(t,x)=\dxa F_{\alpha}(t,x),\quad \alpha\ge 0, \label{in2a}
\ee
where $F_{\alpha}(t,x)$ is bounded and satisfies
\be \label{in2}
\left\{\bln
&\| F^{11}_{\alpha}(t,x)\|\le C\bigg(\sum_{l=\pm 1}(1+t)^{-\alpha-1}\ln^2(2+t)e^{-\frac{\nu_0|x-lt|}2} + e^{-\frac{|x|+t}{D}}\bigg),
\\
&\| F^{1j}_{\alpha}(t,x)\|+\| F^{j1}_{\alpha}(t,x)\|\le C\bigg(\sum_{l=\pm 1}(1+t)^{-\alpha-\frac12}\ln^2(2+t)e^{-\frac{\nu_0|x-lt|}2}  + e^{-\frac{|x|+t}{D}}\bigg),
\\
&\| F^{2j}_{\alpha}(t,x)\|+\| F^{3j}_{\alpha}(t,x)\|\le C\bigg(\sum_{l=\pm 1}(1+t)^{-\alpha}\ln^2(2+t)e^{-\frac{\nu_0|x-lt|}2} + e^{-\frac{|x|+t}{D}}\bigg),
\eln\right.
\ee
for $j=2,3.$ The remaining part $G_{R}(t,x)$ is bounded and satisfies
\be
\| G_{R}(t,x)\|\le Ce^{-\frac{|x|+t}{D}}. \label{in3}
\ee
And the singular part $W_{0}(t,x)$ can be decomposed into
$$
W_{0}(t,x)= W_{0,1}(t,x)+W_{0,2}(t,x) ,
$$ where $ W_{0,1}$ is the singular kinetic wave and  $W_{0,2}$ is the singular leading short wave defined by
\be
\hat{W}_{0,1}(t,\xi)=\sum^6_{n=0}\hat{\U}_{n,1}(t,\xi),\quad \hat{W}_{0,2}(t,\xi)=\hat{G}_2(t,\xi). \label{W_0}
\ee
Here $\hat{\U}_{n,1}$ and $\hat{G}_2$ are defined by \eqref{w1} and \eqref{G_b} respectively.
\end{thm}

 In \cite{Liu1,Liu2}, the authors firstly made systematic and pioneering contribution on the pointwise estimates of the Green's function to the linearized Boltzmann equation with  initial data $f_0$ being  compactly supported  in $L^\infty(\R^3_x\times \R^3_v)$. For the purpose of study
 in this paper, the following theorem basically puts the results  in \cite{Liu1,Liu2}
 on the pointwise estimates of Green's function to the linearized Boltzmann equation~\eqref{LB}
 in the setting when  the initial data $f_0=\delta(x)I_v$.

\begin{thm}\label{green2}
Let $G_b(t,x)$ be the Green's function for the Boltzmann equation defined by \eqref{LB}.
Then, the Green's function $G_b(t,x)$ can be decomposed into
$$ G_b(t,x)=W_{1}(t,x)+ G_{b,0}(t,x)+G_{b,1}(t,x),$$
where  $W_{1}(t,x)$ is the singular kinetic wave, $ G_{b,0}(t,x)$ is the fluid part at low frequency, and $ G_{b,1}(t,x)$ is the remaining part.
There exist positive constants $C,D$ such that   $G_{b,0}(t,x)$ is smooth and satisfies
\be \label{in1b}
\left\{\bln
&\|\dxa G_{b,0}(t,x)\|
\le C\bigg(\sum^1_{j=-1}(1+t)^{-\frac{1+\alpha}2}e^{-\frac{(x-j\mathbf{c}t)^2}{D(1+t)}} + e^{-\frac{|x|+t}{D}}\bigg), \\
&\|\dxa G_{b,0}(t,x)P_1\|
\le C\bigg(\sum^1_{j=-1}(1+t)^{-\frac{2+\alpha}2}e^{-\frac{(x-j\mathbf{c}t)^2}{D(1+t)}} + e^{-\frac{|x|+t}{D}}\bigg),
\\
&\|\dxa P_1G_{b,0}(t,x)P_1\| \le C\bigg(\sum^1_{j=-1}(1+t)^{-\frac{3+\alpha}2}e^{-\frac{(x-j\mathbf{c}t)^2}{D(1+t)}} + e^{-\frac{|x|+t}{D}}\bigg),
\eln\right.     \ee
and $G_{b,1}(t,x)$ is bounded and satisfies
\be
\| G_{b,1}(t,x)\|\le Ce^{-\frac{|x|+t}{D}}.
\ee
The singular wave $W_1(t,x)$ is given by
\be
W_1(t,x)= \sum^{6}_{k=0}J_k(t,x) , \label{W-1}
\ee where
$$
\left\{\bln
J_0(t,x)&=S^t\delta(x)I=e^{\nu(v)t}\delta(x-v_1t)I_v,\\
J_k(t,x)&=\intt S^{t-s}KJ_{k-1}ds,\quad k\ge 1.
\eln\right.
$$
Here $I_v$ is an identity operator in $L^2_v$ and the operator $S^t$ is defined by
\be S^th(x,v)=e^{-\nu(v)t}h(x-v_1t,v). \label{S-t}\ee
\end{thm}

The final result gives  the pointwise estimate of the global solution to the nonlinear VMB system~\eqref{VMB3a}--\eqref{VMB3e} as follows.

\begin{thm} \label{thm1}There exists a small constant $\delta_0>0$ such that if the initial data $U_0=(f_{1,0},f_{2,0},E_0,B_0)
$ satisfies $\|f_{1,0}\|_{H^8_3}+\|f_{2,0}\|_{H^8_3}+\|E_0\|_{H^8_x}+\|B_0\|_{H^8_x}\le \delta_0$ and
\be
\|\dxa f_{j,0}(x)\|_{L^\infty_{v,3}}+\|\Tdv f_{j,0}(x)\|_{L^\infty_{v,3}}+|\dxa  E_0(x)|+|\dxa  B_0(x)|\le C\delta_0(1+|x|^2)^{-\frac{\gamma}2} \label{initial}
\ee
for $j=1,2$, $\alpha\le 3$ and $\gamma>1$,
then the VMB system~\eqref{VMB3a}--\eqref{VMB3e} admits a unique globally solution $ U=(f_1,f_2,E,B) $ satisfying
\bma
\| f_1(t,x)\|_{L^\infty_{v,3}}+\|\Tdv  f_1(t,x)\|_{L^\infty_{v,3}} +| B(t,x)|&\le C\delta_0 \sum_{-2\le j\le 2}(1+t)^{-\frac12 }B_{\frac12}(t,x-\sigma_jt),  \label{t1}
\\
\|\dx f_1(t,x)\|_{L^\infty_{v,3}}  +|\dx  B(t,x)|&\le C\delta_0 \sum_{-2\le j\le 2}(1+t)^{-1 }B_{\frac12}(t,x-\sigma_jt),  \label{t3}
\\
\| f_2(t,x)\|_{L^\infty_{v,3}}+\|\Tdv  f_2(t,x)\|_{L^\infty_{v,3}}+| E(t,x)|&\le C\delta_0 \sum_{-2\le j\le 2}(1+t)^{-1 }B_{\frac12}(t,x-\sigma_jt),  \label{t2}
\ema
where $\sigma_j,$ $-2\le j\le 2$ is wave speed defined by \eqref{speed}, and the parabolic profile $B_{\frac{\beta}2}(t,x-\lambda t)$ is defined by
\be
B_{\frac{\beta}2}(t,x-\lambda t)= \(1+\frac{|x-\lambda t|^2}{1+t}\)^{-\frac{\beta}2},\quad \beta>0,\,\,\lambda\in \R.
\ee
\end{thm}

We now outline the main ideas in  the proof of the above theorems. The results in Theorem~\ref{green1} about the pointwise behavior of the Green's function $G(t,x)$ to the VMB system \eqref{LVMB} are obtained
 based on the spectral analysis \cite{Li1} and the ideas inspired by \cite{Li4,Liu1,Liu2}. Yet new observations and ideas are made to deal with the influence of the lorentz force. First, we estimate the Green's function  $G(t,x)$ inside the Mach region $|x|\le 2t$  based on the spectral analysis. Indeed, we  decompose the Green's function  $G$  into the lower frequency part  $G_{L}$, the middle frequency part  $G_{M}$ and the high frequency part  $G_{H}$,
and further split  $G_L$ and $G_H$ into the fluid parts $G_{L,0}$, $G_{H,0}$, and the non-fluid parts $G_{L,1}$, $G_{H,1}$ respectively, namely
$$
\left\{\bln
&G=G_L+G_M+G_H,\\
&G_L=G_{L,0}+G_{L,1},\,\,\, G_H=G_{H,0}+G_{H,1}.
\eln\right.
$$
By using Fourier analysis techniques, we can show that the low-frequency fluid part $G_{L,0}(t,x)$ is smooth and contains only diffusion waves for $|x|\le 2t$ since the Fourier transform of the linear VMB operator $\AA_0(\xi)$ has two eigenvalues $\{\lambda_j(\xi),\,j=1,2\}$ at the low frequency region $|\xi|\le r_0$ satisfying
$$
\lambda_1(\xi)=\lambda_2(\xi)=-a_1\xi^2+O(\xi^4),\quad a_1>0.
$$
However, it is difficult to estimate the high-frequency fluid part $G_{H,0}(t,x)$ since $G_{H,0} $ contains both the singular waves and hyperbolic waves, 
where the singular waves are constructed as the non-integral part of $\hat{G}_{H,0}(t,\xi)$ in $L^1(\R^3_\xi)$,
and the structures of these waves depend on the eigenvalues and eigenfunctions of $\AA_0(\xi)$  at high frequency.
To overcome this difficulty, we carefully analyze the spectrum of  $\AA_0(\xi)$ to find that due to the influence of the electro-magnetic field,
 $\AA_0(\xi)$ admits four eigenvalues $\{\beta_j(\xi),\,j=1,2,3,4\}$ at the high frequency region $|\xi|\ge r_1$ satisfying
\be\label{h-eigen}
\left\{\bln
&\beta_1(\xi)=\beta_2(\xi)=-i\xi+\gamma_1(\xi)+O(|\xi|^{-2}\ln^2|\xi|),\\
&\beta_3(\xi)=\beta_4(\xi)=i\xi+\gamma_3(\xi)+O(|\xi|^{-2}\ln^2|\xi|),\\
 &C_1(1+|\xi|)^{-1}\le  -{\rm Re}\gamma_j(\xi)\le C_2(1+|\xi|)^{-1}.
\eln\right.
\ee
Therefore,  we can extract the singular part  $\hat{G}_2$ as the zero order expansion of $\hat{G}_{H,0}$ with respect to $\xi^{-1}$, and show that for $|x|\le 2t$, the remaining part $G_{H,0}-G_2$ is bounded and contains the hyperbolic waves with the speeds $\sigma_{\pm1}=\pm1$. Furthermore, since the eigenvalues $\beta_j(\xi)$ have no spectral gaps and satisfy \eqref{h-eigen} so that for $\alpha\ge 0$ and $|\xi|\ge r_1$,
$$
|\xi|^{-\alpha}e^{{\rm Re}\beta_j(\xi)t}\le C(1+t)^{-\alpha},\quad j =1,2,3,4,
$$
we can prove that $G_{H,0}-G_2$ gains more time-decay rate by taking the negative $x-$derivative.  We would like to mention that these behaviors of $G_{H,0}$  are completely new and can not be observed from the Boltzmann equation and VPB system which admit the spectral gaps at high frequency.

Then, we apply  the Picard's iteration  to  estimate the Green's function $G (t,x)$ outside the Mach region $|x|\ge 2t$. Instead of the linear VMB system \eqref{LVMB}, we construct the Picard's iteration on its Fourier transform \eqref{l2}.  It is interesting to observe that based on the curl-div decomposition of electro-magnetic fields and the  orthogonal structure \eqref{orth},
 the VMB system \eqref{l2} for $(\hat{f},\hat{E},\hat{B})$ can be decomposed into two systems \eqref{l3} and \eqref{l4}.  To be more precisely, \eqref{l3} is a self-consistent VMB type system for $(\hat{g}_1,\hat{E}_r,\hat{B}_r)$ which describes the transport of particles affected by 
 the curl part of the electro-magnetic fields. Meanwhile,  \eqref{l4} is a VPB type system  for $(\hat{g}_2,\hat{E}_1 )$ which describes the transport of particles in terms of 
 the divergence part of electric field. 
 Therefore,  we can construct an approximate sequence $ \hat{\U}_n(t,\xi)  $ of $ \hat{G}(t,\xi) $ such that $\hat{\U}_n(t,\xi) $ consists of the curl part $\hat{\mathcal{H}}^c_n(t,\xi)$  and the divergence part $ \hat{\mathcal{H}}^d_n(t,\xi)$ corresponding to \eqref{l3} and \eqref{l4} respectively:
\be
\hat{\U}_n(t,\xi) =  \hat{\mathcal{H}}^c_n(t,\xi) +  \hat{\mathcal{H}}^d_n(t,\xi). \label{div-curl}
\ee
In particular,
 $\hat{\mathcal{H}}^d_n$ can be represented by the combination of the mixture operator of Boltzmann equation $\hat{\M}^t_n(\xi)$ similarly to the VPB system \cite{Li4}, and  $\hat{\mathcal{H}}^c_n$
can be represented by the mixture operator of VMB system $\hat{\Q}^t_n(\xi)$.
From \cite{Li4,Liu1,Liu2}, $\hat{\M}^t_n(\xi)$ satisfies the following mixture lemma  (refer to Lemma  \ref{mix1a})
\be
 \|\hat{\M}^t_{3n}(\xi)\| \le C_n(1+t)^{3n}e^{-\nu_0t}(1+|\xi|)^{-n}.  \label{Q-4a}
\ee
And in terms of the mixed effects of the transport of particle and the rotation of the electro-magnetic fields, $\hat{\Q}^t_n(\xi)$ satisfies the following new mixture lemma (refer to Lemma \ref{mix1b}) 
\be
\hat{\Q}^t_n(\xi)=\hat{\Q}^t_{n,1}(\xi)+\hat{\Q}^t_{n,2}(\xi), \label{Q-1a}
\ee
where $\hat{\Q}^t_{n,1}(\xi)$ and $\hat{\Q}^t_{n,2}(\xi)$ represent the kinetic part and the fluid part respectively,  and they satisfy
\bma
 \|\hat{\Q}^t_{3n,1}(\xi)\|&\le C_n(1+t)^{3n}e^{-\nu_0t}(1+|\xi|)^{-n}\ln^n(2+|\xi|), \label{Q-2a}\\
  \|\hat{\Q}^t_{3n,2}(\xi)\|&\le C_n(1+t)^{2n-1}e^{|{\rm Im}\xi| t} (1+|\xi|)^{-n}\ln^n(2+|\xi|). \label{Q-3}
\ema
By \eqref{div-curl} and  the mixture lemmas \eqref{Q-4a}--\eqref{Q-3} for $\M^t_n(\xi)$ and $\Q^t_n(\xi)$,  the approximate solutions $\hat{\U}_n(t,\xi)$ can also be decomposed into (refer to Lemma \ref{mix3})
$$
\hat{\U}_n (t,\xi)=\hat{\mathcal{H}}^c_n(t,\xi)+\hat{\mathcal{H}}^d_n(t,\xi)=\hat{\U}_{n,1}(t,\xi)+\hat{\U}_{n,2}(t,\xi),
$$
where $\hat{\U}_{n,1}(t,\xi)$ denotes the kinetic part satisfying \eqref{Q-2a}, and $\hat{\U}_{n,2}(t,\xi)$ is the fluid part satisfying \eqref{Q-3}. These implies that both $\hat{\U}_{n,1}$ and $\hat{\U}_{n,2}$ can gain spatial regularities.
Then, we can define the $n^{th}$ degree kinetic wave $Y_n$ as follows:
$$
Y_n =Y_{n,1}+Y_{n,2} \quad{\rm with} \quad Y_{n,1} =\sum_{k=0}^{3n}\U_{k,1} , \quad Y_{n,2} =\sum_{k=0}^{3n}\U_{k,2},
$$
where
$\hat{Y}_{n,1}$  consists of the singular kinetic waves $\hat{W}_{0,1}$ and approaches $\hat{G}_{H,1}$  exponentially in time (refer to  Lemma \ref{l-high2})
 $$
 \|Y_{n,1}(t,x)-G_{H,1}(t,x)\|\le Ce^{-\eta_0t},\quad n\ge 2,
 $$
and $\hat{Y}_{n,2}$ approximates the singular fluid waves $\hat{G}_2$ exponentially in $(t,x)$ for $|x|\ge 2t$ (refer to Lemma \ref{W_1})
$$
\|Y_{n,2}(t,x)-G_{2}(t,x)\|\le Ce^{-\frac{\nu_0(|x|+t)}8}, \quad n\ge 2.
$$
By the mixture lemma (refer to Lemma \ref{mix3}), $\U_2(t,x)$ is bounded with respect to $(t,x)$, which implies that the remaining term $Z_2=G-Y_2$ is also bounded. Thus, we can have the weighted energy estimates on the remaining term $Z_2(t,x)$  to show (refer to Lemma \ref{W_2})
$$
\| Z_{2}(t,x)\|\le Ce^{- \frac{|x|+t}{D}},\quad |x|\ge 2t.
$$
Combining the above decompositions and estimates, we can obtain the expected decomposition and space-time pointwise behavior for each part of the Green's function $G(t,x)$ as listed in Theorem \ref{green1}.

Moreover, we rewrite  the pointwise estimate on  the Green's function $G_b$ of the Boltzmann equation \eqref{LB} as shown in Theorem \ref{green2}. The main idea comes from the previous works \cite{Liu1,Liu2}.  Finally, by using of the  estimates of the Green's functions in Theorems \ref{green1}--\ref{green2} and the estimates of waves coupling  in Lemmas \ref{couple1a}--\ref{wc-1}, one can establish the pointwise space-time
estimate on the global solution to the nonlinear VMB system given in Theorem~\ref{thm1}.

The rest of this paper will be organized as follows. In  Section~\ref{spectrum-two}, we first improve  the spectrum analysis  for the  two-species VMB system in order to estimate the pointwise behavior of the fluid parts of the Green's function.
Based on the spectrum analysis 
 the  pointwise space-time estimate  on the Green's function to the  linearized VMB system \eqref{LVMB} will be
given in Section~\ref{fluid} and Sections~\ref{kinetic}. In Section~\ref{sect3}, we improve the  pointwise space-time estimate  of the Green's function to the  linearized Boltzmann equation \eqref{LB} based on the previous work \cite{Liu1,Liu2}.
In Section~\ref{behavior-nonlinear}, we prove the pointwise space-time estimates of the global solution to the  nonlinear VMB system.

\section{Spectral analysis}
\label{spectrum-two}
\setcounter{equation}{0}
In this section, we first recall the spectral structure of
 the linearized VMB  system~\eqref{LVMB1b} and then analyze the analyticity of the eigenvalues and eigenvectors in order to study the pointwise estimate of the fluid parts of the Green's functions $G(t,x)$.

First, we take the Fourier transform in \eqref{LVMB} with respect to $x$ to have
 \bq \label{LVMB-3}
 \left\{\bln
 &\dt \hat{G}=\AA_0 (\xi)\hat{G}, \quad t>0,\\
 & i\xi \hat{G}^{2j}_1=( \hat{G}^{1j},\chi_0),\quad \hat{G}^{3j}_1=0,\\
 & \hat{G}_0(0,\xi)=\hat{G}_0(\xi),
 \eln\right.
 \eq
where 
\bma
\AA_0 (\xi)&=\left( \ba
L_1-iv_1\xi &v\chi_0  &0\\
-  P_{m }&0 & i \xi \O\\
0 &- i \xi \O &0
\ea\right),\\
 \hat{G}_0(\xi)&=\left( \ba
 P_r &i\xi \mathbb{I}_0 &0\\
0 &  \mathbb{I}_1 & 0\\
0 & 0 & \mathbb{I}_2
\ea\right). \label{initial-1}
\ema
Define a subspace $\hat{\mathcal X}_1$ of $L^2(\R^3_v)\times \C^3\times \C^3$ as
\be \hat{\mathcal X}_1=\{U=(f,E,B)\in L^2(\R^3_v)\times \C^3\times \C^3\,|\, i\xi E_1=(f,\chi_0),\, B_1=0\}. \label{X11}\ee
Then $\hat{G}(t,\xi)$ and $\hat{G}_0(\xi) $ are operators in $\hat{\mathcal X}_1$, and $\hat{G}_0(\xi)U_0=U_0$ for any $U_0\in \hat{\mathcal X}_1$. 

\subsection{Eigenvalues in low frequency}
Take the Fourier transform in \eqref{LVMB1b} with respect to $x$ to have
 \bq  \label{LVMB1a}
 \left\{\bln
 &\dt \hat{U}=\AA_0(\xi) \hat{U}, \quad t>0,\\
  &i\xi \hat{E}_1=(\hat{f}_2,\sqrt M),\quad  \hat{B}_1\equiv0,\\
& \hat{U}_0(0,\xi)=\hat{U}_0(\xi).
 \eln\right.
 \eq
When we consider the spectrum structure of the VMB system \eqref{LVMB1a},
it is convenient to consider the equivalent linearized system for  $\hat{V}=(\hat{f}_2,\hat{E}_r,\hat{B}_r)^T $ with $\hat{E}_r=(\hat{E}_2,\hat{E}_3),\, \hat{B}_r=(\hat{B}_2,\hat{B}_3)$ given by
 \bq     \label{LVMB1}
 \left\{\bln
 &\dt \hat{V}=\AA_1 (\xi)\hat{V}, \quad t>0,\\
 & \hat{V}_0(0,\xi)=\hat{V}_0(\xi),
 \eln\right.
 \eq
 where $\AA_1 (\xi)$ is the operator on  $L^2(\R^3_v)\times \C^2\times \C^2$ defined by
\be
\AA_1 (\xi)=\left( \ba
\BB_1(\xi) &\bar{v}\chi_0  &0\\
-  P_{\bar{m} }&0 & i \xi \O_1\\
0 &- i \xi \O_1 &0
\ea\right).
\ee
Here,
 \bma
\BB_1(\xi) &=L_1- i v_1 \xi -\frac{ i v_1}{\xi}P_d,\quad \xi\ne 0,  \label{B(xi)1}\\
P_{\bar{m}}f&=(f,\bar{v}\chi_0),\quad \bar{v}=(v_2,v_3),\\
\O_1&=\left(\ba 0 & -1 \\  1 & 0 \ea\right). \label{O_1}
 \ema

Denote the weighted Hilbert space $L^2_{\xi}(\R^3_v)$ for $\xi\ne 0$ as
$$
 L^2_{\xi}(\R^3)=\{f\in L^2(\R^3_v)\,|\,\|f\|_{\xi}=\sqrt{(f,f)_{\xi}}<\infty\},
$$
equipped with the inner product
$$
 (f,g)_{\xi}=(f,g)+\frac1{\xi^2}(P_{d} f,P_{d} g).
$$

For any vectors $U=(f,E,B),V=(g,X,Y)\in L^2_{\xi}(\R^3_v)\times \mathbb{C}^3\times \mathbb{C}^3$,  define a weighted inner product and the corresponding
norm by
$$ (U,V)_{\xi}=(f,g)_{\xi}+(E,X)+(B,Y),\quad \|U\|_{\xi}=\sqrt{(U,U)_{\xi}}. $$

First of all, note that
 $P_d$ is a self-adjoint operator satisfying
$(P_d f,P_d g)=(P_d f, g)=( f,P_d g)$. Hence,
 \bq
  (f,g)_{\xi}=(f,g+\frac1{ \xi^2}P_dg)=(f+\frac1{ \xi^2}P_df,g).\label{C-1}
 \eq
By
\eqref{C-1}, we have for any $f,g\in L^2_{\xi}(\R^3_v)\cap D(\BB_1(\xi))$,
\be (\BB_1(\xi)
f,g)_{\xi}=(\BB_1(\xi) f,g+\frac1{ \xi^2}P_d g)=(f,\BB_1(-\xi)g)_{\xi}.\label{L_7}
\ee
Thus for any $U,V\in L^2_{\xi}(\R^3_v)\cap D(\BB_1(\xi))\times \C^2\times \C^2$,
\be (\mathbb{A}_1(\xi)U,V)_{\xi}=(U,\mathbb{A}^*_1(\xi)V)_{\xi}, \label{L_7a}\ee
where
\be
\AA^*_1 (\xi)=\left( \ba
\BB_1(-\xi) &-\bar{v}\chi_0  &0\\
  P_{\bar{m}} &0 & -i \xi \O_1\\
0 & i \xi \O_1 &0
\ea\right).
\ee
Since $L^2_{\xi}(\R^3_v)\times \mathbb{C}^2\times \mathbb{C}^2$ is an invariant subspace of the operator $\AA_1 (\xi)$,  $\AA_1 (\xi)$ can
be regarded as a linear operator on $L^2_{\xi}(\R^3_v)\times \mathbb{C}^2\times \mathbb{C}^2$.
Firstly, we recall the following lemma from \cite{Li1}.

 \begin{lem}[\cite{Li1}]\label{spectrum1}
(1) There exist two constants $b_0,r_0>0$ such that the spectrum $ \sigma(\AA_1 (\xi))\cap \{\lambda\in\mathbb{C}\,|\,\mathrm{Re}\lambda>-b_0\}$  consists of two points $\{\lambda_j(\xi),\ j=1,2\}$ for  $|\xi|\leq r_0$.  In particular, the eigenvalues $\lambda_j(\xi)$ are $C^\infty$ functions of $\xi$ and satisfy the following asymptotic expansion for $|\xi|\leq r_0$:
 $$
 \lambda_{1}(\xi) =\lambda_{2}(\xi) = -a_1\xi^2+O(\xi^3),
 $$
where $a_1>0$ is a constant.

(2) There exists a constant $r_1>0$ such that the spectrum $ \sigma(\AA_1 (\xi))\cap \{\lambda\in\mathbb{C}\,|\,\mathrm{Re}\lambda>-\mu/2\}$ consists of four eigenvalues $\{\beta_j(\xi),\ j=1,2,3,4\}$ for  $|\xi|> r_1$.  In particular, the eigenvalues $\beta_j(\xi)$ satisfy
 \bgrs
 \beta_1(\xi) = \beta_2(\xi) =- i \xi+O(|\xi|^{-\frac12}),
 \\
 \beta_3(\xi) = \beta_4(\xi) = i \xi+O(|\xi|^{-\frac12}),
 \\
\frac{C_1}{|\xi|}\le -{\rm Re}\beta_{j}(\xi)\le \frac{C_2}{|\xi|},
\egrs
for two positive constants $C_1$ and $C_2$.

(3) For any $r_1>r_0>0$, there
exists $\alpha =\alpha(r_0,r_1)>0$ such that for  $r_0\le |\xi|\le r_1$,
$$ \sigma(\AA_1 (\xi))\subset\{\lambda\in\mathbb{C}\,|\, \mathrm{Re}\lambda \leq-\alpha\} .$$

(4) The semigroup $\tilde{S}(t,\xi)=e^{t\AA_1(\xi)}$ with  $|\xi|\neq0$ satisfies
 $$
 \tilde{S}(t,\xi)U=\tilde{S}_1(t,\xi)U+\tilde{S}_2(t,\xi)U+\tilde{S}_3(t,\xi)U,
     \quad U\in L^2_{\xi}(\R^3_v)\times \mathbb{C}^2\times \mathbb{C}^2,
 $$
 where
 \bmas
 \tilde{S}_1(t,\xi)U&=\sum^2_{j=1}e^{t\lambda_j(\xi)} \(U, \psi^*_j(\xi)\) \psi_j(\xi) 1_{\{|\xi|\leq r_0\}},
               \\
 \tilde{S}_2(t,\xi)U&=\sum^4_{j=1} e^{t\beta_j(\xi)}  \(U,\phi^*_j(\xi)\) \phi_j(\xi)  1_{\{|\xi|\ge r_1\}}.
 \emas
Here, $(\lambda_j(\xi), \psi_j(\xi))$, $(\beta_j(\xi),\phi_j(\xi))$ are the eigenvalues and eigenvectors of the operator $\AA_1(\xi)$ for $|\xi|\le r_0$ and $|\xi|>r_1$ respectively, and $\tilde{S}_3(t,\xi) =: \tilde{S}(t,\xi)-\tilde{S}_1(t,\xi)-\tilde{S}_2(t,\xi)$ satisfies
that there exists  a constant $\kappa_0>0$ independent of $\xi$ such that
 $$
 \|\tilde{S}_3(t,\xi)U\|_{\xi}\leq Ce^{-\kappa_0t}\|U\|_{\xi},\quad t>0.
 $$
\end{lem}

We now turn to  study the analyticity and expansion of the eigenvalues and eigenvectors of $\AA_1(\xi)$ for low frequency and high frequency, respectively. The eigenvalue problem $\AA_1 (\xi)U=\lambda  U$ for $U=(f,X,Y)\in L^{2}_\xi(\R^3_v)\times \mathbb{C}^2\times \mathbb{C}^2$ can be written as
 \bma
  \lambda f &=\BB_1(\xi)f+\bar{v}\chi_0\cdot X,\label{L_2}\\
  \lambda X&=- (f,\bar{v}\chi_0)+ i \xi\O_1 Y,\label{L_2a}\\
  \lambda Y&=- i \xi\O_1 X,\quad  \xi\ne0.\label{L_2b}
 \ema

First, we show the analyticity and expansion of the eigenvalues $\lambda_j(\xi)$ and the corresponding eigenvectors $\psi_j(\xi)$ of $\AA_1 (\xi)$ at low frequency as follows.

\begin{thm}\label{eigen_3}
(1)
There exist two constants $b_0,r_0>0$ such that the spectrum $\sigma(\AA_1(\xi))\cap \{\lambda\in\mathbb{C}\,|\,\mathrm{Re}\lambda>-b_0\}$
consists of two points $\{\lambda_j(\xi),j=1,2\}$ for   $|\xi|\leq r_0$. The eigenvalues $\lambda_j(\xi)$  are  even, analytic functions of $\xi$ and satisfy the following expansion for $|\xi|\le r_0$
 \be                                   \label{specr0}
 \lambda_1(\xi)=\lambda_2(\xi)=-a_{1}\xi^{2}+O(\xi^4),
 \ee
where $a_1>0$ is a constant defined by
\bq
a_1=- \frac{1}{(L_1^{-1}\chi_2,\chi_2)} >0. \label{a_1}
\eq

(2) The eigenvectors $\psi_j(\xi)=(h_j(\xi),\y^2_j(\xi),\y^3_j(\xi))\in L^{2}_\xi(\R^3_v)\times \mathbb{C}^2\times \mathbb{C}^2$ are analytic functions of $\xi$ and satisfy
 \bq
  \left\{\bln                      \label{eigf1}
   &(\psi_i(\xi),\psi^*_j(\xi))=(h_i,\overline{h_j})-(\y^2_i,\overline{\y^2_j})-(\y^3_i,\overline{\y^3_j})=\delta_{ij}, \quad  i,j=1,2,\\
&h_j(\xi)=b_1(\xi)\(\lambda_1(\xi)-L_1+ i \xi P_rv_1P_r\)^{-1}(\bar{v}\cdot \e_j)\chi_0,\\
&\y^2_j(\xi)=b_1(\xi)\e_j,\quad \y^3_j(\xi)= \frac{ i \xi b_1(\xi)}{\lambda_1(\xi)}\O_1 \e_j,
  \eln\right.
  \eq
where $ \psi^*_j=(\overline{h_j},-\overline{\y^2_j},-\overline{\y^3_j})$, $\e_1=(1,0)$, $\e_2=(0,1)$ and $b_1(\xi)$ is an analytic and  odd function of $\xi$ for $|\xi|\le r_0$ satisfying
$$b_1(\xi)= a_1\xi+O(\xi^3).$$
\end{thm}

\begin{proof}
Let $U=(f,X,Y)$ be the eigenvector of \eqref{L_2}--\eqref{L_2b}. We rewrite $f$ in the
form $f=f_0+f_1$, where $f_0=P_df=C_0\sqrt M$ and $f_1=(I-P_d)f=P_{r}f$.
Then  \eqref{L_2} gives
 \bma
 &\lambda f_0=- P_d[ i \xi v_1(f_0+f_1)],\label{A_2}
\\
&\lambda f_1=L_1f- P_{r}[ i \xi v_1(f_0+f_1)]-\frac{ i v_1}{\xi}f_0+\bar{v}\chi_0\cdot X.\label{A_3}
 \ema
By   \eqref{A_3}, the microscopic part $f_1$ can be represented  by
 \bq
 f_1=[L_1-\lambda - i\xi P_{r}v_1 P_{r}]^{-1} P_{r}\bigg( i v_1\bigg(\xi+\frac{1}{\xi}\bigg)f_0-\bar{v}\chi_0\cdot X\bigg),
 \quad   \text{Re}\lambda>-\mu. \label{A_4}
 \eq
Substituting \eqref{A_4} into \eqref{A_2} and \eqref{L_2a}, we obtain the eigenvalue problem  for  $(\lambda,C_0,X,Y)$ as
 \bma
 \lambda C_0=&(1+\xi^2)(R(\lambda,\xi)\chi_1,\chi_1)C_0, \label{A_6}\\
  \lambda X=& (R(\lambda,\xi)\chi_2,\chi_2)X+ i \xi\O_1 Y, \label{A_7}\\
  \lambda Y=&- i \xi\O_1 X, \label{A_8}
 \ema
where
$
 R(\lambda,\xi)=(L_1-\lambda - i \xi P_{r} v_1 P_{r})^{-1}.
$

Multiplying \eqref{A_7} by $\lambda$ and using \eqref{A_8}, we obtain
\bq (\lambda^2-(R(\lambda,\xi)\chi_2,\chi_2)\lambda+\xi^2)X=0.\eq

For ${\rm Re}\lambda>-\mu,$ denote
 \bma
 D_0(\lambda,\xi)&=:\lambda-(1+\xi^2)(R(\lambda,\xi)\chi_1,\chi_1),\label{D0}\\
 D_1(\lambda,\xi)&=:\lambda^2-(R(\lambda,\xi)\chi_2,\chi_2)\lambda+\xi^2.   \label{D1}
 \ema

From Lemmas 3.10--3.11 in \cite{Li1}, there exists constants $b_0,r_0,r_1>0$  such that  $D_0(z,\xi)=0$ has no solution for ${\rm Re}z\ge -b_0$ and $|\xi|\le r_0$,  and $D_1(z,\xi)=0$ with ${\rm Re}z\ge -b_0$ has a unique smooth solution $z(\xi)$  for $(\xi,z)\in[-r_0, r_0]\times B(0,r_1)$. Moreover, $z(\xi)$ is an even function of $\xi$  satisfying
$$z(0)=0,\quad z'(0)=0,\quad z''(0)=\frac2{(L_1^{-1}\chi_2,\chi_2)}=:2a_1.$$
Since $R(\lambda,\xi)$ is analytic in $(\lambda,\xi)$, it follows that $D_0(\lambda,\xi)$ and $D_1(\lambda,\xi)$ are analytic in $(\lambda,\xi)$. Thus, by the implicit function theorem (Cf. Section 8 of Chapter 0 in \cite{homo}), the solution $z(\xi)$ of  $D_1(z,\xi)=0$ is an analytic function for $|\xi|\le r_0$.
Thus, we obtain \eqref{specr0}.

The eigenvalues $\lambda_j(\xi)$ and the eigenvectors $\psi_j(\xi)=(h_j,\y^2_j,\y^3_j)(\xi)$ for $j=1,2$ can be constructed as follows.  We take $\lambda_1=\lambda_2=z(\xi)$ to be the solution of the equation  $D_1(z,\xi)=0$,  and choose $C_0=0$ in \eqref{A_6}. The corresponding eigenvectors $\psi_j(\xi)=(h_j,\y^2_j,\y^3_j)(\xi)$, $j=1,2$ are given by
 \bq
 \left\{\bln       \label{C_3d}
  &h_j(\xi)
 =-b_1(\xi)(L_1-\lambda_j - i \xi P_rv_1 P_r)^{-1}(\bar{v}\cdot \e_j)\chi_0,  \\
  &\y^2_j(\xi)=b_1(\xi)\e_j,\quad \y^3_j(\xi)= \frac{ i \xi b_1(\xi)}{\lambda_1(\xi)}\O_1 \e_j,
  \eln\right.
\eq where $\e_1=(1,0)$ and $\e_2=(0,1)$.
It is easy to verify that   $(\psi_1(\xi),  \psi_2^*(\xi))=0,$ where $ \psi^*_j=(\overline{h_j},-\overline{\y^2_j},-\overline{\y^3_j})$ is the eigenvector of $\AA^*_1(\xi)$ corresponding to the eigenvalue $\overline{\lambda_{j}(\xi)}$.
We can normalize them by taking
$$(\psi_j(\xi),\psi^*_j(\xi))=1, \quad j=1,2.$$
The coefficient $b_1(\xi)$   is determined by normalization   as
 \bq    \label{C_2d}
 b_1(\xi)^2\(D_1(\xi)-1+\frac{\xi^2}{\lambda_1(\xi)^2}\)=1,
 \eq
 where $D_1(\xi)=(R(\lambda_1(\xi),\xi)\chi_1, R(\overline{\lambda_1(\xi)},-\xi)  \chi_1).$ Substituting \eqref{specr0} into \eqref{C_2d}, we obtain $b_1(-\xi)=-b_1(\xi)$ and
 \bmas
b_1(\xi)^2&=\(D_1(\xi)-1+\frac{1}{a^2_1\xi^2+O(\xi^4)}\)^{-1}\\
&=a_1^2\xi^2\(1+O(\xi^2)+a_1^2\xi^2(D_1(\xi)-1)\)^{-1}\\
&=a_1^2\xi^2+O(\xi^4).
\emas This and \eqref{C_3d} give the expansion of $\psi_j(\xi)$ for $j=1,2$ given in \eqref{eigf1}. And then this completes
the proof of the theorem.
\end{proof}

\subsection{Eigenvalues in high frequency}
We  study the analyticity and the asymptotic expansions of the eigenvalues
and eigenvectors of $\AA_1(\xi)$ in the high frequency region in this subsection.
\begin{lem}\label{lem1}
Let $\lambda, \xi\in \C$ and
$ c(\xi)=-\nu(v)-iv_1\xi. $
If $\delta=\nu_0+{\rm Re}\lambda >0$ and $|{\rm Im} \xi|<\nu_0$, then we have
 \bma
 \|K_1(\lambda-c(\xi))^{-1}\|  &\leq C\delta^{-1/2}(1+|\xi|)^{-1/2},\label{T_1}
 \\
\|(\lambda-c(\xi))^{-1}K_1\|  &\leq C\delta^{-1/2}(1+|\xi|)^{-1/2},\label{T_2}
 \\
\|(\lambda-c(\xi))^{-1}\chi_j\|  &\leq C\delta^{-1/2}(1+|\xi|)^{-1/2}, \label{T_4}
 \ema
where $j=0,1,2,3,4.$ Furthermore, there exists  $R_0>0$ sufficiently large such that $\lambda-\BB_1(\xi)$ is invertible for $\delta=\nu_0+{\rm Re}\lambda >0$, $|{\rm Im}\xi|<\nu_0$ and $|{\rm Re} \xi|>R_0$, and satisfies
\bq \|(\lambda-\BB_1(\xi))^{-1}\chi_j\|   \leq C\delta^{-1/2}(1+|\xi|)^{-1/2}. \label{T_5}\eq
\end{lem}

\begin{proof}
Let $\lambda=x+iy$ and $\xi=\zeta+i \eta$ with $(x,y),(\zeta,\eta)\in \R\times \R$.
Then
\bma
&\|K_1(\lambda-c(\xi))^{-1}f\|^2\nnm\\
=&\intr\left|\intr k_1(v,u)(\nu(u)+\lambda+iu_1\xi)^{-1}f(u)du\right|^2dv\nnm\\
\le& \intr\(\intr k_1(v,u)|\nu(u)+\lambda+iu_1\xi|^{-2}du\)\(\intr k_1(v,u)|f(u)|^2 du\)dv\nnm\\
\le& C\sup_{v\in\R^3}\intr k_1(v,u)\frac1{(\nu(u)+x-u_1\eta)^2+(y+ u_1\zeta)^2}du\|f\|^2. \label{K2}
\ema

Since
\be |k_1(v,u)|\le C\frac{1}{|\bar{v}-\bar{u}|}e^{-\frac{|v-u|^2}8} ,\quad \bar{u}=(u_2,u_3),\label{K1}\ee
it follows from \eqref{K2} that for $\delta=\nu_0+x >0$, $|\zeta|>1$  and $|\eta|\le \nu_0$,
\bma
&\|K_1(\lambda-c(\xi))^{-1}f\|^2\nnm\\
\le& C\intra \frac1{(\nu_0+x)^2+(y+ u_1\zeta)^2}du_1\int_{\R^2}|k_1(v,u)|du_2du_3 \|f\|^2\nnm\\
\le& C\frac1{|\zeta|}\int_{\R} \frac1{(\nu_0+x)^2+u_1^2}  du_1\|f\|^2\le C\delta^{-1}|\xi|^{-1}\|f\|^2.
\ema
This gives \eqref{T_1}.
Similarly, for $\delta=\nu_0+x >0$, $|\zeta|>1$  and $|\eta|\le \nu_0$,
\bmas
&\|(\lambda+\nu(v)+i v_1\xi)^{-1}K_1f\|^2 \\
\le& \intr |\nu(v)+\lambda+iv_1\xi|^{-2}  \(\intr k_1(v,u)|f(u)|^2 du\)dv \\
\le& C\sup_{u\in\R^3}\intr k_1(v,u)\frac1{(\nu(v)+x-v_1\eta)^2+(y+ v_1\zeta)^2}dv\|f\|^2 \\
\le& C\delta^{-1}|\xi|^{-1}\|f\|^2,
\emas
and
\bmas
&\|(\lambda+\nu(v)+i v_1\xi)^{-1}\chi_j\|^2\\
\le& C\int_{\R} \frac1{(\nu(v)+x-v_1\eta)^2+(y+ v_1\zeta)^2} e^{-\frac{v_1^2}4}\(\int_{\R^2}e^{-\frac{v_2^2+v_3^2}4} dv_2dv_3\)dv_1\\
\le & C\frac1{|\zeta|}\int_{\R} \frac1{(\nu_0+x)^2+v_1^2}  dv_1\le C\delta^{-1}|\xi|^{-1},
\emas
which yield \eqref{T_2} and \eqref{T_4}.

Finally, we prove \eqref{T_5} as follows. For $x>-\nu_0$, we write 
$$\lambda-\BB_1(\xi) =(\lambda-c(\xi))\(I-(\lambda-c(\xi))^{-1}K_1+i(\lambda-c(\xi))^{-1}\frac{v_1}{\xi}P_d\).$$
By \eqref{T_2} and \eqref{T_4}, there exists $R_0>0$ sufficiently large such that  for $ \nu_0+x>0$, $|\eta|<\nu_0$ and $|\zeta|>R_0$,
$$\|(\lambda-c(\xi))^{-1}K_1\|\le 1/4,\quad \|(\lambda-c(\xi))^{-1}\frac{v_1}{\xi}P_d\|\le 1/4.$$
This implies that the operator
$I-(\lambda-c(\xi))^{-1}K_1+ i(\lambda-c(\xi))^{-1}\frac{v_1}{\xi}P_{d}$
is invertible on $L^2(\R^3_v)$.
Hence, $\lambda-\BB_1(\xi)$ is invertible $L^2(\R^3_v)$ for $\nu_0+x>0$, $|\eta|<\nu_0$ and $|\zeta|>R_0$, and satisfies
$$(\lambda-\BB_1(\xi))^{-1} =\(I-(\lambda-c(\xi))^{-1}K_1+i(\lambda-c(\xi))^{-1}\frac{v_1}{\xi}P_d\)^{-1}(\lambda-c(\xi))^{-1}.$$
This together with \eqref{T_4} gives \eqref{T_5}.
\end{proof}

We now consider the analyticity and expansion of the eigenvalues and eigenvectors of $\AA_1(\xi)$ for  high frequency. For any $a,b>0$, denote two strip domains in $\C$ as
\bma
D_{a,b}&=\{\xi\in \mathbb{C}\,|\, |{\rm Re} \xi|\ge a,\,|{\rm Im} \xi|\le b\}, \label{Dab}\\
D_{b}&=\{\xi\in \mathbb{C}\,|\, {\rm Re} \xi\in \R,\,|{\rm Im} \xi|\le b\}. \label{Db}
\ema
Let $U=(f,X,Y)$ be the eigenvector of \eqref{L_2}--\eqref{L_2b}. By \eqref{L_2} and Lemma \ref{lem1}, there exists $R_0>0$ sufficiently large such that
\bq f=-(\BB_1(\xi)-\lambda)^{-1}\bar{v}\chi_0\cdot X,\quad  {\rm Re}\lambda>-\nu_0,\,\,\,\xi\in D_{R_0,\nu_0}.\label{L_5}\eq
Substituting \eqref{L_5} into \eqref{L_2a},
we obtain
\bma
\lambda X&=((\BB_1(\xi)-\lambda)^{-1}\chi_2,\chi_2)X+ i \xi\O_1 Y,\label{A_12}\\
  \lambda Y&=- i \xi\O_1 X,\quad  \xi\in D_{R_0,\nu_0}.  \label{A_13}
\ema
Multiplying \eqref{A_12} by $\lambda$ and using \eqref{A_13}, we obtain
\bq (\lambda^2-((\BB_1(\xi)-\lambda)^{-1}\chi_2,\chi_2)\lambda+\xi^2)X=0, \quad  \xi\in D_{R_0,\nu_0}.\eq
For ${\rm Re}\lambda>-\nu_0$, denote
\bq D(\lambda,\xi)=\lambda^2-((\BB_1(\xi)-\lambda)^{-1}\chi_2,\chi_2)\lambda+\xi^2.\label{D3a}\eq

The equation $D(\lambda,\xi)=0$  can be solved as follows.

\begin{lem}\label{eigen_5}
There exists a large constant $r_1>0$ such that the equation $D(\lambda,\xi)=0$ has two solutions $\lambda_j(\xi)=ji\xi+y_j(\xi)+z_j(\xi)$ for $j=\pm1$ and $\xi\in D_{r_1,\frac{\nu_0}2}$, where $y_j(\xi)$ and $z_j(\xi)$ are analytic functions for   $\xi\in D_{r_1,\frac{\nu_0}2}$  satisfying
\be
  \frac{C_1}{|\xi|}\le -{\rm Re}y_j(\xi)\le \frac{C_2}{|\xi|}, \quad |{\rm Im}y_j(\xi)|\le C_3\frac{\ln |\xi|}{|\xi|},\quad | z_j(\xi)|\le C_4\frac{\ln^2 |\xi|}{|\xi|^2},\label{eigen_h1}
\ee
where $C_1,C_2,C_3,C_4>0$ are constants.
\end{lem}
\begin{proof}
For any fixed $\xi\in D_{R_0,\nu_0}$, we define a function of $\lambda$ by
\bq F^1_j(\lambda)=\frac12\Big(B_{0}(\lambda,\xi)+j\sqrt{B_{0}(\lambda,\xi)^2-4\xi^2}\Big),\quad j=\pm1, \label{fp1}\eq
where $B_{0}(\lambda,\xi)=((\BB_1(\xi)-\lambda)^{-1}\chi_2,\chi_2)$.
It is straightforward to verify that a solution of $D(\lambda,\xi)=0$ for any fixed $\xi\in D_{R_0,\nu_0}$ is a fixed point of $F^1_j(\lambda)$.

Consider an equivalent equation of \eqref{fp1} as
\bq F^2_j(\beta)=F^1_j(\lambda)-ji \xi=\frac12\bigg(B_j(\beta,\xi)+\frac{jB_j(\beta,\xi)^2}{\sqrt{B_j(\beta,\xi)^2-4\xi^2}+2i \xi}\bigg),\quad j=\pm1, \label{fp2}\eq
where
$$ B_j(\beta,\xi)=\((\BB_1(\xi)-ji \xi-\beta)^{-1}\chi_2,\chi_2\). $$
By a similar argument as Lemma \ref{lem1}, there exists $r_1>0$  sufficiently large such that for $\xi\in D_{ \frac{\nu_0}2}$,
\bma
\|(\nu(v)+i(v_1+j)\xi)^{-1}\chi_2\| &\le C (1+|\xi|)^{-\frac12},\label{T_6a}\\
\|(\nu(v)+i(v_1+j)\xi)^{-1}K_1\|&\le C (1+|\xi|)^{-\frac12},\label{T_6b}
\ema
and for $\xi\in D_{r_1,\frac{\nu_0}2}$ and ${\rm Re}\beta \ge -\nu_0/4$,
\be
\|(\BB_1(\xi)-ji \xi-\beta)^{-1}\chi_2\| \le C (1+|\xi|)^{-\frac12}. \label{T_6}
\ee
Thus,  there exists  $\delta>0$
sufficiently small such that for $\xi\in D_{r_1,\frac{\nu_0}2}$ and $|\beta|,|\beta_1|,|\beta_2|\le \delta$,
$$
|F^2_j(\beta)|\le \delta,\quad |F^2_j(\beta_1)-F^2_j(\beta_2)|\le \frac12|\beta_1-\beta_2|.
$$
This implies that $F^2_j(\beta)$ is a contraction mapping on $B (0,\delta)$ and hence there exists a unique fixed point $\beta_j(\xi)\in B (0,\delta)$ of $F^2_j(\beta)$. Thus,  $\lambda_j(\xi)=ji \xi+\beta_j(\xi)$ is the solution of $D(\lambda,\xi)=0$ for  $\xi\in D_{r_1,\frac{\nu_0}2}$. Moreover, since $F^2_j(\beta)$ is analytic with respect to $(\beta,\xi)\in B (0,\delta)\times D_{r_1,\frac{\nu_0}2}$, it follows that $ \beta_j(\xi)$ is an analytic function with respect to $\xi\in D_{r_1,\frac{\nu_0}2}.$ By \eqref{T_6} and \eqref{fp2}, we have
\be |\beta_j(\xi)|\le C|B_j(\beta_j,\xi)|\le C(1+|\xi|)^{-\frac12}\to 0,\quad |\xi|\to \infty. \label{beta}\ee

We now turn to prove \eqref{eigen_h1}.
For this, we decompose $\beta_j(\xi)$, $j=\pm1$, into
\be
\beta_j(\xi)=\frac12 B_j(\beta_j,\xi)+\frac12\frac{jB_j(\beta_j,\xi)^2}{\sqrt{B_j(\beta_j,\xi)^2-4\xi^2}+ 2i \xi} =:\frac12I_1+\frac12I_2.\label{ddd}
\ee
First, we estimate $I_1$.  Decompose
\be
-\(L_1-i(v_1+j)\xi- i \frac{v_1}{\xi} P_d-\beta_{j} \)^{-1}=R_{j}(\xi)+Z_{j}(\xi),\quad j=\pm1, \label{decompose}
\ee
where
\bma
R_{j}(\xi)&=(\nu(v)+i(v_1+j)\xi)^{-1},\nnm\\
 Z_{j}(\xi)&=(I-Y_{j}(\xi))^{-1}Y_{j}(\xi)R_{j}(\xi),\label{az} \\
Y_{j}(\xi)&=   R_{j}(\xi)\Big(K_1- i \frac{v_1}{\xi} P_d-\beta_{j}\Big) .\nnm
\ema
By \eqref{decompose}, we divide $I_1$ into
\be
I_1= -\(R_{j}(\xi)\chi_2,\chi_2\)- \(Z_{j}(\xi)\chi_2,\chi_2\)=:I_3+I_4. \label{I_1}
\ee
Let
$\xi=\zeta+i\eta$ with $(\zeta,\eta)\in \R\times \R$. We have
\bma
I_3
=&-\intr \frac{\nu(v)+(v_1\pm 1)\eta}{(\nu(v)+(v_1\pm 1)\eta)^2+(v_1\pm 1)^2\zeta^2}v^2_2M(v)dv\nnm\\
&-i\intr \frac{(v_1\pm 1)\zeta}{(\nu(v)+(v_1\pm 1)\eta)^2+(v_1\pm 1)^2\zeta^2}v^2_2M(v)dv\nnm\\
=&-{\rm Re}I_3-i {\rm Im}I_3. \label{aaa}
\ema
By changing variable $(v_1\pm1)|\zeta|\to u_1$, it holds that for $|\eta|\le \nu_0/2$ and $|\zeta|\ge 2$,
\bma
{\rm Re}I_3\le& C\intr \frac{1}{\nu_0^2+(v_1\pm 1)^2\zeta^2}e^{-\frac{|v|^2}4}dv\nnm\\
=& \frac{C}{|\zeta|}\intr \frac{1}{\nu_0^2+u_1^2}e^{-\frac14(\frac{u_1}{|\zeta|}\mp1)^2}e^{-\frac{v_2^2+v_3^2}4} du_1dv_2dv_3\nnm\\
\le& \frac{C}{|\zeta|}\int^\infty_0 \frac{1}{\nu_0^2+u_1^2}du_1\le \frac{C_1}{|\xi|}, \label{bbb}
\ema
and
\bma
{\rm Re}I_3&\ge C\intr \frac{1}{\nu_1^2(1+|v|^2)+(v_1\pm1)^2\zeta^2}v_2^2e^{-\frac{|v|^2}2}dv\nnm\\
&= \frac{C}{|\zeta|}\intr \frac{1}{\nu_1^2(1+(\frac{u_1}{|\zeta|}\mp1)^2+v^2_2+v^2_3)+u_1^2} v_2^2e^{-\frac12(\frac{u_1}{|\zeta|}\mp1)^2}
e^{-\frac{v_2^2+v_3^2}2} du_1dv_2dv_3\nnm\\
&\ge \frac{C}{|\zeta|}\intr \frac{1}{\nu_1^2(3+u_1^2+v^2_2+v^2_3)+u_1^2}v_2^2 e^{-\frac{u_1^2+v_2^2+v_3^2}2}du_1dv_2dv_3\ge \frac{C_2}{|\xi|}.
\ema
Similarly,
\bma
|{\rm Im}I_3|\le& \frac{C}{|\zeta|}\intr \frac{|u_1|}{\nu_0^2+u_1^2}v^2_2e^{-\frac12(\frac{u_1}{|\zeta|}\mp1)^2}e^{-\frac{v_2^2+v_3^2}2}du_1dv_2dv_3\nnm\\
\le& \frac{C}{|\zeta|}\int^{|\zeta|}_0 \frac{u_1}{\nu_0^2+u_1^2}du_1+\frac{ C}{|\zeta|^2}\int^\infty_{|\zeta|} e^{-\frac12\frac{u^2_1}{|\zeta|^2}}du_1\le C_3\frac{\ln |\xi|}{|\xi|}. \label{ddd-1}
\ema

To estimate  $I_4$, we first claim that
\be |\beta_j(\xi)|\le (C_1+C_3)|\xi|^{-1} \ln |\xi| ,\quad \xi\in D_{r_1,\frac{\nu_0}2}.\label{assume1}\ee
By changing variable $v_2\to-v_2$, it can be verified that $$P_dR_{j}(\xi)^n \chi_2 =P_d(R_{j}(\xi) K_1)^nR_{j}(\xi) \chi_2=0,\quad \forall n\ge 1.$$
This together with the fact that
$(I-Y_j(\xi))^{-1}=I+(I-Y_j(\xi))^{-1}Y_{j}(\xi)  $
yields
\be Z_{j}(\xi)\chi_2=R_{j} (K_1 -\beta_{j})R_{j} \chi_2+(I-Y_{j} )^{-1} [R_{j} (K_1 -\beta_{j})]^2R_{j} \chi_2. \label{Z1}\ee
By \eqref{K1}, \eqref{bbb} and \eqref{ddd-1}, we obtain
\bma
|K_1R_{j}(\xi)\chi_2|
&\le C\int_{\R} e^{-\frac{|v_1-u_1|^2}{8}}\frac{1}{\nu_0+|(u_1\pm1)\zeta|}e^{-\frac{u_1^2}{4}}du_1 \int_{\R^2}\frac{1}{|\bar{v}-\bar{u}|}e^{-\frac{|\bar{v}-\bar{u}|^2}{8}}e^{-\frac{|\bar{u}|^2}{4}}d\bar{u}\nnm\\
&\le Ce^{-\frac{|v|^2}{8}}\int_{\R} \frac{1}{\nu_0+|(u_1\pm1)\zeta|}e^{-\frac{u_1^2}{8}}du_1\le C\frac{\ln|\xi|}{|\xi|}e^{-\frac{|v|^2}{8}}, \label{ggg1}
\ema
which leads to
\bma
\|R_{j}K_1R_{j}\chi_2 \|^2_{L^2_v}&\le C\frac{\ln^2|\xi|}{|\xi|^2}\intr \frac{1}{\nu_0^2+(v_1\pm1)^2\zeta^2}e^{-\frac{|v|^2}{4}}dv\le C\frac{\ln^2|\xi|}{|\xi|^3}, \label{fff}\\
\|R_{j}K_1R_{j}\chi_2 \|_{L^1_v}&\le C\frac{\ln|\xi|}{|\xi|}\intr \frac{1}{\nu_0+|(v_1\pm1)\zeta|}e^{-\frac{|v|^2}{8}}dv\le C\frac{\ln^2|\xi|}{|\xi|^2}. \label{ggg}
\ema
By \eqref{beta} and Lemma \ref{lem1}, it holds that for $\xi\in D_{r_1,\frac{\nu_0}2}$ with $r_1>0$ large enough,
\be \|(I-Y_j(\xi))^{-1}\|\le 2, \quad j=\pm 1. \label{Yj}\ee
Thus, it follows from \eqref{bbb},  \eqref{assume1} and \eqref{fff}--\eqref{Yj} that
\bma
|I_4|\le &|(R_{j}(K_1 -\beta_{j}) R_{j}\chi_2, \chi_2)|+|([R_{j}(K_1 -\beta_{j})]^2 R_{j}\chi_2, (I-\overline{Y_j })^{-1}\chi_2)|\nnm\\
\le &C(\|R_{j}K_1R_{j}\chi_2\|_{L^1_v} +|\beta_{j}| \| R_{j}^2\chi_2\|_{L^1_v} )\nnm\\
&+C(\|R_{j}K_1\|+|\beta_{j}|)(\|R_{j}K_1R_{j}\chi_2\|_{L^2_v}+|\beta_{j}| \| R_{j}^2\chi_2\|_{L^2_v}) \nnm\\
\le &C_4\frac{\ln^2|\xi|}{|\xi|^2}+C_4\frac{\ln|\xi|}{|\xi|^2}. \label{I4}
\ema
Finally, we estimate $I_2$ as
\bq |I_2|\le \frac{C}{|\xi|}|B_{j}(\beta_{j},\xi)|^2\le C_5\frac{\ln^2|\xi|}{|\xi|^3}.\label{abb}\eq
Combining \eqref{ddd}, \eqref{I_1}--\eqref{ddd-1}, \eqref{I4} and \eqref{abb}, we obtain \eqref{assume1} and \eqref{eigen_h1}, where
$$
y_j(\xi)=\frac12I_3=-\frac12\(R_{j}(\xi)\chi_2,\chi_2\),\quad
z_j(\xi)=\frac12I_2+\frac12I_4.
$$
The proof of the lemma is then completed.
\end{proof}

With the help of Lemma \ref{eigen_5}, we have the anayticity and expansion of the eigenvalues $\beta_j(\xi)$
and eigenvectors $\phi_j(\xi)$ of $\AA_1(\xi)$ in the high frequency region.

\begin{thm}\label{eigen_4}
(1) There exists a constant $r_1>0$ such that the spectrum $\sigma(\AA_1(\xi))\cap \{\lambda\in\mathbb{C}\,|\,\mathrm{Re}\lambda>-\mu/2\}$  consists of four eigenvalues $\{\beta_j(\xi),\ j=1,2,3,4\}$ for  $\xi\in D_{r_1,\frac{\nu_0}2}$. In particular, the eigenvalues $\beta_j(\xi)$ are analytic in $D_{r_1,\frac{\nu_0}2}$ and have the following expansions
 \be   \label{specr1}
  \left\{\bln
 &\beta_j(\xi)= \alpha_j(\xi)+\gamma_j(\xi)+O\(\frac{\ln^2 |\xi|}{|\xi|^2}\),\\
 & \alpha_1(\xi)=\alpha_2(\xi)=-i\xi, \quad \alpha_3(\xi)=\alpha_4(\xi)=i\xi,\\
 & \gamma_j(\xi)=-\frac12\intr \frac{1}{\nu(v)+iv_1\xi+\alpha_j}v_2^2M(v)dv,
  \eln\right.
\ee
where  $\gamma_j(\xi)$ is an analytic function for $\xi\in D_{\frac{\nu_0}2}$ satisfying
\be \label{specr3}
 \frac{C_1}{1+|\xi|}\le -{\rm Re}\gamma_j(\xi)\le \frac{C_2}{1+|\xi|},\quad  |{\rm Im}\gamma_j(\xi)|\le C_3 \frac{\ln (2+|\xi|)}{1+|\xi|},
\ee
for   positive constants $C_1,C_2$ and $C_3$. Here, the domains $D_{r_1,\frac{\nu_0}2}$ and $D_{\frac{\nu_0}2}$ are defined by \eqref{Dab}--\eqref{Db}, respectively.

(2) The eigenvectors $\phi_j(\xi)=\(u_j(\xi),\z^2_j(\xi),\z^3_j(\xi)\)$, $j=1,2,3,4$ are analytic functions for $\xi\in D_{r_1,\frac{\nu_0}2}$  and satisfy
 \bq   \label{eigf2}
  \left\{\bln
  &(\phi_i(\xi),\phi^*_j(\xi))=(u_i,\overline{u_j})-(\z^2_i,\overline{\z^2_j})-(\z^3_i,\overline{\z^3_j})=\delta_{ij}, \quad 1\le i,j\le 4,\\
&u_j(\xi)=c_j(\xi)\Big(\beta_j(\xi)-L_1 + i v_1 \xi+i\frac{  v_1}{\xi} P_d\Big)^{-1}(\bar{v}\cdot \e_j\chi_0),\\
&\z^2_j(\xi)=c_j(\xi)\mathbf{e}_j,\quad \z^3_j(\xi)= -\frac{ i \xi c_j(\xi)}{\beta_{j}(\xi)}\O_1 \mathbf{e}_j,
  \eln\right.
  \eq
where  $ \phi^*_j=(\overline{u_j},-\overline{\z^2_j},-\overline{\z^3_j})$, $\e_1=\e_3=  (\sqrt{1/2},0)$, $\e_2=\e_4= (0,\sqrt{1/2})$, and $c_j(\xi)$ are analytic functions of $\xi$ for $\xi\in D_{r_1,\frac{\nu_0}2}$ satisfying
\be \label{eigfr3}
\left\{\bln
&c_j(\xi)= i+i\frac12 d_j(\xi) +O\(\frac{\ln|\xi|}{|\xi|^2}\),\\
&d_j(\xi)=\frac12\intr \frac{1}{(\nu(v)+iv_1\xi+\alpha_j)^{2}}v_2^2M(v)dv.
 \eln\right.
\ee
\end{thm}

\begin{proof}
The eigenvalues $\beta_j(\xi)$ and the eigenvectors $\phi_j(\xi)=(u_j(\xi),\z^2_j(\xi),\z^3_j(\xi))$  for $j=1,2,3,4$ can be constructed as follows. We take $\beta_1=\beta_2=\lambda_{-1}(\xi)$ and $\beta_3=\beta_4=\lambda_{1}(\xi)$ to be the solution of the equation  $D(\lambda,\xi)=0$ defined in Lemma \ref{eigen_5}.
The corresponding eigenvectors $\phi_j(\xi)=(u_j(\xi),\z^2_j(\xi),\z^3_j(\xi))$, $1\le j\le 4$  are given by
 \bq
 \left\{\bln       \label{C_3a}
  &u_j(\xi)=c_j(\xi)\Big(\beta_j(\xi)-L_1 + i v_1\xi+i\frac{  v_1}{\xi} P_d\Big)^{-1}
           (\bar{v}\cdot \e_j\chi_0),  \\
  &\z^2_j(\xi)=c_j(\xi)\e_j,\quad \z^3_j(\xi)=-\frac{ i \xi c_j(\xi)}{\beta_j(\xi)}\O_1 \e_j ,
  \eln\right.
\eq
where $\e_1=\e_3=(\sqrt{1/2},0)$ and $\e_2=\e_4=(0,\sqrt{1/2})$.  It is easy to verify that   $(\phi_1(\xi),  \phi_3^*(\xi))=(\phi_2(\xi),  \phi_4^*(\xi))=0,$ where $ \phi^*_j=(\overline{u_j},-\overline{\z^2_j},-\overline{\z^3_j})$ is the eigenvector of $\AA^*_1(\xi)$ corresponding to the eigenvalue $\overline{\beta_{j}(\xi)}$.

Rewrite the eigenvalue problem as
$$
 \AA_1 (\xi)\phi_j(\xi) =\beta_j(\xi)\phi_j(\xi), \quad j=1,2,3,4.
 $$
By taking the inner product
$(\cdot,\cdot)_{\xi}$ of above with $\phi^*_j(\xi)$, and using
the fact that
 \bmas
(\AA_1 (\xi) U,V)_{\xi}&=(U,\AA^*_1 (\xi)V)_{\xi},\quad U,V\in  D(\BB_1(\xi))\times \mathbb{C}^2\times \mathbb{C}^2,\\
\AA^*_1 (\xi)\phi^*_j(\xi)&=\overline{\beta_j(\xi)} \phi^*_j(\xi),
\emas
we have
 $$
(\beta_i(\xi)-\beta_{j}(\xi))(\phi_i(\xi),\phi^*_j(\xi))_{\xi}=0,\quad 1\le i,j\le 4.
$$
Since  $\beta_i(\xi)\neq \beta_{j}(\xi)$ for  $i=1,3,\, j=2,4$ and $P_du_j(\xi)=0$,
we have the orthogonal relation
 $$
(\phi_i(\xi),\phi^*_j(\xi))_{\xi}=(\phi_i(\xi),\phi^*_j(\xi))=0,\quad 1\leq i\neq j\leq 4.
 $$
 This can be normalized so that
$$(\phi_j(\xi),\phi^*_j(\xi))=1, \quad j=1,2,3,4.$$
Precisely, the coefficients $c_j(\xi)$  is determined by the normalization condition
 \bq    \label{C_2a}
 c_j(\xi)^2\(D_j(\xi)-1+\frac{\xi^2}{\beta_j(\xi)^2}\)=1,\quad j=1,2,3,4,
 \eq
 where $D_j(\xi)=(R_1(\beta_j(\xi),\xi)\chi_2, R_1(\overline{\beta_j(\xi)},-\xi) \chi_2).$

Then, we turn to show the expansion of $c_j(\xi)$. For simplicity, we only consider the term $c_3(\xi)$. By \eqref{specr1}, we can expand $ \beta^2_3(\xi)/\xi^2$ as
\be
\frac{\beta^2_3(\xi)}{\xi^2}=\(i+O\(\frac{\ln|\xi|}{|\xi|^2}\)\)^2=-1 +O\(\frac{\ln|\xi|}{|\xi|^2}\). \label{bk}
 \ee
By \eqref{decompose},  we obtain
\bma
D_3(\xi)=&(R_{1}(\xi)\chi_2,\overline{R_1(\xi)}\chi_2)+2(R_{1}(\xi)\chi_2,\overline{Z_1(\xi)}\chi_2)+(Z_1(\xi)\chi_2,\overline{Z_1(\xi)}\chi_2)\nnm\\
=&2d_3(\xi)+O(|\xi|^{-2} \ln|\xi| ), \label{dk}
\ema
where we have used (cf. \eqref{T_6a}, \eqref{T_6b}, \eqref{Z1} and \eqref{fff})
\be \|R_{1}(\xi)\chi_2\| \le C|\xi|^{-\frac12},\quad \|Z_{1}(\xi)\chi_2\| \le   C|\xi|^{-\frac32}\ln |\xi| . \label{I4a}\ee

By substituting \eqref{bk} and \eqref{dk} into \eqref{C_2a}, we obtain
\bmas
c_3^2(\xi)=&-\bigg(2- \bigg(\frac{\xi^2}{\beta_3^2(\xi)}+1\bigg)- D_3(\xi)\bigg)^{-1}\\
=&-\frac{1}{2} - \frac{1}{2}d_3(\xi)+O\(\frac{\ln|\xi|}{|\xi|^2}\).
\emas
Thus, we obtain \eqref{eigf2} and \eqref{eigfr3} so that
the proof of the theorem is completed.
\end{proof}

 With the help of Theorems \ref{spectrum1}, \ref{eigen_3} and \ref{eigen_4}, we have the estimates on the semigroup $e^{t\AA_0(\xi)}$ to the linear system \eqref{LVMB1a}.

\begin{thm} \label{semigroup}
The semigroup $S(t,\xi)=e^{t\AA_0(\xi)}$ satisfies
 $$
 S(t,\xi)U=S_1(t,\xi)U+S_2(t,\xi)U+S_3(t,\xi)U,
     \quad U\in \hat{\mathcal X}_1, \ \ t>0, 
 $$
 where
 \bma
 S_1(t,\xi)U&=\sum^2_{j=1}e^{t\lambda_j(\xi)} \(U, \Psi^*_j(\xi)\) \Psi_j(\xi) 1_{\{|\xi|\leq r_0\}}, \label{S_2a}
               \\
 S_2(t,\xi)U&=\sum^4_{j=1} e^{t\beta_j(\xi)}  \(U,\Phi^*_j(\xi)\) \Phi_j(\xi)  1_{\{|\xi|\ge r_1\}}.  \label{S_3a}
 \ema
Here, $(\lambda_j(\xi),\Psi_j(\xi))$ and $(\beta_j(\xi),\Phi_j(\xi))$ are the eigenvalues and eigenvectors of the operator $\AA_0(\xi)$ for $|\xi|\le r_0$ and $|\xi|>r_1$ respectively,
and $S_3(t,\xi) =: S(t,\xi)-S_1(t,\xi)-S_2(t,\xi)$ satisfies
that there exists  a constant $\kappa_0>0$ independent of $\xi$ such  that
 $$
 \|S_3(t,\xi)U\| \leq Ce^{-\kappa_0t}\|U\| ,\quad t>0.
 $$
\end{thm}

\begin{proof}
For any $U_0=(f_0,E_0,B_0)\in \hat{\mathcal X}_1$,
define
 \be
  e^{t\AA_0(\xi)}U_0=\(V_1,\frac{1}{i\xi}(V_1,\chi_0), V_2,0,V_3\)\in \hat{\mathcal X}_1, \label{solution}
  \ee
where
\bgrs
e^{t\AA_1(\xi)}V_0=(V_1,V_2,V_3)\in L^2_{\xi}(\R^3_v)\times \mathbb{C}^2\times \mathbb{C}^2,\\
V_0=( f_0, E^r_0, B^r_0),\,\,\ E^r_0=(E^2_0,E^3_0),\,\,\ B^r_0=(B^2_0,B^3_0),
\egrs
with $e^{t\AA_1(\xi)}=\tilde{S}(t,\xi)$ given by Lemma \ref{spectrum1}.
By \eqref{LVMB1a} and \eqref{LVMB1}, we conclude that $e^{t\AA_0(\xi)}$ is the semigroup generated by the operator $\AA_0(\xi)$ satisfying
$$\|e^{t\AA_0(\xi)}U_0\|=\|e^{t\AA_1(\xi)}V_0\|_{\xi}\le \| V_0\|_{\xi}=\|U_0\|.$$
Set
\be \label{eigenvector}
\left\{\bal
\Psi_k(\xi)=(h_k,\mathrm{Y}^2_k,\Y^3_k)(\xi), & \Y^n_k(\xi)=(0,\y^n_k(\xi))\in \C^3, \quad n=2,3,\,\,k=1,2,\\
\Phi_j(\xi)=(u_j,\mathrm{Z}^2_j,\Z^3_j)(\xi), & \Z^n_j(\xi)=(0,\z^n_j(\xi))\in \C^3,  \quad n=2,3,\,\,j=1,2,3,4,
\ea\right.
\ee
where 
$(h_k,\y^2_k,\y^3_k)(\xi)$ and $(u_j,\z^2_j,\z^3_j)(\xi)$ are the eigenvectors of $\AA_1(\xi)$ for $|\xi|\le r_0$ and $|\xi|\ge r_1$ given in Lemma \ref{eigen_3} and Lemma \ref{eigen_4} respectively.
Since $(h_k,\chi_0)=(u_j,\chi_0)=0$, it follows that $(\lambda_k(\xi),\Psi_k(\xi))$ and $(\beta_j(\xi),\Phi_j(\xi))$ are the eigenvalues and the corresponding eigenvectors of $\AA_0(\xi)$ for $|\xi|\le r_0$ and $|\xi|\ge r_1$ respectively.
And set
 $$\Psi^*_k(\xi)=(\overline{h_k},-\overline{\mathrm{Y}^2_k},-\overline{\Y^3_k})(\xi),  \quad \Phi^*_j(\xi)=(\overline{u_j},-\overline{\mathrm{Z}^2_j},-\overline{\Z^3_j})(\xi).$$
Then $(\overline{\lambda_k(\xi)},\Psi^*_k(\xi))$ and $(\overline{\beta_j(\xi)},\Phi^*_j(\xi))$ are the eigenvalues and the corresponding eigenvectors of $\AA^*_0(\xi)$ for  $|\xi|\le r_0$ and $|\xi|\ge r_1$ respectively.

Thus, we can decompose $S(t,\xi)=e^{t\AA_0(\xi)}$ into
$$
 S(t,\xi)U_0=S_1(t,\xi)U_0+S_2(t,\xi)U_0+S_3(t,\xi)U_0,
 $$
where $S_1(t,\xi)$ and $S_2(t,\xi)$ are given by \eqref{S_2a} and \eqref{S_3a} respectively. And $S_3(t,\xi) =: S(t,\xi)-S_1(t,\xi)-S_2(t,\xi)$ satisfies
$$\|S_3(t,\xi)U_0\|=\|\tilde{S}_3(t,\xi)V_0\|_{\xi}\le Ce^{-\kappa_0t}\|V_0\|_{\xi}=Ce^{-\kappa_0t}\|U_0\|.$$
The proof of the theorem is completed.
\end{proof}

\section{Green's function:  fluid part}
\label{fluid}
\setcounter{equation}{0}

In this section, we will show the pointwise estimates of the fluid part of  Green's function $G(t,x)$  based on the spectral analysis given  in Section 2.
First, we decompose the operator $G(t,x)$ in $\mathcal{X}_1$ into the low-frequency part, middle-frequency part and high frequency part:
\be \label{L-R}
\left\{\bln
&G(t,x)=G_L(t,x)+G_M(t,x)+G_H(t,x),\\
&G_L(t,x)=\frac1{\sqrt{2\pi}}\int_{|\xi|\le \frac{r_0}{2}}e^{ i x \xi +t\AA_0(\xi)}d\xi,
\\
&G_M(t,x)=\frac1{\sqrt{2\pi}}\int_{\frac{r_0}{2}\le |\xi|\le r_1}e^{ i x \xi +t\AA_0(\xi)}d\xi,
\\
&G_H(t,x)=\frac1{\sqrt{2\pi}}\int_{|\xi|\ge r_1}e^{ i x \xi +t\AA_0(\xi)}d\xi.
\eln\right.
\ee

Furthermore, the operators $G_L(t,x)$ and $G_H(t,x)$ can be decomposed into the fluid parts and the non-fluid parts \cite{Li1}:
\bmas
G_L(t,x)&=G_{L,0}(t,x)+G_{L,1}(t,x),\\
G_H(t,x)&=G_{H,0}(t,x)+G_{H,1}(t,x),
\emas
where
\be \label{GL0}
\left\{\bln
\hat G_{L,0}(t,\xi)&=\sum_{j=1,2}e^{\lambda_j(\xi)t}\Psi_j(\xi)\otimes \langle\Psi^*_j(\xi)|,
\\
\hat G_{H,0}(t,\xi)&=\sum_{1\le j\le 4}e^{\beta_j(\xi)t}\Phi_j(\xi)\otimes \langle\Phi^*_j(\xi)|,
\\
\hat G_{L,1}(t,\xi)&=\hat G_{L}(t,\xi)-\hat G_{L,0}(t,\xi),\\
\hat G_{H,1}(t,\xi)&=\hat G_{H}(t,\xi)-\hat G_{H,0}(t,\xi),
\eln\right.
\ee
and $(\lambda_k(\xi),\Psi_k(\xi))$ $(k=1,2)$ and $(\beta_j(\xi),\Phi_j(\xi))$ $(j=1,2,3,4)$ are the eigenvalues and the eigenvectors of $\AA_0(\xi)$ given in Theorem \ref{semigroup}.
Precisely,
\bma
\hat{G}_{L,0} &=\sum^2_{j=1}e^{\lambda_j(\xi)t }\left( \ba
h_j\otimes\langle h_j| & -h_j\otimes\langle \Y^2_j| & -h_j\otimes\langle \Y^3_j|\\
 \Y^2_j\otimes \langle h_j| & -\Y^2_j\otimes\langle \Y^2_j| &  -\Y^2_j\otimes\langle \Y^3_j|\\
  \Y^3_j\otimes \langle h_j| & -\Y^3_j\otimes\langle \Y^2_j| &  -\Y^3_j\otimes\langle \Y^3_j|
 \ea\right)=:\(\hat{G}_{L,0}^{ij}\)_{3\times 3}, \label{G_L0}
 \\
 \hat{G}_{H,0} &=\sum^4_{j=1}e^{\beta_j(\xi)t }\left( \ba
u_j\otimes\langle u_j| & -u_j\otimes\langle \Z^2_j| & -u_j\otimes\langle \Z^3_j|\\
 \Z^2_j\otimes \langle u_j| & -\Z^2_j\otimes\langle \Z^2_j| &  -\Z^2_j\otimes\langle \Z^3_j|\\
  \Z^3_j\otimes \langle u_j| & -\Z^3_j\otimes\langle \Z^2_j| &  -\Z^3_j\otimes\langle \Z^3_j|
 \ea\right)=:\(\hat{G}_{H,0}^{ij}\)_{3\times 3}, \label{G_H0}
\ema
where $(h_k,\Y^2_k,\Y^3_k)$, $k=1,2$ and $(u_j,\Z^2_j,\Z^3_j)$, $j=1,2,3,4$ are defined by \eqref{eigenvector}.
Here, the operator $f\otimes\langle g|$  is defined in \cite{Liu1,Liu2} as
$$ f\otimes\langle g|h=(h,\overline{g})f .$$

 From \cite{Li1}, we have the following time-decay rates of each part of $\hat{G}(t,\xi)$ defined by \eqref{L-R}.

\begin{lem} \label{l-1}
For any $U_0\in  \mathcal{\hat{X}}_1$, there exist constants $C>0$ and $\kappa_0>0$ such that
\bma \|\hat{G}(t,\xi)U_0\| &\le  \|U_0\| , \label{l1}\\
\|\hat{G}_M(t,\xi)U_0\| ,\|\hat{G}_{L,1}(t,\xi)U_0\|,\|\hat{G}_{H,1}(t,\xi)U_0\|&\le Ce^{-\kappa_0t} \|U_0\| , \label{high1}
\ema
where $\hat{G}(t,\xi)$, $\hat{G}_M(t,\xi)$, $\hat{G}_{H,1}(t,\xi)$ and $\hat{G}_{L,1}(t,\xi)$ are defined by \eqref{L-R} and \eqref{GL0} respectively.
\end{lem}

\begin{lem}\label{gl0}
For  any given constant $C_1>1$, there exist $C,D>0$ such that for $|x|\le C_1t$,
\bma
\|\dxa  G_{L,0}^{1i}(t,x)\|,\|\dxa  G_{L,0}^{2i}(t,x)\|&\le C\((1+t)^{-\frac{3+\alpha}2}e^{-\frac{x^2}{Dt}}+e^{-\frac{t}{D}}\), \label{gl_1}\\
\|\dxa  G_{L,0}^{i3}(t,x) \|,\|\dxa  G_{L,0}^{3i}(t,x) \|&\le C\((1+t)^{-\frac{2+\alpha}2}e^{-\frac{x^2}{Dt}}+e^{-\frac{t}{D}}\), \label{gl_2}\\
\|\dxa  G_{L,0}^{33}(t,x) \|&\le C\((1+t)^{-\frac{1+\alpha}2}e^{-\frac{x^2}{Dt}}+e^{-\frac{t}{D}}\),\label{gl_3}
\ema
where $i=1,2,$ and $\alpha\ge 0$.
\end{lem}

\begin{proof}By \eqref{G_L0}, we have
$$
 \hat{G}_{L,0}^{11}(t,\xi)=e^{\lambda_1(\xi)t}\sum_{j=1}^2h_j(\xi)\otimes\langle h_j(\xi)|.
$$
 Without loss of generality, we assume $x\ge 0$. Since $e^{\lambda_1(\xi)t}h_j(\xi)\otimes\langle h_j(\xi)|$ is analytic for $|\xi|\le r_0$, we obtain by Cauchy integral theorem that
 \bma
 \dxa G_{L,0}^{11}(t,x)&=\int_{|\xi|\le \frac{r_0}{2}} e^{ix\xi}e^{\lambda_1(\xi)t}\xi^{\alpha}\sum_{j=1}^2h_j(\xi)\otimes\langle h_j(\xi)|d\xi\nnm\\
 &=\(\int_{\Gamma_1\cup\Gamma_3}+ \int_{\Gamma_2}\)e^{ix\xi}e^{\lambda_1(\xi)t}\xi^{\alpha}\sum_{j=1}^2h_j(\xi)\otimes\langle h_j(\xi)|d\xi\nnm\\
 &=:J_1+J_2, \label{g1}
 \ema
 where
$$
\left\{\bln
\Gamma_1=&\left\{\xi\,|\,\xi=\frac{r_0}{2}+i y,\,0\le y \le \frac{x}{C_2t}\right\},\\
 \Gamma_2=&\left\{\xi\,|\,\xi=u+i \frac{x}{C_2t},\,-\frac{r_0}{2}\le u \le \frac{r_0}{2}\right\},\\
\Gamma_3=& \left\{\xi\,|\,\xi=-\frac{r_0}{2}+i y,\,0\le y \le \frac{x}{C_2t}\right\}.
\eln\right.
$$
For $|x|\le C_1t$, by taking $C_2=\max\{\frac{2C_1}{r_0},2a_1 \}$ we obtain
\bma
\|J_1(t,x)\|&\le Ce^{-\frac{a_1r^2_0}4t}\int^{\frac{x}{C_2t}}_0  e^{-y(\frac{x}t-a_1y)t +o(1)(r_0^2+y^2)t}\(r_0^2+y^2\)^{1+\frac{\alpha}2} dy\nnm\\
&\le  Ce^{-\frac{a_1r^2_0}8t},\label{g2}
\\
\|J_2(t,x)\|
&\le  Ce^{-\frac{x^2}{C_2t}(1-\frac{a_1}{C_2})}\int^{\frac{r_0}2}_{-\frac{r_0}2}
 e^{-a_1u^2t+o(1)(u^2+\frac{x^2}{t^2})t}\(u^2+\frac{x^2}{t^2}\)^{1+\frac{\alpha}2}du\nnm\\
&\le C(1+t)^{-\frac32-\frac{\alpha}2}e^{-\frac{x^2}{Dt}}.\label{g3}
\ema

Combining \eqref{g1}--\eqref{g3}, we obtain \eqref{gl_1}. By a similar argument as above, we can prove \eqref{gl_2} and \eqref{gl_3}. This completes the proof of the lemma.
\end{proof}

Set
\be \label{G_b}
\hat{G}_{2}(t,\xi)=\sum^4_{j=1}e^{\alpha_jt+\gamma_jt}\left( \ba
-u^0_j\otimes \langle u^0_j| & u^0_j\otimes \langle \X^2_j| & u^0_j\otimes \langle \X^3_j|\\
-\X^2_j\otimes \langle u^0_j| & (1+d_{j})\X^2_j\otimes \langle \X^2_j| &  (1+d_{j})\X^2_j\otimes \langle \X^3_j|\\
-\X^3_j\otimes \langle u^0_j| & (1+d_{j})\X^3_j\otimes \langle \X^2_j| &  (1+d_{j})\X^3_j\otimes \langle \X^3_j|
\ea\right),
\ee
where $\alpha_j$, $\gamma_j$ and $d_j$ are given by Theorem \ref{eigen_4}, and $u^0_j$, $\X^2_j$, $\X^3_j$ are given by
\be \label{R0j}
\left\{\bln
&\X^2_1=\X^2_3=\sqrt{\frac12}(0,1,0),\quad \X^2_2=\X^2_4=\sqrt{\frac12}(0,0,1), \\
&\X^3_1=-\X^3_3=\sqrt{\frac12}(0,0,1),\quad \X^3_2=-\X^3_4=\sqrt{\frac12}(0,-1,0),\\
&u^0_j = \frac{1}{\nu(v)+i v_1\xi+\alpha_j} ( v\cdot \X^2_j)\chi_0,\quad j=1,2,3,4.
\eln\right.
\ee

Decompose $G_2$ into
\be \hat{G}_2=\hat{G}_21_{\{|\xi|< r_1\}}+\hat{G}_21_{\{|\xi|\ge r_1\}}=:\hat{G}_{2,1}+\hat{G}_{2,2}.\label{G2-1}\ee
For any $\alpha\ge 0$ and $b>0$, define the operator $\dx^{-\alpha}$ and $(b+\dx)^{-\alpha}$ by
\bma
\dx^{-\alpha}g(x)&=\frac1{\sqrt{2\pi}}\intra e^{ix\xi}(i\xi)^{-\alpha}\hat{g}(\xi)d\xi, \label{dxa}\\
(b+\dx)^{-\alpha}g(x)&=\frac1{\sqrt{2\pi}}\intra e^{ix\xi}(b+i\xi)^{-\alpha}\hat{g}(\xi)d\xi. \label{dxa-1}
\ema
We will show that the term $(G_{H,0}-G_2)(t,x)$ is bounded and has the following pointwise estimate.

\begin{lem}\label{gh0}
Let $G_{H,0}$, $G_{2,1}$ and $G_{2,2}$  be defined by \eqref{G_H0} and \eqref{G2-1}, respectively. For any integer $\alpha\ge 0$, define
\be
 F_{\alpha}(t,x)=\dx^{-\alpha}\[ G_{H,0}(t,x)-G_{2,2}(t,x)\], \label{gl_4}
\ee
then we have
  \be \label{fij}
  \left\{\bln
 &\|F^{11}_{\alpha}(t,x)\|\le  C\bigg(\sum_{l=\pm1}(1+t)^{-\alpha-1}\ln^2(2+t)e^{-\frac{\nu_0|x-lt|}2} +e^{-\frac{t}{D}}\bigg),\\
 &\|F^{1j}_{\alpha}(t,x)\|, \|F^{j1}_{\alpha}(t,x)\|\le  C\bigg(\sum_{l=\pm1}(1+t)^{-\alpha-\frac12}\ln^2(2+t)e^{-\frac{\nu_0|x-lt|}2} +e^{-\frac{t}{D}}\bigg), \\
 & \|F^{2j}_{\alpha}(t,x)\|, \|F^{3j}_{\alpha}(t,x)\|\le  C\bigg(\sum_{l=\pm1}(1+t)^{-\alpha}\ln^2(2+t)e^{-\frac{\nu_0|x-lt|}2} +e^{-\frac{t}{D}}\bigg),
 \eln\right.
 \ee
 for $j=2,3,$ and $C,D>0$ two constants. In addition,
 \be \|G^{ij}_{2,1}(t,x)\|\le Ce^{-\frac tD},\quad 1\le i,j\le 3. \label{G_21}\ee
\end{lem}
\begin{proof}
By \eqref{specr1}, \eqref{specr3} and \eqref{G_b}, we have
\bmas \|G_{2,1}(t,x)\|&\le C\int^\infty_{-\infty}\left\| \hat{G}_2(t,\xi)1_{\{|\xi|< r_1\}}\right\|d\xi
\\
&\le C\int^{r_1}_{-r_1}e^{{\rm Re}\gamma_j(\xi)t}d\xi\le Ce^{-t/D}.
\emas

Write $I_1=(I^{ij}_1)_{3\times 3}$ with $\hat{I}^{ij}_1=\hat{G}^{ij}_{H,0} -\hat{G}^{ij}_{2} 1_{\{|\xi|\ge r_1\}}$. We first  estimate $I^{12}_1(t,x) $ as follows.
$$
I^{12}_1=\sum^4_{j=1}C\int_{|\xi|\ge r_1}e^{ix\xi} \(-e^{\beta_j t} u_j\otimes \langle \Z^2_j|-e^{\alpha_jt+\gamma_jt} u^0_j\otimes \langle \X^2_j|\)d\xi
=:\sum^4_{j=1} U_j.
$$
For brevity, we only give the estimation on the term $ U_{1}$.
By Theorem \ref{eigen_4}, each $\beta_j(\xi)$ is analytic for $\xi\in D_{r_1,\frac{\nu_0}2}$ and satisfies
$$
\beta_j(\xi)=\alpha_j(\xi)+\gamma_j(\xi)+\eta_j(\xi),\quad j=1,2,3,4,
$$
where  $\eta_j(\xi)=O(|\xi|^{-2}\ln^2|\xi|)$ is analytic for $\xi\in D_{r_1,\frac{\nu_0}2}$.

Assume that $x-t\ge 0$. Since $ e^{-\beta_1(\xi)t} u_1(\xi)\otimes \langle \Z^2_1(\xi)|$ and $e^{\alpha_1t+\gamma_1t} u^0_1(\xi)\otimes \langle \X^2_1|$ are analytic for $\xi\in D_{r_1,\frac{\nu_0}2}$, we obtain by Cauchy integral theorem that
\bma
\dx^{-\alpha} U_1(t,x)
&=C\int_{|\xi|\ge r_1}e^{ix\xi} (i\xi)^{-\alpha}\(-e^{\beta_1 t} u_1 \otimes \langle \Z^2_1 |-e^{\alpha_1t+\gamma_1t} u^0_1 \otimes \langle \X^2_1|\)d\xi\nnm\\
&=C\(\int_{\Gamma_1\cup\Gamma_2}+\int_{\Gamma_3\cup\Gamma_4}\)e^{ i (x-t)\xi} (i\xi)^{-\alpha}\nnm\\
&\quad\times \(-e^{\gamma_1(\xi)t+\eta_1(\xi)t} u_1(\xi)\otimes \langle \Z^2_1(\xi)|-e^{ \gamma_1(\xi)t} u^0_1(\xi)\otimes \langle \X^2_1|\)d\xi\nnm\\
&=:J_1+J_2, \label{H_1c}
\ema
where
 \be \label{gamma1}
\left\{\bln
\Gamma_1=&\left\{\xi\,|\,\xi=r_1+i y,\,0\le y \le \frac{\nu_0}2\right\},\\
 \Gamma_2=&\left\{\xi\,|\,\xi=-r_1+i y,\,0\le y \le \frac{\nu_0}2\right\},\\
 \Gamma_3=&\left\{\xi\,|\,\xi=u+i \frac{\nu_0}2,\,r_1\le u <+\infty\right\},\\
 \Gamma_4=&\left\{\xi\,|\,\xi=u+i \frac{\nu_0}2,\,-\infty< u \le-r_1\right\}.
\eln\right.
\ee

By \eqref{eigenvector} and Theorem \ref{eigen_4}, we decompose
\bma
u_1(\xi)&= iu^0_1(\xi)+w_1(\xi) , \label{phi}\\
\Z^2_1(\xi)&= i\X^2_1+i\frac12 d_1(\xi)\X^2_1+\V^2_1(\xi) , \label{yj}
\ema
where 
$$
\left\{\bln
w_1(\xi)&=\(c_1(\xi)-i \)u^0_1(\xi)+u_1(\xi)-c_1(\xi)u^0_1(\xi),\\
 \V^2_1(\xi)&=\Big(c_1(\xi)-i -i\frac12 d_1(\xi)\Big)\X^2_1.
\eln\right.
$$
Moreover, it holds for $\xi\in D_{r_1,\frac{\nu_0}{2}}$ that
\be \label{rj}
\left\{\bln
\|u^0_1(\xi)\|&\le C|\xi|^{-\frac12},\quad \|w_1(\xi)\|\le C|\xi|^{-\frac32} \ln |\xi| , \\
|d_1(\xi)|&\le C|\xi|^{-1},\quad |\V^2_1(\xi)| \le C|\xi|^{-2} \ln^2 |\xi| .
\eln\right.
\ee
Thus, by \eqref{H_1c}--\eqref{rj} and \eqref{specr1} we obtain
\be
\|J_1(t,x)\|\le C\int^{\frac{\nu_0}2}_{0}e^{-y|x-t|}e^{ -\frac{C_1t}{|r_1+iy|}} |r_1+iy|^{-\frac12-\alpha}  dy \le Ce^{-t/D}. \label{J_1}
\ee

In addition by  \eqref{phi} and \eqref{yj}, we have
\bma
&-e^{\gamma_1(\xi)t+\eta_1(\xi)t} u_1(\xi)\otimes \langle \Z^2_1(\xi)|-e^{ \gamma_1(\xi)t} u^0_1(\xi)\otimes \langle \X^2_1|\nnm\\
=& e^{ \gamma_1(\xi)t}(e^{\eta_1(\xi)t}-1)h_1(\xi)+e^{ \gamma_1(\xi)t+\eta_1(\xi)t}h_2(\xi), \label{decom1}
\ema
where
$$
\left\{\bln
h_1(\xi)=&u^0_1(\xi)\otimes \langle \X^2_1|, \\
h_2(\xi)= & \frac12 d_1(\xi)u^0_1(\xi)\otimes \langle \X^2_1|-iu^0_1(\xi)\otimes \langle \V^2_1(\xi)|-w_1(\xi)\otimes \langle \Z^2_1(\xi)|.
\eln\right.
$$
It follows from \eqref{rj} that
\be \|h_1(\xi)\|\le C  |\xi|^{-\frac12}, \quad \|h_2(\xi)\|\le C  |\xi|^{-\frac32}\ln |\xi| . \label{h5a}\ee

Thus
\bma
J_2(t,x)&= e^{ -\frac{\nu_0}2 |x-t|}\int_{|u|\ge r_1}e^{ i (x-t)u} (i\xi)^{-\alpha}e^{ \gamma_1(\xi)t}(e^{\eta_1(\xi)t}-1)h_1(\xi)du\nnm\\
&\quad+e^{ -\frac{\nu_0}2 |x-t|}\int_{|u|\ge r_1}e^{ i (x-t)u} (i\xi)^{-\alpha}e^{ \gamma_1(\xi)t+\eta_1(\xi)t}h_2(\xi)du\nnm\\
&=:e^{ -\frac{\nu_0}2|x-t|}(J_{21}+J_{22}),\quad \xi=u+i\frac{\nu_0}2. \label{J_2}
\ema
Since
\be
|e^{ \gamma_1(\xi)t}(e^{\eta_1(\xi)t}-1)|\le  e^{{\rm Re} \gamma_1(\xi)t} e^{|\eta_1(\xi)|t}|\eta_1(\xi)|t\le Ce^{-\frac{C_1t}{|\xi|}} t|\xi|^{-2}\ln^2|\xi|, \label{rrr1}
\ee
it follows from \eqref{h5a} and \eqref{rrr1} that
\bma
\|J_{21}\|&\le C\int^\infty_{ r_1} tu^{-\frac52-\alpha}\ln^2 u e^{-\frac{C_1t}{u}}du\le C(1+t)^{-\alpha-\frac12}\ln^2(2+t), \label{J_3}\\
\|J_{22}\|&\le C\int^\infty_{ r_1} u^{-\frac32-\alpha}\ln u e^{-\frac{C_1t}{u}}du\le C(1+t)^{-\alpha-\frac12}\ln(2+t).\label{J_4}
\ema
By combining \eqref{H_1c}, \eqref{J_1} and \eqref{J_2}--\eqref{J_4}, we obtain
\be
\|\dx^{-\alpha} U_1(t,x)\|\le Ce^{-\frac{t}{D}}+C(1+t)^{-\alpha-\frac12}\ln^2(2+t)e^{-\frac{\nu_0}2|x-t|}.
\ee
Thus, we can prove 
$$
\left\{\bln
&\|\dx^{-\alpha} I^{11}_1 \|
\le Ce^{-\frac{t}{D}}+C(1+t)^{-\alpha-1}\ln^2(2+t)e^{-\frac{\nu_0}2|x\pm t|},\\
&\|\dx^{-\alpha} I^{1l}_1 \|,\|\dx^{-\alpha} I^{l1}_1\|\le Ce^{-\frac{t}{D}}+C(1+t)^{-\alpha-\frac12}\ln^2(2+t)e^{-\frac{\nu_0}2|x\pm t|},
\eln\right.
$$
for $l=2,3$.

Finally, we estimate $I^{22}_1$ as follows. It holds that
$$
 I^{22}_1=\sum^4_{j=1} C\int_{|\xi|\ge r_1}e^{ix\xi} \(e^{\beta_j t} \Z^2_j \otimes \langle \Z^2_j |-e^{\alpha_jt+\gamma_jt} (1+d_j)\X^2_j\otimes \langle \X^2_j|\)
=:\sum^4_{j=1} V_j.
$$
For brevity, we only give the estimation on  the term $ V_{1}$. Assume that $x-t\ge 0$. Since $ e^{-\beta_1(\xi)t} \Z^2_1(\xi)\otimes \langle \Z^2_1(\xi)|$ and $e^{\alpha_1t+\gamma_1t} (1+d_1 )\X^2_1\otimes \langle \X^2_1|$ are analytic for $\xi\in D_{r_1,\frac{\nu_0}2}$, it follows that
\bma
\dx^{-\alpha} V_1(t,x)
&=C\int_{|\xi|\ge r_1}e^{ix\xi} (i\xi)^{-\alpha}\(e^{\beta_1t} \Z^2_1 \otimes \langle \Z^2_1 |-e^{\alpha_1t+\gamma_1t} (1+d_1)\X^2_1\otimes \langle \X^2_1|\)d\xi\nnm\\
&=C\(\int_{\Gamma_1\cup\Gamma_2}+\int_{\Gamma_3\cup\Gamma_4}\)e^{ i (x-t)\xi} (i\xi)^{-\alpha}\nnm\\
&\quad\times \(e^{\gamma_1(\xi)t+\eta_1(\xi)t} \Z^2_1(\xi)\otimes \langle \Z^2_1(\xi)|-e^{ \gamma_1(\xi)t}(1+d_1)\X^2_1\otimes \langle \X^2_1|\)d\xi\nnm\\
&=:J_3+J_4, \label{U_3}
\ema
where $\Gamma_j$, $j=1,2,3,4$ are defined by \eqref{gamma1}.
For $J_3$, it holds that
\be
\|J_3(t,x)\|\le C\int^{\frac{\nu_0}2}_{0}e^{-y|x-t|}e^{ -\frac{C_1t}{|r_1+iy|}} |r_1+iy|^{-\alpha}  dy \le Ce^{-t/D}. \label{J_5}
\ee
By  \eqref{yj}, we have
\bma
&-e^{\gamma_1(\xi)t+\eta_1(\xi)t} \Z^2_1(\xi)\otimes \langle \Z^2_1(\xi)|-e^{ \gamma_1(\xi)t}(1+d_1(\xi))\X^2_1\otimes \langle \X^2_1|\nnm\\
=& e^{ \gamma_1(\xi)t}(e^{\eta_1(\xi)t}-1)h_3(\xi)+e^{ \gamma_1(\xi)t+\eta_1(\xi)t}h_4(\xi), \label{decom3}
\ema
where
$$
\left\{\bln
h_3(\xi)=&(1+d_1(\xi))\X^2_1\otimes \langle \X^2_1|, \\
h_4(\xi)= & \frac14d^2_1(\xi)\X^2_1\otimes \langle \X^2_1|-i\Big(1+\frac12d_1(\xi)\Big)\X^2_1\otimes \langle \V^2_1(\xi)|-\V^2_1(\xi)\otimes \langle \Z^2_1(\xi)|.
\eln\right.
$$

Thus
\bma
J_4=& e^{ -\frac{\nu_0}2 |x-t|}\int_{|u|\ge r_1}e^{ i (x-t)u} (i\xi)^{-\alpha}e^{ \gamma_1(\xi)t}(e^{\eta_1(\xi)t}-1)h_3(\xi)du\nnm\\
&+e^{ -\frac{\nu_0}2 |x-t|}\int_{|u|\ge r_1}e^{ i (x-t)u} (i\xi)^{-\alpha}e^{ \gamma_1(\xi)t+\eta_1(\xi)t}h_4(\xi)du\nnm\\
=&e^{ -\frac{\nu_0}2|x-t|}(J_{41}+J_{42}),\quad \xi=u+i\frac{\nu_0}2. \label{J_6}
\ema
By \eqref{rj}, we have
\be
\|h_3(\xi)\|\le C  , \quad  \|h_4(\xi)\|\le C |\xi|^{-2} \ln^2 |\xi| ,\quad \xi\in D_{r_1,\frac{\nu_0}{2}}. \label{h6a}
\ee
This together with \eqref{specr1}, \eqref{h6a}  and \eqref{rrr1} gives
\bma
\|J_{41}\|&\le C\int^\infty_{ r_1}tu^{-2-\alpha}\ln^2 u e^{-\frac{C_1t}{u}}du\le C(1+t)^{-\alpha}\ln^2(2+t),\label{J_7}\\
\|J_{42}\|&\le C\int^\infty_{ r_1}u^{-2-\alpha}\ln^2 ue^{-\frac{C_1t}{u}}du\le C(1+t)^{-1-\alpha}\ln^2(2+t).\label{J_8}
\ema
By combining \eqref{U_3}, \eqref{J_5} and \eqref{J_6}--\eqref{J_8}, we obtain
\be
|\dx^{-\alpha} V_1(t,x)|\le Ce^{-\frac{t}{D}}+C(1+t)^{-\alpha}\ln^2(2+t)e^{-\frac{\nu_0}2|x-t|}.
\ee
Thus, we can prove 
$$
\|\dx^{-\alpha} I^{kl}_1(t,x)\|\le Ce^{-\frac{t}{D}}+C(1+t)^{-\alpha}\ln^2(2+t)e^{-\frac{\nu_0}2|x\pm t|},
$$
for $k,l=2,3.$ The proof of the lemma is completed.
\end{proof}

Then, we can show the pointwise behavior of the singular leading short wave $G_2(t,x)$  stated as follows.

\begin{lem}\label{gh1}
Let $G_2(t,x)$ be defined by \eqref{G_b}.
For any integer $\alpha\ge 0 $, we have the following decomposition
\be
 (\nu_0+\dx)^{-\alpha}G_{2}(t,x)= \Lambda_{\alpha}(t,x)+\Omega_{\alpha}(t,x), \label{gl_5}
\ee
 where $\Lambda_{\alpha}(t,x) $ is supported in $|x\pm t|> 1$ with its elements satisfying
  \be \label{F_a}
  \left\{\bln
 & \|\Lambda^{11}_{\alpha}(t,x)\|\le C\sum_{l=\pm1}(1+t)^{-\alpha-1}\ln^2(2+t)e^{-\frac{\nu_0|x-lt|}2} ,\\
 & \|\Lambda^{1j}_{\alpha}(t,x)\|,\|\Lambda^{j1}_{\alpha}(t,x)\|\le C\sum_{l=\pm1}(1+t)^{-\alpha-\frac12}\ln^2(2+t)e^{-\frac{\nu_0|x-lt|}2} , \\
 & \|\Lambda^{2j}_{\alpha}(t,x)\|,\|\Lambda^{3j}_{\alpha}(t,x)\|\le C\sum_{l=\pm1}(1+t)^{-\alpha}\ln^2(2+t)e^{-\frac{\nu_0|x-lt|}2},
 \eln\right.
 \ee
and $\Omega_{\alpha}(t,x) $ is supported in $|x\pm t|\le 1$ with its elements satisfying
  \be \label{H_a}
  \left\{\bln
 &\intra \|\Omega^{11}_{\alpha}(t,x) \|dx\le C(1+t)^{-\alpha-1}\ln^2(2+t) ,\\
 &\intra \|\Omega^{1j}_{\alpha}(t,x) \|dx,\intra \|\Omega^{j1}_{\alpha}(t,x) \|dx\le C(1+t)^{-\alpha-\frac12}\ln^2(2+t),\\
 &\intra \|\Omega^{2j}_{\alpha}(t,x) \|dx,\intra \|\Omega^{3j}_{\alpha}(t,x) \|dx\le C(1+t)^{-\alpha}\ln^2(2+t),
\eln\right.
 \ee for $j=2,3.$  In particular, for $\alpha=0$,  $\Omega_{0}(t,x)$ has the singularity at $|x\pm t|=0$ and satisfies for $|x \pm t|\le 1$ that
   \be \label{H_b}
  \left\{\bln
 & \|\Omega^{11}_{0}(t,x)\| \le C,\\
 & \|\Omega^{1j}_{0}(t,x) \| , \|\Omega^{j1}_{0}(t,x) \| \le C \frac1{\sqrt{|x\pm t|}}\ln \frac1{|x\pm t|},\,\,\, j=2,3,\\
 & \|\Omega^{2j}_{0}(t,x) \| , \|\Omega^{3j}_{0}(t,x) \| \le \delta(x\pm t)+C\ln \frac1{|x\pm t|},\,\,\, j=2,3.
\eln\right.
\ee 
\end{lem}
\begin{proof}
First, we estimate $G_{2}^{12}(t,x)$. Recalling \eqref{G_b}, we have
\bma
 &(\nu_0+\dx)^{-\alpha}G_{2}^{12}(t,x)\nnm\\
 =&\sum^4_{j=1}C\int^{\infty}_{-\infty}e^{ix\xi}(\nu_0+i\xi)^{-\alpha} e^{\alpha_j(\xi)t+\gamma_j(\xi) t} u^0_j(\xi)\otimes \langle \X^2_j|d\xi=:\sum^4_{j=1}U_{\alpha}^j. \label{U_1a}
\ema
For brevity, we give only the estimate on the term $ U_{\alpha}^1$. Assume that $x-t\ge 0$. Since $ e^{-\gamma_1(\xi)t} u^0_1(\xi)\otimes \langle \X^2_1|$ is analytic for $\xi\in D_{\frac{\nu_0}2}$, by Cauchy integral theorem we obtain
\bma
U_{\alpha}^1(t,x) &=C\int^{\infty}_{-\infty}e^{ix\xi} e^{-i\xi t+\gamma_1(\xi)t}  (\nu_0+i\xi)^{-\alpha}u^0_1(\xi)\otimes \langle \X^2_1|d\xi\nnm\\
&=Ce^{ -\frac{\nu_0}2|x-t|}\int^{\infty}_{-\infty}e^{ i (x-t)u} e^{\gamma_1(z)t} (\nu_0+iz)^{-\alpha}u^0_1(z)\otimes \langle \X^2_1|du,  \label{I_2a}
\ema
where $z=u+i\frac{\nu_0}2$. We estimate 
\eqref{I_2a} as follows.
Let
\be
I_1(t,x)=\int^{\infty}_{-\infty}e^{i(x-t)u}g_{\alpha}(z) h_1(z)  du, \quad z=u+i\frac{\nu_0}2, \label{I11}
\ee
where
$$
g_{\alpha}(z)= e^{ \gamma_1(z)t}(\nu_0+i z)^{-\alpha} , \quad h_1(z)=u^0_1(z)\otimes \langle \X^2_1|.
$$
By change of variables $(v_1-1)|u|\to w_1$, we obtain that for $z=u+i\frac{\nu_0}2$ and $|u|\ge 2$,
\bma
|\gamma'_1(z)|&\le  C\intr \frac{|v_1- 1|}{\nu_0^2+(v_1- 1)^2u^2}v^2_2e^{-\frac{|v|^2}2}dv\nnm\\
&\le \frac{C}{|u|^2}\int^{|u|}_0 \frac{w_1}{\nu_0^2+w_1^2}dw_1+\frac{ C}{|u|^3}\int^\infty_{|u|} e^{-\frac12\frac{w^2_1}{|u|^2}}dw_1
\le C\frac{\ln |z|}{|z|^2}, \label{h6}\\
\|(u^0_1)'(z)\|^2&\le  C\intr \frac{|v_1- 1|^2}{[\nu_0^2+(v_1- 1)^2u^2]^2}v^2_2e^{-\frac{|v|^2}2}dv\nnm\\
&\le \frac{C}{|u|^3}\int^{\infty}_0 \frac{w_1^2}{(\nu_0^2+w_1^2)^2}dw_1\le C\frac{1}{|z|^3}. \label{h7}
\ema
Thus, for $z=u+i\frac{\nu_0}2$,
\bma
 &| g'_{\alpha}(z)|\le C\(\frac{t\ln(2+|z|)}{(1+|z|)^{2+\alpha}}+\frac1{(1+|z|)^{1+\alpha}}\) e^{-\frac{C_1t}{1+|z|}},\label{fa} \\
&\|h'_1(z)\|\le C\frac{1}{(1+|z|)^{3/2}}. \label{ha}
\ema

We estimate $I_1$ by considering two cases: $|x-t|\le 1$ and $|x-t|>1$. For $|x-t|\le 1$, we take $r_2= \frac{1}{|x-t|}\ge 1$ and divide $I_1$ into
\bmas
I_1(t,x)=&\(\int_{ |u|\ge r_2}+\int_{ |u|\le r_2}\)e^{i(x-t)u}g_{\alpha}(z) h_1(z)du \\
=& \frac1{i(x-t)}g_{\alpha}\(\pm r_2+i\frac{\nu_0}{2}\) h_1\(\pm r_2+i\frac{\nu_0}{2}\) e^{\pm i(x-t)r_2}\\
&-\frac1{i(x-t)}\int_{ |u|\ge r_2}e^{i(x-t)u}\[(g'_{\alpha} h_1)(z)+(g_{\alpha} h'_1)(z)\]du\\
&+\int_{ |u|\le r_2}e^{i(x-t)u}g_{\alpha}(z) h_1(z)du, \quad z=u+i\frac{\nu_0}2.
\emas
This together with \eqref{h5a}, \eqref{fa}--\eqref{ha} implies that
\bma
\|I_1(t,x)\|&\le C|x-t|^{\alpha-\frac12}e^{-\frac{C_1|x-t|t}{2}}+C \int^{ \frac{1}{|x-t|}}_{0}(1+u)^{-\frac12-\alpha} e^{-\frac{C_1t}{1+u}}du \nnm\\
&\quad+C\frac{1}{|x-t|}\int^{\infty}_{ \frac{1}{|x-t|}}(1+u)^{-\frac32-\alpha}\ln (2+u)e^{-\frac{C_1t}{1+u}}du \nnm\\
&=:J_1+J_2+J_3. \label{ii-1}
\ema
The right hand side terms of \eqref{ii-1} can be estimated as follows. By change of  variables $x-t\to y$, we have that for $\alpha\ge 0$,
\bma
\int_{|x-t|\le 1} J_1 dx&\le C\int^1_{0} y^{\alpha-\frac12} e^{ -\frac{yt}{C}}dy \nnm\\
&=C(1+t)^{-\alpha-\frac12}\int^t_0  z^{\alpha-\frac12} e^{ -\frac{z}{C}}dz\le  C(1+t)^{-\alpha-\frac12},\label{J11}\\
\int_{|x-t|\le 1} J_2 dx&\le C\int^1_{0} dy\int^{\frac{1}{y}}_{0}(1+u)^{-\frac12-\alpha} e^{-\frac{C_1t}{1+u}}du\nnm\\
&=C\(\int^\infty_{1} du\int^{\frac{1}u}_0 dy+\int^1_{0}  du\int^{1}_0dy\) (1+u)^{-\frac12-\alpha} e^{-\frac{C_1t}{1+u}} \nnm\\
&\le C\int^\infty_{1}u^{-\frac32-\alpha} e^{-\frac{C_1t}u}du+Ce^{-\frac{C_1t}2}\le C(1+t)^{-\alpha-\frac12},\label{J2a}
\ema
and
\bma
\int_{|x-t|\le 1} J_3 dx&\le C\int^1_{0} \frac{1}{y}dy \int^{\infty}_{ \frac{1}{y}}u^{-\frac32-\alpha}\ln(2+u)e^{-\frac{C_1t}u}du\nnm\\
&=C\int^\infty_{1} u^{-\frac32-\alpha}\ln (2+u) e^{-\frac{C_1t}u}du \int^{1}_{\frac{1}u}\frac{1}{y} dy\nnm\\
&\le C(1+t)^{-\alpha-\frac12}\ln^2(2+t). \label{J31}
\ema
Thus
\be
\int_{|x-t|\le 1} \|I_1(t,x)\| dx\le C(1+t)^{-\alpha-\frac12}\ln^2(2+t). \label{I31}
\ee
In particular, for $\alpha=0$, we obtain the pointwise singularity of $I_1$ at $|x\pm t|=0$ as
\bmas
\|I_1(t,x)\|&\le C\frac1{|x-t|^{\frac12}}e^{-\frac{C_1|x-t|t}{2}}+C \int^{ \frac{1}{|x-t|}}_{0}(1+u)^{-\frac12} du\\
&\quad+C\frac{1}{|x-t|}\int^{\infty}_{ \frac{1}{|x-t|}}(1+u)^{-\frac32}\ln (2+u) du\\
&\le C\frac1{|x-t|^{\frac12}}\ln \frac1{|x-t|}, \quad |x-t|\le 1.
\emas

For $|x-t|> 1$,  we have
$$
I_1(t,x)=-\frac1{i(x-t)}\int^{\infty}_{-\infty}e^{i(x-t)u}\[(g'_{\alpha} h_1)+(g_{\alpha} h'_1)\](z)du,
$$
which gives rise to
\bma
\|I_1(t,x)\|&\le C \int^{\infty}_{ 0}(1+u)^{-\frac32-\alpha}\ln (2+u)e^{-\frac{C_1t}{1+u}}du \nnm\\
&\le C(1+t)^{-\alpha-\frac12}\ln(2+t),\quad |x-t|> 1. \label{I32}
\ema
By combining \eqref{I_2a}, \eqref{I31} and \eqref{I32}, we  obtain
$$
\left\{\bln
 &(\nu_0+\dx)^{-\alpha}G^{12}_{1}(t,x)= \Lambda^{12}_{\alpha}(t,x)+\Omega^{12}_{\alpha}(t,x) , \quad \alpha\ge 0,\\
 &\Lambda^{12}_{\alpha}=(U_{\alpha}^1+U_{\alpha}^2)1_{\{|x-t|\le 1\}}+(U_{\alpha}^3+U_{\alpha}^4)1_{\{|x+t|\le 1\}},\\
 &\Omega^{12}_{\alpha}=(U_{\alpha}^1+U_{\alpha}^2)1_{\{|x-t|> 1\}}+(U_{\alpha}^3+U_{\alpha}^4)1_{\{|x+t|> 1\}},
 \eln\right.
$$
 where $U_{\alpha}^j$, $j=1,2,3,4$ are defined by \eqref{U_1a}, and $\Lambda^{12}_{\alpha}$ and $\Omega^{12}_{\alpha}$ satisfy \eqref{F_a}, \eqref{H_a} and \eqref{H_b}. Similarly, we can prove that $G^{1l}(t,x)$ and $G^{l1}(t,x)$, $l=1,2,3$ satisfy \eqref{gl_5}--\eqref{H_b}.

Finally, we consider $G_{2}^{22}(t,x)$. Recalling \eqref{G_b}, we have
\bma
&\quad(\nu_0+\dx)^{-\alpha}G_{2}^{22}(t,x)\nnm\\
&=\sum^4_{j=1}C\int^{\infty}_{-\infty}e^{ix\xi} e^{\alpha_j(\xi)t+\gamma_j(\xi)t}(\nu_0+i\xi)^{-\alpha}(1+d_j(\xi))\X^2_j\otimes \langle \X^2_j|d\xi=:\sum^4_{j=1}V_{\alpha}^j. \label{V_a}
\ema
For brevity, we only estimate $V_{\alpha}^1$ as follows. Assume that $x-t\ge 0$. For $\alpha\ge 1$,
\bmas
V_{\alpha}^1(t,x)&=C\int^{\infty}_{-\infty}e^{ i (x-t)\xi} e^{\gamma_1(\xi)t}(\nu_0+i\xi)^{-\alpha}(1+d_1(\xi))\X^2_1\otimes \langle \X^2_1|d\xi\\
&=Ce^{ -\frac{\nu_0}2|x-t|}\int^{\infty}_{-\infty}e^{ i (x-t)u} e^{\gamma_1(z)t} (\nu_0+iz)^{-\alpha}(1+d_1(z))\X^2_1\otimes \langle \X^2_1|du,
\emas
where $z=u+i\frac{\nu_0}2$.
Let
$$
I_2= \int^{\infty}_{-\infty}e^{i(x-t)u}g_{\alpha}(z) (1+d_1(z))\X^2_1\otimes \langle \X^2_1|du,\quad z=u+i\frac{\nu_0}2.
$$
By change of variables $(v_1-1)|u|\to w_1$, we obtain that for $z=u+i\frac{\nu_0}2$ and $|u|\ge 2$,
\bma
|d'_1(z)|\le &C\intr \frac{|v_1- 1|}{ \nu_0^3+|v_1- 1|^3|u|^3}v^2_2e^{-\frac{|v|^2}2}dv\nnm\\
\le& \frac{C}{|u|^2}\int^{\infty}_0 \frac{w_1}{\nu_0^3+w_1^3}dw_1
\le C\frac{1}{|z|^2}. \label{d_1}
\ema
By \eqref{fa}, \eqref{d_1} and similar arguments as proving \eqref{I31} and \eqref{I32}, we can obtain that for $\alpha\ge 1$,
\bmas
\intra \|I_{2}(t,x)1_{\{|x-t|\le 1\}}\|dx&\le C(1+t)^{-\alpha}\ln^2(2+t), 
\\
\|I_{2}(t,x)1_{\{|x-t|> 1\}}\|&\le  C(1+t)^{-\alpha}\ln(2+t) . 
\emas
For $\alpha=0$, we obtain by Cauchy integral theorem that
\bmas
V_0^1(t,x)
&=C\int^{\infty}_{-\infty}e^{ i (x-t)\xi}\(1+ (e^{\gamma_1(\xi)t}-1)+e^{\gamma_1(\xi)t}d_1(\xi)\)\X^2_1\otimes \langle \X^2_1|d\xi\\
&=\bigg(\delta(x-t) +Ce^{ -\frac{\nu_0}2|x-t|}\int^{\infty}_{-\infty}e^{ i (x-t)u}(e^{\gamma_1(z)t}-1)du  \\
&\quad+Ce^{ -\frac{\nu_0}2|x-t|}\int^{\infty}_{-\infty}e^{ i (x-t)u}e^{\gamma_1(z)t}d_1(z)du\bigg) \X^2_1\otimes \langle \X^2_1|,
\emas
where $z=u+i\frac{\nu_0}2$. Let
$$
I_3=\int^{\infty}_{-\infty}e^{ i (x-t)u}(e^{\gamma_1(z)t}-1)du,\quad
I_4=\int^{\infty}_{-\infty}e^{ i (x-t)u}e^{\gamma_1(z)t}d_1(z)du.
$$

By \eqref{d_1} and similar arguments as proving \eqref{I31} and \eqref{I32}, we can obtain 
\bmas
\intra |I_{4}(t,x)1_{\{|x-t|\le 1\}}|dx&\le C(1+t)^{-1}\ln^2(2+t),  \\
|I_{4}(t,x)1_{\{|x-t|> 1\}}|&\le  C(1+t)^{-1}\ln(2+t) .
\emas

 We estimate $I_3(t,x)$ as follows.
For $|x-t|\le 1$, we take $r_2= \frac{1}{|x-t|}\ge 1$ and divide $I_3$ into
\bmas
I_3(t,x)=&\(\int_{ |u|\ge r_2}+\int_{ |u|\le r_2}\)e^{i(x-t)u}(e^{\gamma_1(z)t}-1) du \\
=& \frac1{i(x-t)} \(e^{\gamma_1(\pm r_2+i\frac{\nu_0}{2})t}-1\)e^{\pm i(x-t)r_2}  \\
&-\frac1{i(x-t)}\int_{ |u|\ge r_2}e^{i(x-t)u} e^{\gamma_1(z)t}\gamma'_1(z)t du\\
&+\int_{ |u|\le r_2}e^{i(x-t)u}(e^{\gamma_1(z)t}-1)du, \quad z=u+i\frac{\nu_0}2.
\emas
Since
\bmas
e^{\gamma_1(z)t}-1=&(e^{{\rm Re}\gamma_1(z)t}-1)+ie^{{\rm Re}\gamma_1(z)t}\sin({\rm Im}\gamma_1(z)t)\\
&+e^{{\rm Re}\gamma_1(z)t}\(\cos({\rm Im}\gamma_1(z)t)-1\)\\
\le &\lt|e^{-\frac{C_1t}{1+|z|}} -1\rt|+Ce^{-\frac{C_1t}{1+|z|}} \(\frac{\ln (2+|z|)}{1+|z|}t+\frac{\ln^2 (2+|z|)}{(1+|z|)^2}t^2\),
\emas
it follows from \eqref{h6} that for $|x-t|\le 1$,
\bma
|I_3(t,x)|&\le C\frac1{|x-t|}\(1-e^{-\frac{C_1|x-t|t}{2}}\)+C \int^{ \frac{1}{|x-t|}}_{0}\(1- e^{-\frac{C_1t}{1+u}}\)du \nnm\\
&\quad+Ce^{-\frac{C_1|x-t|t}{2}}\(t\ln\frac1{|x-t|}+t^2|x-t|\ln^2\frac1{|x-t|}\) \nnm\\
&\quad+C \int^{ \frac{1}{|x-t|}}_{0} e^{-\frac{C_1t}{1+u}}\(\frac{\ln (2+u)}{1+u}t+\frac{\ln^2 (2+u)}{(1+u)^2}t^2\)du \nnm\\
&\quad+C\frac{1}{|x-t|}\int^{\infty}_{ \frac{1}{|x-t|}}t(1+u)^{-2}\ln(2+u) e^{-\frac{C_1t}{1+u}}du \nnm\\
&=:J_4+J_5+J_6+J_7+J_8. \label{ii-2}
\ema
The right hand side terms  of \eqref{ii-2} can be estimated as follows. By change of  variables $x-t\to y$, we have
\bmas
\int_{|x-t|\le 1} J_4 dx&\le C\int^1_{0} y^{-1}\(1- e^{ -\frac{C_1yt}{2}}\)dy\\
&=C\int^t_{0} z^{-1}\(1- e^{ -\frac{C_1z}{2}}\)dz\le C\ln(2+t),\\
\int_{|x-t|\le 1} J_5 dx&\le C\int^1_{0} dy\int^{\frac{1}{y}}_{0}\(1- e^{-\frac{C_1t}{1+u}}\)du\\
&=C\(\int^\infty_{1} du\int^{\frac{1}u}_0 dy+\int^1_{0}  du\int^{1}_0dy\) \(1- e^{-\frac{C_1t}{1+u}}\) \\
&\le C\int^\infty_{1}u^{-1}\(1- e^{-\frac{C_1t}{u}}\)du+C \le C \ln(2+t).
\emas
By a similar argument as \eqref{J11}--\eqref{J31}, we obtain
$$
\int_{|x-t|\le 1} (J_6+J_7+J_8)dx\le C\ln^2(2+t).
$$
Thus
\be
\int_{|x-t|\le 1} |I_3(t,x)| dx\le C \ln^2(2+t).  \label{I42}
\ee
For $|x-t|> 1$,  we have
$$
I_3(t,x)=-\frac1{i(x-t)}\int^{\infty}_{-\infty}e^{i(x-t)u} e^{\gamma_1(z)t}\gamma'_1(z)t du,
$$
which gives
\bma
|I_3(t,x)|&\le C \int^{\infty}_{ 0}t(1+u)^{-2}\ln (2+u)e^{-\frac{C_1t}{1+u}}du \nnm\\
&\le C \ln(2+t),\quad |x-t|> 1.  \label{I43}
\ema
Thus, we obtain the estimate of $G^{22}_{1}(t,x)$ as listed in \eqref{gl_5}, \eqref{F_a} and \eqref{H_a}.
 Similarly, we can prove that $G^{kl}(t,x)$, $k,l=2,3$ satisfy \eqref{gl_5}--\eqref{H_b}.
The proof of the lemma is completed.
\end{proof}

\section{Green's function:  kinetic part}
\setcounter{equation}{0}
\label{kinetic}
In this section, we introduce a  Picard's iteration to construct the singular waves and and study the remaining part of  Green's function by the weighted energy estimate.

Let $f$ be the $1^{th}$ line of $G$, and $E=(E_1,E_2,E_3)^T ,B=(B_1,B_2,B_3)^T $ 
with $E_k,B_k$, $k=1,2,3$ be the $(k+1)^{th}$ and $(k+4)^{th}$ line of $G$, respectively.
Then the Green's function $\hat G$ can be rewritten as
$$\hat G=(\hat{f},\hat{E}, \hat{B})^T=(\hat{f},\hat{E}_1,\hat{E}_r,0, \hat{B}_r)^T , \quad \hat{E}_r=(\hat{E}_2,\hat{E}_3)^T,\,\, \hat{B}_r=(\hat{B}_2,\hat{B}_3)^T .$$
By \eqref{LVMB-3}, $(\hat{f},\hat{E}_1,\hat{E}_r,0, \hat{B}_r)$ satisfies
\be\label{l2}
\left\{\bln
&\dt \hat f+i v_1\xi \hat f-L_1\hat f-(v_1 \hat E_1+\bar{v} \cdot \hat E_r)\chi_0=0, \\
 &\dt \hat{E}_1= -(\hat{f},v_1\chi_0),\quad   i\xi \hat{E}_1=(\hat f,\chi_0),\\
 &\dt \hat{E}_r=i \xi \mathbb{O}_1\hat B_r-(\hat{f},\bar{v}\chi_0),\quad \dt \hat B_r=-i \xi \mathbb{O}_1\hat{E}_r, \\
& \hat{G}(0,\xi)=(\hat{f}(0),\hat{E}_1(0),\hat{E}_r(0),0, \hat{B}_r(0))^T=\hat{G}_{0}(\xi),
\eln\right.
\ee
where $\bar{v}=(v_2,v_3)$, $\O_1$ is defined by \eqref{O_1}, and the initial data $\hat{G}_0(\xi) $ is defined by \eqref{initial-1}.

Define the following two systems for $(\hat g_1,\hat E_r,\hat B_r)^T$ describing the particle transport under the electro-magnetic fields (lorentz force) as
\be \label{l3}
\left\{\bln
&\dt \hat g_1+i v_1\xi \hat g_1-L_1\hat g_1- \bar{v} \cdot \hat E_r \chi_0=0, \\
 &\dt \hat{E}_r=i \xi \mathbb{O}_1\hat B_r-(\hat{g}_1,\bar{v}\chi_0),\quad \dt \hat B_r=-i \xi \mathbb{O}_1\hat{E}_r, \\
&(\hat g_1,\hat E_r,\hat B_r)(0,\xi)=(\hat{f}(0),\hat{E}_r(0),\hat{B}_r(0)),
\eln\right.
\ee
and for $(\hat g_2,\hat E_1)^T$ describing the particle transport under the electric field as
\be \label{l4}
\left\{\bln
&\dt \hat g_2+i v_1\xi \hat g_2-L_1\hat g_2- v_1 \hat E_1 \chi_0=0, \\
 & i\xi \hat{E}_1=(\hat g_2,\chi_0)+(\hat{g}_1,\chi_0),\\
&(\hat g_2,\hat E_1)(0,\xi)=(0,\hat{E}_1(0)).
\eln\right.
\ee

By the fact that $L_1-iv_1\xi$ is invariant under  change of variables $v_l\to  -v_l$ with $l=2,3$, it holds that
\be (e^{t(L_1-iv_1\xi)}v_1\chi_0, v_l\chi_0)=0,\quad l=2,3 . \label{orth}\ee
This together with \eqref{l3}--\eqref{l4} and
$$
\hat{g}_2(t)=\intt e^{(t-s)(L_1-iv_1\xi)}v_1\chi_0\hat{E}_1(s)ds\Longrightarrow (\hat{g}_2 ,\bar v\chi_0)=0,
$$
implies that $(\hat g_1+\hat g_2,\hat E_1,\hat E_r,0,\hat B_r)^T$ satisfies the system \eqref{l2}.

By \eqref{l2}, \eqref{l3} and \eqref{l4}, we define the approximate  sequence $\hat{\U}_n=(\hat{I}_n,\hat{U}_n^{1},\hat{U}_n^{2},0,\hat{U}_n^{3} )^T $ for $\hat{G}=(\hat{f},\hat{E}_1,\hat{E}_r,0, \hat{B}_r)^T$, and $\hat{\U}_n$ can be decomposed into  the following two parts:
\be
\hat{\U}_n=(\hat{H}_n,\hat{U}^{1}_n,0,0,0)^T+(\hat{J}_n,0,\hat{U}_n^{2},0,\hat{U}_n^{3})^T,\quad n\ge 0. \label{U_4}
\ee
Here, $ (\hat{J}_n,\hat{U}_n^{2},\hat{U}_n^{3})^T$, $n\ge 0$ is the the approximate  sequence for $(\hat{g}_1,\hat{E}_r,\hat{E}_r)$ constructed through \eqref{l3} as
\be \label{c-w}
\left\{\bln
&\dt \hat{J}_0+i \xi v_1 \hat{J}_0+\nu(v)\hat{J}_0=0,\\
 & \dt \hat{U}_0^{2}=i \xi \O_1 \hat{U}_0^{3} ,\,\,\,  \dt \hat{U}_0^{3}=-i \xi \O_1\hat{U}_0^{2},  \\
 &(\hat{J}_0,\hat{U}_0^{2},\hat{U}_0^{3})(0)=(\hat{f}(0), \hat{E}_{r}(0), \hat{B}_{r}(0))^T,
  \eln\right.
\ee
and
\be \label{c-w1}
\left\{\bln
&\dt \hat{J}_n+i \xi v_1 \hat{J}_n+\nu(v)\hat{J}_n=K_1\hat{J}_{n-1}+ \bar{v }\cdot \hat{U}_{n-1}^{2}\chi_0,\\
 &\dt \hat{U}_{n}^{2}=i \xi \O_1 \hat{U}_{n}^{3}-(\hat{J}_{n-1},\bar{v}\chi_0),\,\,\,  \dt \hat{U}_{n}^{3}=-i \xi \O_1\hat{U}_{n}^{2},
 \\
 &(\hat{J}_n,\hat{U}_n^{2},\hat{U}_n^{3})(0)=(0,0,0), \quad n\ge 1.
  \eln\right.
\ee
And $(\hat{H}_n,\hat{U}^{1}_n)^T$, $n\ge 0$  is the the approximate  sequence for $(\hat{g}_2,\hat{E}_1 )$  constructed through \eqref{l4} as
\be \label{c-w3}
\left\{\bln
&\dt \hat{H}_0+i \xi v_1 \hat{H}_0+\nu(v)\hat{H}_0=0,\\
& (\nu_0+i\xi)\hat{U}_0^{1}=(H_0+ \hat{J}_0,\chi_0),\\
 &(\hat{H}_0,\hat{U}^{1}_0)(0)=(0,\mbox{$\frac{i\xi  }{\nu_0+i\xi}$}\hat{E}_1(0))^T,
  \eln\right.
\ee
and
\be \label{c-w4}
\left\{\bln
&\dt \hat{H}_n+i \xi v_1 \hat{H}_n+\nu(v)\hat{H}_n=K_1\hat{H}_{n-1}+ v_1\hat{U}_{n-1}^{1} \chi_0,\\
 &(\nu_0+i\xi)\hat{U}_{n}^{1}=( \hat{H}_n+ \hat{J}_n,\chi_0)+\nu_0\hat{U}_{n-1}^{1},\\
 &(\hat{H}_n,\hat{U}^{1}_n)(0)=(0,\mbox{$\frac{i\xi \nu^n_0 }{(\nu_0+i\xi)^{n+1}}$}\hat{E}_1(0))^T, \quad n\ge 1.
  \eln\right.
\ee

Set
\be
\BB_2(\xi) =\left( \ba
 0 & i\xi\O_1\\
-i\xi\O_1 & 0
\ea\right). \label{BB-2}
\ee
We can solve the systems \eqref{c-w}--\eqref{c-w1} and \eqref{c-w3}--\eqref{c-w4} as
\be \label{U_1}
\left\{\bln
&\hat{J}_0(t)=\hat{S}^t\hat{f}(0), \,\,\, (\hat{U}_{0}^{2},\hat{U}_{0}^{3})^T(t)= e^{t\BB_2(\xi)}(\hat{E}_{r} (0),\hat{B}_{r}(0))^T,\\
& \hat{H}_0(t)\equiv0,\,\,\, \hat{U}_{0}^{1}(t)=\frac{1}{\nu_0+i\xi}(\hat{S}^t\hat{f}(0),\chi_0),
  \eln\right.
\ee
and
\be \label{U_2}
\left\{\bln
&\hat{J}_n(t)=\intt \hat{S}^{t-s}K_1\hat{J}_{n-1}ds+\intt \hat{S}^{t-s} \bar{v }\cdot \hat{U}_{n-1}^{2} \chi_0ds,\\
&(\hat{U}_{n}^{2},\hat{U}_{n}^{3})^T(t)=-\intt e^{(t-s)\BB_2(\xi)}(P_{\bar m}\hat{J}_{n-1},0,0)^Tds,\\
&\hat{H}_n(t)=\intt \hat{S}^{t-s}K_1\hat{H}_{n-1}ds+\intt \hat{S}^{t-s} v_1\hat{U}_{n-1}^{1} \chi_0 ds,\\
&\hat{U}_{n}^{1}(t)=\frac{1}{\nu_0+i\xi}(\hat{H}_{n}+\hat{J}_{n},\chi_0)+\frac{\nu_0}{\nu_0+i\xi}\hat{U}_{n-1}^{1}, \,\,\, n\ge 1 ,
  \eln\right.
\ee
where $P_{\bar m}g=(g,\bar{v}\chi_0)$,  and
$\hat{S}^t $ is an operator on $L^2(\R^3_v)$ defined by
$$\hat{S}^t=e^{-(\nu(v)+iv_1\xi)t}.$$

 Define the $j^{th}$ degree mixture operator $\hat{\M}^t_{j}(\xi)$ related to the Boltzmann equation by (cf. \cite{Liu1,Liu2})
\be
\hat{\M}^t_{j}(\xi)=\intt\int^{s_1}_0\cdots\int^{s_{j-1}}_0 \hat{S}^{t-s_1}K_1 \hat{S}^{s_1-s_2}\cdots  \hat{S}^{s_{j-1}-s_{j}}K_1 \hat{S}^{s_{j}}ds_{j}\cdots ds_1, \label{Mk}
\ee
where $\hat{\M}^t_0=\hat{S}^t$ and $\xi\in\mathbb{C}$. 

By the fact that $\hat{\M}_j^{t}(\xi)$ is invariant under change of  variables $v_l\to  -v_l$ with $l=2,3$,  it holds that
\be
( \hat{\M}_j^{t}\bar v\chi_0,\chi_0)=0,\quad ( \hat{\M}_j^{t}v_1\chi_0,\bar v\chi_0)=0,\quad \forall j\ge 0. \label{M1}
\ee
Then, by \eqref{M1} and \eqref{U_2}, we conclude that
\be (\hat H_n(t),\bar v\chi_0)=0,\quad (\hat J_n(t), \chi_0)=(\M^t_n \hat f(0),\chi_0), \label{U_5}\ee
and hence  $(\hat{H}_n,\hat{U}^{1}_n)$ for $n\ge 1$ can be represented as
\be \label{H_n}
\left\{\bln
&\hat{H}_{n}(t)=\sum^{n-1}_{l=0}\intt \hat{\M}^{t-s}_{l} v_1\chi_0\hat{U}^{1}_{n-l-1} ds,\\
&\hat{U}_{n}^{1}(t)=\frac{1}{\nu_0+i\xi}(\hat{H}_{n}+\hat{\M}^t_{n}\hat{f}(0),\chi_0)+\frac{\nu_0}{\nu_0+i\xi}\hat{U}_{n-1}^{1}.
\eln\right.
\ee

Set
\be \label{A12}
 \AA_2(\xi) =\left( \ba
c(\xi) &0_{1\times 4}  \\
0_{4\times 1} & \BB_2(\xi)
\ea\right),\quad
 \AA_3 =\left( \ba
K_1 &K_2\\
-K_3 & 0_{4\times 4}
\ea\right),
\ee
where $c(\xi)=-\nu(v)-iv_1\xi$, $\BB_2(\xi)$ is given by \eqref{BB-2}, and  $K_2:\C^4\to L^2_v$, $K_3:L^2_v\to \C^4$ are defined by
$$
K_2=(\bar v\chi_0,0,0),\quad K_3=(P_{\bar m},0,0)^T.
$$

Define the $j^{th}$ degree mixture operator $\hat{\Q}^t_{j}(\xi)$ related to the VMB system:
\be
\hat{\Q}^t_{j}(\xi)=\intt\int^{s_1}_0\cdots\int^{s_{j-1}}_0 e^{(t-s_1)\AA_2}\AA_3 e^{(s_1-s_2)\AA_2}\cdots  e^{(s_{j-1}-s_{j})\AA_2}\AA_3 e^{s_{j}\AA_2}ds_{j}\cdots ds_1,\label{Q_3}
\ee
where $\hat{\Q}^t_0(\xi)=e^{t\AA_2(\xi)}$ and $\xi\in\mathbb{C}$. Note that  $\hat{\mathbb{Q}}^t_j(\xi) $ is an operator in  $L^2_v\times \C^2\times \C^2$.

Denote $\hat{V}_n(t)=:(\hat{J}_n,\hat{U}_n^{2},\hat{U}_n^{3})^T(t) $.
By \eqref{U_2}, we can represent $\hat{V}_n(t)  $ as
\be \hat{V}_n(t) =\hat{\Q}^t_{n}(\xi)\hat{V}_0(0), \quad t>0, \label{V_5}
\ee where $\hat{V}_0(0)=(\hat{f}(0), \hat{E}_{r}(0), \hat{B}_{r}(0))^T$.

We have the following mixture lemmas in the frequency space for the operators $\hat{\mathbb{M}}^t_{j}(\xi)$ and $\hat{\mathbb{Q}}^t_{j}(\xi)$ 
respectively.

\begin{lem}[Mixture Lemma] \label{mix1a}
For any $j\ge 1$, $\hat{\M}^t_{j}(\xi)$ is analytic for $\xi\in D_{\nu_0}$
and satisfies
\bma
\|\hat{\mathbb{M}}^t_{2}(\xi)g_0\| &\le C(1+|\xi|)^{-1}(1+t)^2e^{-\nu_0t}(\|g_0\|+\|\Tdv g_0\|), \label{mix2}\\
\|\hat{\mathbb{M}}^t_{3j}(\xi)g_0\| &\le C_j(1+|\xi|)^{-j}(1+t)^{3j}e^{-\nu_0t}\|g_0\|, \label{mix3a}
\ema
for any $g_0\in L^2_v$ and  two positive constants $C$ and  $C_j$.
\end{lem}

\begin{proof}
This lemma is a slightly modification of the mixture lemmas in \cite{Liu1,Li4}. The proof is same as Lemma 3.7 in \cite{Li4} by using the fact that
$|e^{-[\nu(v)+iv_1\xi]t}|\le e^{-\nu_0t}$ for $\xi\in D_{\nu_0}.$
Note that the lemma also holds for the case of $K_1$ replaced by $K$.
\end{proof}

\begin{lem}[Mixture Lemma] \label{mix1b}
For any  $j\ge 1$, $\hat{\Q}^t_{j}(\xi)$ can be decomposed into
\be
\hat{\Q}^t_{j}(\xi)=\hat{\Q}^t_{j,1}(\xi)+\hat{\Q}^t_{j,2}(\xi), \label{M3}
\ee
where $\hat{\Q}^t_{j,1}(\xi)$ and $\hat{\Q}^t_{j,2}(\xi)$ are analytic for $\xi\in D_{\nu_0/2}$ and satisfy
\bma
\|\hat{\Q}^t_{3j,1}(\xi)V_0 \|&\le C_j (1+t)^{3j}e^{-\nu_0t}(1+|\xi|)^{-j}\ln^{j}(2+|\xi|)\|V_0\| ,\label{M31}\\
\|\hat{\Q}^t_{3j,2}(\xi)V_0 \|&\le C_j (1+t)^{2j-1}e^{|{\rm Im}\xi|t}(1+|\xi|)^{-j}\ln^{j}(2+|\xi|)\|V_0\|, \label{M32}
\ema  for any $V_0\in L^2_v\times \C^2\times \C^2$ and a positive constant $C_j$.
In particular, for $j=3,4,5,$
\be
\|\hat{\Q}^t_{j,2}(\xi)V_0 \|\le C_j (1+t)^{3}e^{|{\rm Im}\xi|t}(1+|\xi|)^{-\frac32}\ln (2+|\xi|)\|V_0\| .\label{M33}
\ee
\end{lem}

\begin{proof}
Note that
$$
e^{t\AA_2(\xi)} =\left( \ba
e^{tc(\xi)} &0_{1\times 4} \\
0_{4\times 1} & e^{t\BB_2(\xi)} \ea\right).
$$
It is easy to verify that
$$e^{t\BB_2(\xi)}=\sum^4_{j=1} e^{ \alpha_{j}t}\tilde{\X}_j\otimes \langle \tilde{\X}_j|,$$
where $ (\alpha_j,\tilde{\X}_{j})$, $j=1,2,3,4$ are the eigenvalues and eigenvectors of $\BB_2(\xi)$ as follows
$$
\left\{\bln
&\alpha_{1}=\alpha_{2}= -i\xi,\quad \alpha_{3}=\alpha_{4}=i\xi, \\
&\tilde{\X}_1=\sqrt{\frac12}(1,0,0,1),\quad \tilde{\X}_2=\sqrt{\frac12}(0,1,-1,0),\\
&\tilde{\X}_3=\sqrt{\frac12}(1,0,0,-1),\quad \tilde{\X}_4=\sqrt{\frac12}(0,1,1,0).
\eln\right.
$$

Thus, we can decompose $e^{t\AA_2(\xi)}$ into
$$ e^{t\AA_2(\xi)}=e^{-c(\xi)}\mathbb{E}_1+\sum^4_{j=1} e^{ \alpha_{j}t}\mathbb{D}_{jj}=:\S^t_1+\S^t_2, $$
where $\mathbb{E}_1$ and $\mathbb{D}_{jk}$ $(\mathbb{D}^0_j=\mathbb{D}_{jj})$ are $5\times 5$ matrices defined by
$$\mathbb{E}_1=\left( \ba
1 &0_{1\times 4} \\
0_{4\times 1} & 0_{4\times 4} \ea\right)_{5\times 5},\quad \mathbb{D}_{jk}=\left( \ba
0 &0_{1\times 4}  \\
0_{4\times 1}  & \tilde{\X}_k\otimes \langle \tilde{\X}_j| \ea\right)_{5\times 5}.$$

Thus,
\be
\hat{\Q}^t_0(\xi)=e^{t\AA_2(\xi)} =\hat{\Q}^t_{0,1}(\xi)+\hat{\Q}^t_{0,2}(\xi), \label{Q_0}
\ee
where $\hat{\Q}^t_{0,1}=\S^t_1$ and $\hat{\Q}^t_{0,2}=\S^t_2.$

Then, we deal with $\hat{\Q}^t_1(\xi)$. Since
$\mathbb{D}_{ij}\AA_3\mathbb{D}_{kl}=0$ for $i,j,k,l=1,2,3,4,$ it follows that
$$\intt \S^{t-s}_2\AA_3\S^s_2ds=0.$$
This implies that
\bmas
\hat{\Q}^t_1(\xi)&=\intt \S^{t-s}_1\AA_3\S^s_1ds+\intt \S^{t-s}_1\AA_3\S^s_2ds+\intt \S^{t-s}_2\AA_3\S^s_1ds \\
&=:I_{11}+I_{12}+I_{13}.
\emas
A direct computation gives
\bmas
I_{11}=&\intt e^{c(\xi)(t-s)}K_1e^{c(\xi)s}\mathbb{E}_1ds=\M^t_1\mathbb{E}_1,\\
I_{12}=&\sum^4_{j=1}\(e^{\alpha_jt} -e^{c(\xi)t}\)\frac1{\alpha_j-c(\xi)}\mathbb{E}_1\AA_3\mathbb{D}^0_j\\
=&\sum^4_{j=1}\(e^{\alpha_jt} -e^{c(\xi)t}\) \left( \ba
0 & u^0_j\otimes \langle \tilde{\X}_j| \\
0 & 0 \ea\right) ,
\\
I_{13}=&\sum^4_{j=1} \mathbb{D}^0_j\AA_3\mathbb{E}_1\frac1{\alpha_j-c(\xi)} \(e^{\alpha_jt} -e^{c(\xi)t}\)\\
=&\sum^4_{j=1}\left( \ba
0 & 0 \\
 -\tilde{\X}_j\otimes \langle u^0_j| & 0 \ea\right)  \(e^{\alpha_jt} -e^{c(\xi)t}\),
\emas
where $u^0_j=(\nu(v)+\alpha_j+iv_1\xi)^{-1} (v\cdot \X^2_j)\chi_0 $, $j=1,2,3,4$. 
Thus
\be
\hat{\Q}^t_1(\xi)=\hat{\Q}^t_{1,1}(\xi)+\hat{\Q}^t_{1,2}(\xi),\label{Q_1}
\ee
where
\bmas
\hat{\Q}^t_{1,1}&=\hat{\M}^t_1\mathbb{E}_1-\sum^4_{j=1}\left( \ba
0 & \hat{S}^tu^0_j\otimes \langle \tilde{\X}_j| \\
-\tilde{\X}_j\otimes \langle \hat{S}^t u^0_j| & 0 \ea\right)
=:\hat{\M}^t_1\mathbb{E}_1-\mathbb{Z}^t_1,
\\
\hat{\Q}^t_{1,2}&=\sum^4_{j=1}e^{\alpha_jt}\left( \ba
0 &   u^0_j\otimes \langle \tilde{\X}_j| \\
 -\tilde{\X}_j\otimes \langle u^0_j| & 0 \ea\right)
 =:\sum^4_{j=1}e^{\alpha_jt}\mathbb{D}^1_j.
\emas
Since $u^0_j(\xi)$ is analytic for $\xi\in D_{\nu_0/2}$ and satisfies
\be \|u^0_j(\xi)\| \le C(1+|\xi|)^{-\frac12},\quad  \xi\in D_{\nu_0/2}, \label{R_0}\ee
it follows that $\mathbb{Z}^t_1(\xi)$ and $\mathbb{D}^1_j(\xi)$ are analytic for $\xi\in D_{\nu_0/2}$ and satisfy for any $V_0\in L^2_v\times \C^2\times \C^2$,
\be \label{Q-1}
\left\{\bln
\|\mathbb{Z}^t_1(\xi)V_0\| &\le C(1+t)e^{-\nu_0t}(1+|\xi|)^{-\frac12}\|V_0\| ,\\
\|\mathbb{D}^1_j(\xi)V_0\| &\le C(1+|\xi|)^{-\frac12}\|V_0\| .
\eln\right.
\ee

To estimate $\hat{\Q}^t_2(\xi)$, we first decompose
\bmas
\hat{\Q}^t_2(\xi)&=\intt \mathbb{S}^{t-s}_1\AA_3\hat{\Q}^t_{1,1}ds+\intt \S^{t-s}_1\AA_3\hat{\Q}^t_{1,2}ds+\intt \S^{t-s}_2\AA_3\hat{\Q}^t_{1,1}ds \nnm\\
&\quad+\intt \S^{t-s}_2\AA_3\hat{\Q}^t_{1,2}ds=:I_{21}+I_{22}+I_{23}+I_{24}.
\emas
A direct computation shows that
\bmas
I_{21}=&\intt \S^{t-s}_1\AA_3\hat{\M}_1^s\mathbb{E}_1ds-\intt \S^{t-s}_1\AA_3\mathbb{Z}_1^s ds\\
=&\hat{\M}^t_2\mathbb{E}_1-\sum^4_{j=1}\left( \ba
-\intt \hat{S}^{t-s}r_j\otimes \langle \hat{S}^s u^0_j|ds & \hat{\M}^t_1u^0_j\otimes \langle \tilde{\X}_j| \\
0 & 0 \ea\right),
\\
I_{22}=&\sum^4_{j=1}\(e^{\alpha_jt} -e^{c(\xi)t}\)\frac1{\alpha_j-c(\xi)}\mathbb{E}_1\AA_3\mathbb{D}^1_j \\
=&\sum^4_{j=1}\(e^{\alpha_jt} -e^{c(\xi)t}\) \left( \ba
-u^0_j\otimes \langle u^0_j| & u^1_j\otimes \langle \tilde{\X}_j| \\
0 & 0 \ea\right) ,\\
I_{23}
=&\sum^4_{j=1} \left( \ba
0 & 0 \\
\tilde{\X}_j\otimes \langle \hat{\M}^t_1 u^0_j| & 0 \ea\right)- \sum^4_{j=1}\left( \ba
0 & 0 \\
\tilde{\X}_j\otimes \langle (e^{\alpha_jt} -e^{c(\xi)t})u^1_j| & 0 \ea\right)  \\
&+\sum^4_{j,l=1} \left( \ba
0 & 0 \\
0 & \Big( (e^{\alpha_jt} -e^{c(\xi)t})u^0_l,\bar{u}^0_j\Big)\tilde{\X}_j\otimes \langle \tilde{\X}_l| \ea\right),\\
I_{24}=&\sum^4_{j,l =1}\intt e^{\alpha_j(t-s)}e^{ \alpha_l s}\mathbb{D}^0_j\AA_3\mathbb{D}^1_l ds\\
=&-\sum^4_{j,l =1}e^{\alpha_jt}\intt e^{-(\alpha_j-\alpha_l)s}ds\left( \ba
0 & 0 \\
0 &\(  u^0_l,r_j\)\tilde{\X}_j\otimes \langle \tilde{\X}_l| \ea\right),
\emas
where $ u^l_j$, $l\ge 1$ and $r_j$ are defined by
\be
u^l_j=(\nu(v)+\alpha_j+iv_1\xi)^{-1}K_1u^{l-1}_j  ,\quad r_j=(v\cdot \X^2_j)\chi_0. \label{R1j}
\ee

Denote
\be
b_{jl}=-( u^0_j,r_l),\quad d_{jl}=( u^0_j,\bar{u}^0_l),\quad j,l=1,2,3,4.
\ee
It is straightforward to verify that
$$
\left\{\bln
&b_{jl}=b_{jj},\,\,\, l=l_j;\quad b_{jl}=0,\,\,\, l\ne j \,\,{\rm and} \,\, l\ne l_j,\\
&d_{jl}=\frac1{2i\xi}(b_{11}-b_{33}),\,\,\,  l=l_j;\quad d_{jl}=0,\,\,\, l\ne j \,\,{\rm and} \,\, l\ne l_j,
\eln\right.
$$
where $l_j$ is defined by
\be
l_1=3,\,\,\,\ l_2=4,\,\,\,\ l_3=1,\,\,\,\ l_4=2. \label{lj}
\ee
Thus
\be
\hat{\Q}^t_2(\xi)=\hat{\Q}^t_{2,1}(\xi)+\hat{\Q}^t_{2,2}(\xi), \label{Q_2}
\ee
where
\bmas
\hat{\Q}^t_{2,1}&= \hat{\M}^t_2\mathbb{E}_1-\sum^4_{j=1}\left( \ba
-\intt \hat{S}^{t-s}r_j\otimes \langle \hat{S}^s u^0_j|ds & \hat{\M}^t_1u^0_j\otimes \langle \tilde{\X}_j| \\
-\tilde{\X}_j\otimes \langle \hat{\M}^t_1 u^0_j| & 0 \ea\right)\\
&\quad-\sum^4_{j=1}\left( \ba
-\hat{S}^tu^0_j\otimes \langle u^0_j| & \hat{S}^tu^1_j\otimes \langle \tilde{\X}_j| \\
-\tilde{\X}_j\otimes \langle \hat{S}^t u^1_j| & (\hat{S}^tu^0_j,\bar{u}^0_{l_j})\tilde{\X}_{l_j}\otimes \langle \tilde{\X}_j| \ea\right)\\
&=:\hat{\M}^t_2\mathbb{E}_1-\mathbb{Z}^t_2,\\
\hat{\Q}^t_{2,2}&=\sum^4_{j=1}e^{\alpha_jt}\left( \ba
-u^0_j\otimes \langle u^0_j| &   u^1_j\otimes \langle \tilde{\X}_j| \\
 -\tilde{\X}_j\otimes \langle u^1_j| & (d_{jj}+tb_{jj})\tilde{\X}_j\otimes \langle \tilde{\X}_j| \ea\right)\\
 &\quad+\sum^4_{j=1}\left( \ba
0 &   0 \\
 0 & \(\frac{e^{i\xi t}-e^{-i\xi t}}{2i\xi} b_{jj}+e^{\alpha_jt}d_{jl_j}\)\tilde{\X}_{l_j}\otimes \langle \tilde{\X}_{j}|
 \ea\right)\\
 &=:\sum^4_{j=1}e^{\alpha_jt}(\mathbb{D}^2_j+tb_{jj}\mathbb{D}^0_{j})+\frac{e^{i\xi t}-e^{-i\xi t}}{2i\xi}\sum^4_{j=1}b_{jj}\mathbb{D}_{jl_j}.
\emas

By \eqref{bbb}, \eqref{ddd-1}, \eqref{ggg1} and \eqref{fff}, it holds that for $\xi\in D_{\nu_0/2}$,
\be \label{R_1}
\left\{\bln
& \|u^0_j\|_{L^1_v},\|K_1u^0_j\|_{L^2_v}\le C(1+|\xi|)^{-1}\ln(2+|\xi|),\\
& \|u^1_j\|_{L^2_v}\le C(1+|\xi|)^{-\frac32}\ln(2+|\xi|),\\
&|b_{jl}|\le C(1+|\xi|)^{-1}\ln(2+|\xi|),\quad |d_{jl}|\le C(1+|\xi|)^{-1}.
\eln\right.
\ee
This and \eqref{mix2}, \eqref{R_0} imply that $\mathbb{\hat{M}}^t_2(\xi)$, $\mathbb{Z}^t_2(\xi)$ and $\mathbb{D}^2_j(\xi)$ are analytic for $\xi\in D_{\nu_0/2}$ and satisfy  for any $V_0=(f_0,E_0,B_0)\in L^2_v\times \C^2\times \C^2$,
\be \label{Z_2}
\left\{\bln
&\|\mathbb{\hat{M}}^t_2 \mathbb{E}_1V_0\|\le C(1+t)^2e^{-\nu_0t}(1+|\xi|)^{-1} (\|f_0\|_{L^2_v }+\|\Tdv f_0\|_{L^2_v }),
\\
&\|\mathbb{Z}^t_2V_0\|\le C(1+t)^2e^{-\nu_0t}(1+|\xi|)^{-1}\ln(2+|\xi|)(\|V_0\|+\|f_0\|_{L^\infty_v }),
\\
&\|\mathbb{D}^2_{j}V_0\|,\|b_{jj}\mathbb{D}_{jl}V_0\|\le C(1+|\xi|)^{-1}\ln(2+|\xi|)\|V_0\| .
\eln\right.
\ee
Since
\be
|e^{\pm i\xi t}|,|\intt e^{i\xi(t-s)}e^{-i\xi s}ds|\le  e^{|{\rm Im}\xi|t}, \quad \xi\in \C, \label{F_1}
\ee
 it follows from \eqref{Z_2} that  for $\xi\in D_{\nu_0/2}$,
\bma
\|\hat{\Q}^t_{2,1}(\xi)\AA_3V_0\| &\le C(1+t)^2e^{-\nu_0t}(1+|\xi|)^{-1}\ln(2+|\xi|)\|V_0\| , \label{Q_4}\\
\|\hat{\Q}^t_{2,2}(\xi)\AA_3V_0\| &\le C(1+t)e^{|{\rm Im}\xi|t}(1+|\xi|)^{-1}\ln(2+|\xi|)\|V_0\| .\label{Q_5}
\ema

We now turn to estimate $\Q^t_3(\xi)$ as follows.
\bma
\hat{\Q}^t_3(\xi)&=\intt \hat{\Q}^{t-s}_{2,1}\AA_3\S^s_{1}ds+\intt \hat{\Q}^{t-s}_{2,1}\AA_3\S^s_{2}ds +\intt  \hat{\Q}^{t-s}_{2,2}\AA_3\S^s_{1}ds\nnm\\
&\quad+\intt \hat{\Q}^{t-s}_{2,2}\AA_3\S^s_{2}ds=:I_{31}+I_{32}+I_{33}+I_{34}. \label{Q33}
\ema
A direct computation shows that
\bma
I_{33}
=& \sum^4_{j=1}\(\mathbb{D}^2_{j}\AA_3\mathbb{E}_1 R_{j} - b_{jj}\mathbb{D}^0_j\AA_3\mathbb{E}_1 (R_{j})^2\)(e^{\alpha_jt}-e^{c(\xi)t})\nnm\\
&+\sum^4_{j=1}e^{\alpha_jt} tb_{jj}\mathbb{D}^0_j\AA_3\mathbb{E}_1 R_{j}-\intt \frac{\sin(\xi(t-s))}{\xi}\sum^4_{j=1}b_{jj}\mathbb{D}_{jl_j}\AA_3\mathbb{E}_1e^{c(\xi)s}ds, \label{I33}
\\
I_{34}
=&\sum^4_{j=1}e^{\alpha_jt}t\mathbb{D}^2_{j}\AA_3\mathbb{D}^0_j+\frac{\sin(\xi t)}{\xi}\sum^4_{j=1}\mathbb{D}^2_{j}\AA_3\mathbb{D}^0_{l_j}, \label{I34}
\ema
where $R_{j} =(\nu(v)+iv_1\xi+\alpha_j)^{-1},$ $j=1,2,3,4.$
Since
$$
\AA_3\mathbb{E}_1R_{j}=\left( \ba
K_1R_j & 0 \\
-K_3R_j & 0 \ea\right),
\quad
\mathbb{D}^2_j\AA_3\mathbb{D}^0_{l}=\left( \ba
0 & -(u^0_j,r_l)u^0_j\otimes \langle \tilde{\X}_l| \\
0 & -(u^1_j,r_l) \tilde{\X}_j\otimes \langle \tilde{\X}_l| \ea\right),
$$
 it follows from Lemma \ref{lem1}, \eqref{R_0} and \eqref{R_1} that
\bma
&\|\AA_3\mathbb{E}_1R_{j}\|\le  C(1+|\xi|)^{-\frac12},\label{I35}\\
&\|\mathbb{D}^2_j\AA_3\mathbb{D}^0_l\|\le C(1+|\xi|)^{-\frac32}\ln(2+|\xi|). \nnm
\ema
This together with \eqref{Z_2} and \eqref{F_1} imply that
\be \|I_{33} \|,\|I_{34} \|\le C(1+t)e^{|{\rm Im }\xi| t}(1+|\xi|)^{-\frac32}\ln(2+|\xi|). \label{Q-2}\ee

We have
\be
I_{32}=-\mathbb{Z}^t_{3}+\sum^4_{j=1}e^{\alpha_jt}\mathbb{D}^{3}_j,
\ee
where
\bmas
\mathbb{Z}^t_{3}=& \sum^4_{j,l=1}\left( \ba
0 &  \( ( u^0_{l} ,r_{j})\hat{S}^tR_ju^0_l+\hat{S}^tu^2_j+\hat{\M}^t_1u^1_j+\hat{\M}^t_2u^0_j\) \otimes \langle \tilde{\X}_j|  \\
0 & 0 \ea\right)\\
&+\sum^4_{j,l=1}\left( \ba
0 &  \[ ( u^0_{l} ,\bar{u}^0_{j}) \hat{S}^t R_{j}r_l +\intt \hat{S}^{t-s} r_l (\hat{S}^su^0_{l} ,\bar{u}^0_{j})ds \] \otimes \langle \tilde{\X}_{j}|\\
0 & \[(\hat{\M}^t_1 u^0_{l},\bar{u}^0_{j})+2(\hat{S}^t u^0_{l},\bar{u}^1_{j})\]\tilde{\X}_{l}\otimes\langle \tilde{\X}_{j}| \ea\right),
\\
\mathbb{D}^{3}_j=& \left( \ba
0 &  u^2_j \otimes \langle \tilde{\X}_j|  \\
0 & 0 \ea\right)+\sum^4_{l=1}\left( \ba
0 &   ( u^0_{l} ,\bar{u}^0_{j}) R_{j}r_l \otimes \langle \tilde{\X}_{j}|\\
0 & 2(u^0_{l},\bar{u}^1_{j})\tilde{\X}_{l}\otimes\langle \tilde{\X}_{j}| \ea\right).
\emas
It follows from Lemma \ref{lem1} and \eqref{R_1} that
\be\label{Q-4}
\left\{\bln
\|\mathbb{Z}^t_{3}(\xi) \|&\le C(1+t)^2e^{-\nu_0t}(1+|\xi|)^{-1}\ln (2+|\xi|) ,\\
\|\mathbb{D}^{3}_j(\xi) \|&\le C (1+|\xi|)^{-\frac32}\ln (2+|\xi|) .
\eln\right.
\ee

Thus
\be
\hat{\Q}^t_3(\xi)=\hat{\Q}^t_{3,1}(\xi)+\hat{\Q}^t_{3,2}(\xi), \label{Q_3a}
\ee
where
\bmas
\hat{\Q}^t_{3,1} &=  \intt \Q^{t-s}_{2,1}\AA_3\S^s_1ds-\mathbb{Z}^t_{3},
\\
\hat{\Q}^t_{3,2} &=\sum^4_{j=1}e^{\alpha_jt}\mathbb{D}^{3}_j +\intt  \hat{\Q}^{t-s}_{2,2}\AA_3\S^s_{1}ds+\intt \hat{\Q}^{t-s}_{2,2}\AA_3\S^s_{2}ds.
\emas
Then, it follows from \eqref{Q_4}, \eqref{Q33} and \eqref{Q-2}--\eqref{Q-4}  that for any $V_0\in L^2_v\times \C^2\times \C^2$,
\bma
\|\hat{\Q}^t_{3,1}(\xi)V_0 \| &\le C(1+t)^3e^{-\nu_0t}(1+|\xi|)^{-1}\ln(2+|\xi|) \|V_0\|, \label{Q_4a}\\
\|\hat{\Q}^t_{3,2}(\xi)V_0 \| &\le C(1+t)e^{|{\rm Im}\xi|t}(1+|\xi|)^{-\frac32}\ln(2+|\xi|)\|V_0\| .\label{Q_5a}
\ema
Thus, we obtain \eqref{M3}--\eqref{M32} for $j=1$ and \eqref{M33} for $j=3$.

Next, we estimate $\hat{\Q}^t_{j}(\xi)$ for $j\ge 4$. Indeed, we can decompose $\hat{\Q}^t_{j}(\xi)$ for $j\ge 4$ into
\be
\hat{\Q}^t_{j}(\xi)=\hat{\Q}^t_{j,1}(\xi)+\hat{\Q}^t_{j,2}(\xi), \label{Q_4b}
\ee
where
\bmas
\hat{\Q}^t_{j,1}(\xi)&=\intt \hat{\Q}^{t-s}_{2,1}\AA_3\hat{\Q}^s_{j-3,1}ds=: I_{j1}, \\
\hat{\Q}^t_{j,2}(\xi)&=\intt \hat{\Q}^{t-s}_{2,1}\AA_3\hat{\Q}^s_{j-3,2}ds+\intt  \hat{\Q}^{t-s}_{2,2}\AA_3 \hat{\Q}^s_{j-3,1}ds\nnm\\
&\quad+\intt  \hat{\Q}^{t-s}_{2,2}\AA_3\hat{\Q}^s_{j-3,2} ds=: I_{j2}+I_{j3}+I_{j4}.
\emas

We first estimate $\hat{\Q}^t_{4}(\xi)$, i.e., $I_{4k}$ for $k=1,2,3,4$. For $I_{41}$, it follows from \eqref{Q_4} that
\bmas
\|I_{41}\|&\le C(1+|\xi|)^{-1}\ln(2+|\xi|)\intt (1+t-s)^2e^{-\nu_0(t-s)}(1+s)e^{-\nu_0s}ds \\
&\le C(1+t)^4e^{-\nu_0t}(1+|\xi|)^{-1}\ln(2+|\xi|).
\emas
By \eqref{Q-1}, \eqref{Q_4} and \eqref{Q_5}, we have
\bmas
\|I_{42}\|
\le &C(1+|\xi|)^{-\frac32}\ln(2+|\xi|)\intt e^{-\frac{\nu_0}2(t-s)}e^{|{\rm Im}\xi|s} ds\\
\le & Ce^{|{\rm Im}\xi|t}  (1+|\xi|)^{-\frac32}\ln(2+|\xi|),
\\
\|I_{44}\|
\le &C (1+|\xi|)^{-\frac32}\ln(2+|\xi|)\intt (1+t-s)e^{|{\rm Im}\xi|(t-s)}e^{|{\rm Im}\xi|s} ds \\
\le& C(1+t)^2e^{|{\rm Im}\xi|t}(1+|\xi|)^{-\frac32}\ln(2+|\xi|).
\emas
To estimate $I_{43}$, we decompose
\bmas
I_{43}&= \intt \hat{\Q}^{t-s}_{2,2}\AA_3 \hat{\M}^s_{1}\mathbb{E}_1ds+\intt \hat{\Q}^{t-s}_{2,2}\AA_3  \mathbb{Z}^s_{1}ds\\
&=\intt I_{33}(t-s) \AA_3 \hat{\S}^s_{1} ds+\intt \hat{\Q}^{t-s}_{2,2}\AA_3  \mathbb{Z}^s_{1}ds,
\emas where $I_{33}=\intt  \hat{\Q}^{t-s}_{2,2}\AA_3\S^s_{1}ds$  satisfies \eqref{I33}.
Thus, it follows from \eqref{Q-1}, \eqref{Q-2} and \eqref{Q_5} that
\bmas
\|I_{43}\| \le &C\intt (1+t-s)e^{|{\rm Im}\xi|(t-s)}e^{-\frac{\nu_0}2s} ds (1+|\xi|)^{-\frac32}\ln(2+|\xi|)\\
\le & C(1+t)e^{|{\rm Im}\xi|t}  (1+|\xi|)^{-\frac32}\ln(2+|\xi|).
\emas
Thus, we obtain  \eqref{M33} for $j=4$. Similarly, we can prove  \eqref{M33} for $j=5$.

Finally, we prove \eqref{M3}--\eqref{M32} by induction argument. Suppose \eqref{M3}--\eqref{M32} hold for $j\le l$ with $l\ge 1$ given.
By \eqref{Q_4}, \eqref{Q_5}, \eqref{Q_4b} and induction hypothesis, we have
\bmas
\|\hat{\Q}^t_{3l+3,1}\|&\le C\intt (1+t-s)^2e^{-\nu_0(t-s)}s^{3l}e^{-\nu_0s}ds (1+|\xi|)^{-l-1}\ln^{l+1} (2+|\xi|)\\
&\le C(1+t)^{3l+3}e^{-\nu_0t} (1+|\xi|)^{-l-1}\ln^{l+1} (2+|\xi|),
\\
\|\hat{\Q}^t_{3l+3,2}\| &\le C\intt (1+t-s)e^{|{\rm Im}\xi|(t-s)}s^{2l-1}e^{|{\rm Im}\xi|s}ds (1+|\xi|)^{-l-1}\ln^{l+1} (2+|\xi|)\\
&\le C(1+t)^{2l+1}e^{|{\rm Im}\xi|t} (1+|\xi|)^{-l-1}\ln^{l+1} (2+|\xi|),
\emas  which implies \eqref{M3}--\eqref{M32} hold for $j = l+1$ so that it
holds for any $j\ge 1 $.
\end{proof}


With the help of   Lemmas \ref{mix1a}--\ref{mix1b},
we have the following estimate on $\hat{\U}_{n}(t,\xi)$ defined by \eqref{U_4}.

\begin{lem}\label{mix3}
For each $n\ge 1$,  we have the following decomposition
\be
\hat{\U}_{n}(t,\xi) =\hat{\U}_{n,1}(t,\xi)+\hat{\U}_{n,2}(t,\xi),\label{w1}
\ee
where $\hat{\U}_{n,1}(t,\xi)$ and $\hat{\U}_{n,2}(t,\xi)$ are analytic for $\xi\in D_{\nu_0/2}$ and satisfy 
\bma
\|\hat{\U}_{3n,1}(t,\xi)\| &\le C_n (1+t)^{3n} e^{- \nu_0t }(1+|\xi|)^{-n}\ln^{n}(2+|\xi|) , \label{w2}
\\
\|\hat{\U}_{3n,2}(t,\xi)\| &\le C_n (1+t)^{2n-1}e^{|{\rm Im}\xi|t} (1+|\xi|)^{-n}\ln^{n}(2+|\xi|) .
\ema
In  particular, for $n=3,4,5$,
\be
\|\hat{\U}_{n,2}(t,\xi)\| \le C_n (1+t)^{3}e^{|{\rm Im}\xi|t} (1+|\xi|)^{-\frac32}\ln (2+|\xi|) . \label{w2a}
\ee
\end{lem}
\begin{proof}
By \eqref{U_4}, \eqref{H_n}  and \eqref{V_5},  we can construct $\hat{\U}_n(t,\xi) $ for $n\ge 0$ as follows:
\be \label{U-1}
\left\{\bln
&\hat{\U}_n =(\hat{I}_n,\hat{U}^{1}_n,\hat{U}_n^{2},0,\hat{U}_n^{3})^T, \quad \hat{I}_n= \hat{H}_n+\hat{J}_n ,\\
 &\hat{H}_0 \equiv0, \quad \hat{U}_{0}^{1} =\frac{1}{\nu_0+i\xi}(\hat{S}^t\hat{f}(0),\chi_0),\\
& \hat{H}_{n+1}=\sum^{n}_{l=0}\intt \M^{t-s}_{l} v_1\chi_0\hat{U}^{1}_{n-l} ds, \\
&\hat{U}^{1}_{n+1}=\frac{1}{\nu_0+i\xi}(\hat{H}_{n+1}+\hat{\M}^t_{n+1}\hat{f}(0),\chi_0)+\frac{\nu_0}{\nu_0+i\xi}\hat{U}_{n}^{1},\\
&\hat{V}_{n} =(\hat{J}_{n},\hat{U}^{2}_{n},\hat{U}^{3}_{n})^T=\hat{\Q}^t_{n} \hat{V}_0(0),
\eln\right.
\ee
where 
$\hat{V}_0(0)=(\hat{f}(0),\hat{E}_r(0),\hat{B}_r(0))$  defined by \eqref{l2}. Note that $\hat{\U}_n(t,\xi) $ is an operator from $\mathcal{\hat{X}}_1$ to $L^2_v\times \C^3\times \C^3$. For any $\hat{X}_0=(\hat f_0,\hat E_0,\hat B_0) \in \mathcal{\hat{X}}_1$, we have $\hat{G}_0(\xi)\hat{X}_0=\hat{X}_0$, i.e.,
$$\hat{f}(0)\hat{X}_0=\hat f_0, \quad \hat{E}(0)\hat{X}_0=\hat E_0, \quad \hat{B}(0)\hat{X}_0=\hat B_0.$$
This leads to $\|\hat{f}(0)\|$, $\|\hat{E}(0)\|$, $\|\hat{B}(0)\|\le 1.$

First, we will prove by induction for $\xi\in D_{\nu_0/2}$ that
\be \|\hat{H}_{3n}\| \le C(1+t)^{3n}e^{- \nu_0t }(1+|\xi|)^{-n}, \quad \|\hat{U}^1_{3n}\|\le C(1+t)^{3n}e^{- \nu_0t }(1+|\xi|)^{-n-1}. \label{w3}\ee
For $n=0,$ it holds that
$$\hat{H}_0\equiv0, \quad \|\hat{U}^1_{0}\|\le \frac{1 }{|\nu_0+i\xi|}\|\hat{S}^t\hat{f}(0)\|\le Ce^{-\nu_0t}(1+|\xi|)^{-1}.$$
Suppose \eqref{w3} holds for $l\le n-1.$ By Lemma \ref{mix1a}, we have
\bmas
\|\hat{H}_{3n}\|&\le \sum^{3n-1}_{l=0}\intt \|\hat{\M}^{t-s}_{l} v\chi_0\|\|\hat{U}^1_{3n-l-1}\| ds \le C(1+t)^{3n}e^{- \nu_0t }(1+|\xi|)^{-n} ,\\
\|\hat{U}^1_{3n}\|&\le \frac{1}{|\nu_0+i\xi| }(\| \hat{H}_{3n}\|+\|\hat{\M}^t_{3n}\hat{f}(0) \|+\nu_0 \|\hat{U}^1_{3n-1}\|)\le C(1+t)^{3n}e^{- \nu_0t }(1+|\xi|)^{-n-1} ,
\emas
which gives \eqref{w3}.

By Lemma \ref{mix1a}, we decompose
$$ \hat{V}_n(t)=\hat{\Q}^t_{n,1} \hat{V}_{0}(0)+\hat{\Q}^t_{n,2} \hat{V}_{0}(0)=:\hat{V}_{n,1}(t) +\hat{V}_{n,2}(t), $$
where
$$\hat{V}_{n,1}=(\hat{J}_{n,1},\hat{U}^{2}_{n,1},\hat{U}^{3}_{n,1}),\quad \hat{V}_{n,2}=(\hat{J}_{n,2},\hat{U}^{2}_{n,2},\hat{U}^{3}_{n,2}).$$
Set
\be \label{w6}
\left\{\bln
\hat{\U}_{n}&= \hat{\U}_{n,1}+\hat{\U}_{n,2},\\
\hat{\U}_{n,1}&= (\hat{J}_{n,1}+\hat{H}_n,\hat{U}^{1}_n,\hat{U}^{2}_{n,1},0,\hat{U}^{3}_{n,1})^T ,\\
\hat{\U}_{n,2}&= (\hat{J}_{n,2},0,\hat{U}^{2}_{n,2},0,\hat{U}^{3}_{n,2})^T ,
\eln\right.
\ee
where $(\hat{H}_n,\hat{U}^1_n)$ is defined by \eqref{U-1}.

By \eqref{w3}, \eqref{w6} and Lemma \ref{mix1b}, we conclude that $\hat{\U}_{n,1}(t,\xi)$ and $\hat{\U}_{n,2}(t,\xi)$ are analytic for $\xi\in D_{\nu_0/2}$ and satisfy \eqref{w2}--\eqref{w2a}.
\end{proof}

Define the $n^{th}$ degree kinetic wave $Y_n$ and  the $n^{th}$ degree remaining part $Z_n $  as follows:
\be
Y_n =\sum_{k=0}^{3n}\U_k ,\quad Z_n =G-Y_n, \label{YZ}
\ee
where $Y_n=(T_n,Y_n^{1},Y_n^{2},0,Y_n^{3})^T$ and  $Z_n=(R_n,Z_n^{1},Z_n^{2},0,Z_n^{3})^T$. Note that $\hat{\U}_n $, $\hat{Y}_n $ and $\hat{Z}_n$ are operators from $\mathcal{\hat{X}}_1$ to $L^2_v\times \C^3\times \C^3$.
Furthermore, from \eqref{w1}  and \eqref{YZ}, we write
\be \hat{Y}_{n}(t,\xi)=\hat{Y}_{n,1}(t,\xi)+\hat{Y}_{n,2}(t,\xi), \label{Yn2} \ee
where
\be
\hat{Y}_{n,1}(t,\xi)=\sum^{3n}_{k=0}\hat{\U}_{k,1}(t,\xi),\quad \hat{Y}_{n,2}(t,\xi)=\sum^{3n}_{k=0}\hat{\U}_{k,2}(t,\xi). \label{Yn1}
\ee
And we will prove the following theorem.

\begin{lem}\label{W_1}
For each $n\ge 2$,  $\hat{Y}_{n,2}(t,\xi) $ is  analytic for $\xi\in D_{\nu_0/2}$
and satisfies
\be
\|(\hat{Y}_{n,2}  -\hat{G}_{2})(t,\xi)\|\le C_ne^{|{\rm Im}\xi|t}(1+t)^{2n-1}(1+|\xi|)^{-\frac32}\ln (2+|\xi|),  \label{Y_1}
\ee
where $\hat{G}_2$ is defined by \eqref{G_b}  and $C_n>0$ is a positive constant. In  particular, for $|x|\ge 2t$,
\be
\|Y_{n,2}(t,x) -G_{2}(t,x)\|\le C_n e^{-\frac{\nu_0(|x|+t)}{8}}. \label{Y_2}
\ee
\end{lem}

\begin{proof}By \eqref{w6}, \eqref{Q_0}, \eqref{Q_1} and \eqref{Q_2}, we have
\be \label{G_2}
\hat{G}_3(t,\xi)=:\sum^2_{l=0}\hat{\U}_{l,2}(t,\xi)=\left( \hat{G}^{ij}_3(t,\xi)\right)_{3\times 3},
\ee
where
\bmas
\hat{G}^{11}_3(t,\xi)=& \sum^4_{j=1}e^{\alpha_jt}  u^0_j\otimes \langle u^0_j|,\\
\hat{G}^{1k}_3(t,\xi)=& \sum^4_{j=1}e^{\alpha_jt}  \(u^0_j\otimes  \langle \X^k_j|+u^1_j\otimes  \langle \X^k_j|\),\\
\hat{G}^{k1}_3(t,\xi)=& \sum^4_{j=1}e^{\alpha_jt}  \(\X^k_j\otimes  \langle u^0_j|+\X^k_j\otimes  \langle u^1_j|\),\\
\hat{G}^{kn}_3(t,\xi)=&\sum^4_{j=1}e^{\alpha_jt}\(1+tb_{jj}+d_{jj}\)\X^k_j\otimes \langle \X^n_j|\\
&+\sum^4_{j= 1} \( \frac{e^{i\xi t}-e^{-i\xi t}}{2i\xi}b_{jj}+e^{\alpha_jt}d_{jl_j}\)\X^k_j\otimes\langle \X^n_{l_j} |
\emas
with $k,n=2,3$ and $l_j$ defined by \eqref{lj}.

For $n\ge 1$, we decompose
\be \hat{Y}_{n,2}(t,\xi) -\hat{G}_{2}(t,\xi)=(\hat{G}_3-\hat{G}_{2})(t,\xi)+(\hat{Y}_{n,2}-\hat{G}_3)(t,\xi), \label{Y_3}\ee
where
$$ (\hat{Y}_{n,2}-\hat{G}_3)(t,\xi) =\sum^{3n}_{l=3}\hat{\U}_{l,2}(t,\xi) .$$
By Lemma \ref{mix3}, it holds that for $l\ge 3$ and $\xi\in D_{\nu_0/2}$,
\be \|\hat{\U}_{l,2}(t,\xi) \|\le C_le^{|{\rm Im}\xi|t}(1+t)^{2[\frac{l}3]-1}(1+|\xi|)^{-\frac32}\ln(2+|\xi|). \label{Y_4}\ee

By \eqref{G_b}, \eqref{G_2} and noting that $b_{jj}=\gamma_j$, we obtain
\bma
\hat{G}^{11}_{2}-\hat{G}^{11}_3=&\sum^4_{j=1}e^{\alpha_jt}\(e^{b_{jj}t}-1\)u^0_j\otimes \langle u^0_j|,
\nnm\\
\hat{G}^{1k}_{2}-\hat{G}^{1k}_3=&\sum^4_{j=1}e^{\alpha_jt}\(e^{b_{jj}t}-1\)u^0_j\otimes \langle \X^k_j|+\sum^4_{j=1}e^{\alpha_jt}u^1_j \otimes \langle \X^k_j| ,
\nnm\\
\hat{G}^{k1}_2-\hat{G}^{k1}_3=& \sum^4_{j=1}e^{\alpha_jt}  \(e^{b_{jj}t}-1\) \X^k_j\otimes  \langle u^0_j|+\sum^4_{j=1}e^{\alpha_jt}\X^k_j\otimes  \langle u^1_j|, \label{G_3}
\\
\hat{G}^{kn}_{2}-\hat{G}^{kn}_3=&\sum^4_{j=1}e^{\alpha_jt}\[\(e^{b_{jj}t}- 1-tb_{jj}\)+\(e^{b_{jj}t}-1\) d_{jj}\] \X^k_j \otimes \langle \X^n_j |\nnm\\
&+\sum^4_{j= 1} \(  \frac{\sin(\xi t)}{\xi}b_{jj}+e^{\alpha_jt}d_{jl_j}\)\X^k_j\otimes\langle \X^n_{l_j} |,\quad k,n=2,3.\nnm
\ema
Since $b_{jj}(\xi)$, $d_{jl_j}(\xi)$ and $\sin(\xi t)/\xi$ are analytic for $\xi\in D_{\nu_0/2}$ and satisfy
\bmas
 &| e^{b_{jj}t}-1|\le C|b_{jj}|t\le C\frac{\ln (2+|\xi|)}{1+|\xi|}t,\\
 &|e^{b_{jj}t}-1-tb_{jj}|\le C |b_{jj}|^2t^2\le C\frac{\ln^2 (2+|\xi|)}{(1+|\xi|)^2}t^2,\\
&\frac{|\sin(\xi t)|}{|\xi|}|b_{jj}| \le \frac{1+t}{1+|\xi|}|b_{jj}|\le C\frac{\ln (2+|\xi|)}{(1+|\xi|)^2}(1+t),\\
&|d_{jl_j}|=\left|\frac1{2i\xi}(b_{11} -b_{33} )\right|\le C\frac{\ln(2+|\xi|)}{(1+|\xi|)^{2}},
\emas
it follows from \eqref{G_3} and \eqref{F_1} that for $\xi\in D_{\nu_0/2}$,
\be
\|(\hat{G}_{2} -\hat{G}_3)(t,\xi)\|\le Ce^{|{\rm Im}\xi|t}(1+t)^{2}(1+|\xi|)^{-\frac32}\ln (2+|\xi|). \label{G_23}
\ee
This together with \eqref{Y_3} and \eqref{Y_4} proves \eqref{Y_1}.

Finally, we can prove \eqref{Y_2} as follows. Assume that $x\ge 0$.
By \eqref{G_23}, we have
\bma
\| (G_{2} -G_3)(t,x)\|&= \left\|\int^{\infty}_{-\infty} e^{ix\xi} (\hat{G}_{2} -\hat{G}_3)(t,\xi) d\xi\right\| \nnm\\
&=\left\|\int^{\infty}_{-\infty} e^{ix(u+i\frac{\nu_0}2)} (\hat{G}_{2} -\hat{G}_3)\(t,u+i\frac{\nu_0}2\) du\right\|\nnm\\
&\le C e^{-\frac{\nu_0(|x|-t)}{2}} \int^{\infty}_{-\infty} (1+t)^{2} (1+| u|)^{-\frac32}\ln (2+|u|)d u\nnm\\
&\le C (1+t)^{2}e^{-\frac{\nu_0(|x|-t)}{2}}\le  C e^{-\frac{\nu_0(|x|+t)}{8}} \label{G23}
\ema for $|x|\ge 2t$.
Similarly, by \eqref{Y_4} we obtain
$$ \|(Y_{n,2} -G_{3})(t,x)\|\le C e^{-\frac{\nu_0(|x|+t)}{8}},\quad |x|\ge 2t, $$
which together with  \eqref{G23} implies \eqref{Y_2}. 
\end{proof}

Then, we can show the pointwise estimates on the remaining terms $Z_n(t,x)$  defined by \eqref{YZ}. To this end, we  first derive the following equation for $\dt\hat{U}_{n}^{1}$. By \eqref{c-w}--\eqref{c-w4} and noting that $\hat{I}_n=\hat{J}_n+\hat{H}_n$, we have
\bmas
\dt\hat{U}_{n}^{1}=&-(\hat{I}_n,v_1\chi_0)-\frac{1}{\nu_0+i\xi}(\nu(v)(\hat{I}_n-\hat{I}_{n-1}),\chi_0)\\
&+\frac{\nu_0}{\nu_0+i\xi}( \hat{I}_n,v_1\chi_0)+\frac{\nu_0}{\nu_0+i\xi}\dt\hat{U}_{n-1}^{1},\,\,\, n\ge 1,
\emas
and hence
\bma
\dt \hat{U}_{0}^{1}=&-(\hat{I}_0,v_1\chi_0)+\frac{\nu_0}{\nu_0+i\xi}( \hat{I}_0,v_1\chi_0)-\frac{1}{\nu_0+i\xi}(\nu(v)\hat{I}_0,\chi_0),\label{u10}\\
\dt \hat{U}_{n}^{1}=&-(\hat{I}_n,v_1\chi_0)+\sum^{n}_{l=1}\frac{\nu_0^{n-l+1}}{(\nu_0+i\xi)^{n-l+1}}(\hat{I}_{l}-\hat{I}_{l-1},v_1\chi_0)\nnm\\
&-\sum^{n}_{l=1}\frac{\nu_0^{n-l}}{(\nu_0+i\xi)^{n-l+1}}(\nu(v)(\hat{I}_{l}-\hat{I}_{l-1}),\chi_0), \quad n\ge 1. \label{u11}
\ema
By \eqref{u10} and \eqref{u11}, we obtain
\be
\dt \hat{Y}^{1}_{n}=-( \hat{T}_n,v_1\chi_0)-\hat{ \Theta}_{3n}, \label{er2}
\ee
where
\be
\hat{\Theta}_{n}=\sum^{n}_{l=0}\frac{\nu_0^{n-l}}{(\nu_0+i\xi)^{n-l+1}}(\nu(v)\hat{I}_{l}-\nu_0v_1\hat{I}_{l},\chi_0).
\label{psij}
\ee

Thus, it follows from \eqref{U_4}--\eqref{c-w4}, \eqref{U_5}  and \eqref{er2} that $\hat{Y}_n=(\hat T_n,\hat Y_n^{1},\hat Y_n^{2},0,\hat Y_n^{3})^T$  satisfies
\be \label{ew}
\left\{\bln
&\dt \hat{T}_n+i \xi v_1 \hat{T}_n-L_1\hat{T}_n-v \cdot \hat{Y}_n^{4}\chi_0=-K_1\hat{I}_{3n}- v\cdot \hat{U}_{3n}^{4} \chi_0,\\
&\dt \hat{Y}^{1}_{n}=-( \hat{T}_n,v_1\chi_0)-\hat{\Theta}_{3n},\quad i\xi\hat{Y}_n^{1}=-( \hat{T}_n,\chi_0)-\nu_0 \hat{U}_{3n}^{1},\\
 &\dt \hat{Y}_n^{2}=i \xi \O_1 \hat{Y}_n^{3}-(\hat{T}_n,\bar{v}\chi_0)-( \hat{I}_{3n},\bar{v}\chi_0),\\
 & \dt \hat{Y}_n^{3}=-i \xi \O_1\hat{Y}_n^{2},\\
 &\hat{Y}_n(0)
 =\hat{G}_0-\hat{Z}_n(0),
  \eln\right.
\ee
and   $\hat{Z}_n=(\hat R_n,\hat Z_n^{1},\hat Z_n^{2},0,\hat Z_n^{3})^T$ satisfies
\be \label{er}
\left\{\bln
&\dt \hat{R}_n+i \xi v_1 \hat{R}_n-L_1\hat{R}_n-v \cdot \hat{Z}_n^{4}\chi_0=K_1\hat{I}_{3n}+v\cdot \hat{U}_{3n}^{4}\chi_0,\\
&\dt \hat{Z}^{1}_{n}=-( \hat{R}_n,v_1\chi_0)+\hat{\Theta}_{3n},\quad i\xi\hat{Z}_n^{1}=-( \hat{R}_n,\chi_0)+\nu_0 \hat{U}_{3n}^{1},\\
 &\dt \hat{Z}_n^{2}= i \xi \O_1 \hat{Z}_n^{3}-(\hat{R}_n,\bar{v}\chi_0)+( \hat{I}_{3n},\bar{v}\chi_0),\\
 & \dt \hat{Z}_n^{3}= -i \xi \O_1\hat{Z}_n^{2},\\
 &\hat{Z}_n(0)
  =(0, \mbox{($\frac{\nu_0}{\nu_0+i\xi})^{3n+1}$}\hat E_1(0),0,0,0)^T,
  \eln\right.
\ee
where $\hat{Y}_n^{4}=(\hat{Y}_n^{1},\hat{Y}_n^{2})$, $\hat{Z}_n^{4}=(\hat{Z}_n^{1},\hat{Z}_n^{2})$ and  $\hat{U}_n^{4}=(\hat{U}_n^{1},\hat{U}_n^{2})$.

\begin{lem}\label{W_2}
For any $n\ge 1$, there exists a small constant $\delta_0>0$ such that for any $\delta\in (0,\delta_0)$,
\be
\| \hat{Z}_{n}(t,\xi)\|\le C_{n}\delta^{-2n}e^{\delta t}(1+|\xi|)^{-n}\ln^{n}(2+|\xi|) , \label{w1c}
\ee
where $\xi\in \R$ and $C_{n}>0$ is a constant. Moreover, for $n\ge 2$ and $|x|\ge 2t$,
\be
\| Z_{n}(t,x)\|\le C_ne^{- (|x|+t)/D}, \label{w4}
\ee
where $C_{n},D>0$ are  constants.
\end{lem}

\begin{proof}
Taking inner product of $\eqref{er}_1$ and $\hat{R}_n$ and choosing the real part, we obtain
\bma
&\frac12\Dt \|\hat{R}_n\|^2 - (L_1\hat{R}_n,\hat{R}_n) -{\rm Re} (\hat{Z}^{4}_n, (\hat{R}_n,v\chi_0))  \nnm\\
=& {\rm Re} (K_1\hat{I}_{3n},\hat{R}_n) + {\rm Re} (\hat{U}^{4}_{3n},(\hat{R}_n,v\chi_0)) . \label{R_n}
\ema
Since
\bmas
-{\rm Re} (\hat{Z}^{4}_n, (\hat{R}_n,v\chi_0)) =& {\rm Re} (\hat{Z}^{1}_n, \dt \hat{Z}^{1}_n-\hat{\Theta}_{3n})  \\
&+ {\rm Re} (\hat{Z}^{2}_n, \dt \hat{Z}^{2}_n-i\xi \O_1\hat{Z}^{3}_n+(\hat{I}_{3n},\bar{v}\chi_0)) \\
=&\frac12\Dt (|\hat{Z}^{4}_n|^2+|\hat{Z}^{3}_n|^2)-  {\rm Re} (\hat{Z}^{1}_n,\hat{\Theta}_{3n})  -  {\rm Re} (\hat{Z}^{2}_n,(\hat{I}_{3n},\bar{v}\chi_0) ),
\emas
it follows from \eqref{R_n} that
\be
\frac12\Dt\|\hat{Z}_n\|^2 +\frac\mu2 (P_r\hat{R}_n,P_r\hat{R}_n)
\le \frac{C}{\delta} (|\hat{\Theta}_{3n}|^2 + |\hat{U}^{4}_{3n}|^2 + \|\hat{I}_{3n}\|^2  )+\delta \|  \hat{Z}_n\|^2,\label{Zj}
\ee
where $\delta\in (0,\delta_0)$ with $\delta_0>0$ being a small constant.

By Lemma \ref{mix3}, it holds that for  $\xi\in D_{\nu_0/2}$,
\be \|\hat \U_{3n}(t,\xi)\|\le C(1+t)^{2n-1}e^{|{\rm Im}\xi| t}(1+|\xi|)^{-n}\ln^n(2+|\xi|). \label{lll}\ee
This together with \eqref{psij} implies that for $\xi\in D_{\nu_0/2}$,
\bma
|\hat{\Theta}_{3n}(t,\xi)|&\le \sum^{3n}_{k=0}\frac{\nu_0^{3n-k}}{(\nu_0+|\xi|)^{3n-k+1}}\|\hat{\mathcal{U}}_{k}(t,\xi)\| \nnm\\
&\le C_n(1+t)^{2n-1}e^{|{\rm Im}\xi| t}(1+|\xi|)^{-n-1}\ln^n(2+|\xi|). \label{phi1}
\ema
Applying Gronwall's inequality to \eqref{Zj}, we have
\bmas
\| \hat{Z}_{n}(t,\xi)\|^2&\le e^{2\delta t}\|\hat{Z}_{n}(0)\|^2+\frac{C}{\delta}\intt  e^{2\delta (t-s)}(|\hat{\Theta}_{3n}|^2 + |\hat{U}^{4}_{3n}|^2 + \|\hat{I}_{3n}\|^2  )ds\\
&\le C(1+|\xi|)^{-2n}\ln^{2n}(2+|\xi|)e^{2\delta t}\(1 + \frac{C}{\delta}\intt  e^{-2\delta s}(1+s)^{4n-2}ds\)\\
&\le C\delta^{-4n}e^{2\delta t}(1+|\xi|)^{-2n}\ln^{2n}(2+|\xi|) ,\quad \xi\in \R,
\emas
where we have used the following inequality
$$
\intt e^{-\delta s}s^nds=\frac{n!}{\delta^{n+1}}\(1-\(1+\delta t+\frac{(\delta t)^2}{2!}+\cdots+\frac{(\delta t)^n}{n!}\)e^{-\delta t}\)\le \frac{n!}{\delta^{n+1}}.
$$
This proves \eqref{w1c}.

To prove \eqref{w4}, set
$$w(t,x)=e^{\eps (|x|-\theta t)}\quad {\rm with}\,\,\, \eps>0,\,\, \theta>1.$$
Then
$$\dt w=-\eps \theta  w,\quad \dx w= \eps\frac{x}{|x|} w.$$

From \eqref{er}, we have
\be \label{er1}
\left\{\bln
&\dt R_n+v_1\dx R_n-L_1R_n-v \cdot Z^{4}_n\chi_0=K_1I_{3n}+v \cdot U^{4}_{3n}\chi_0,\\
&\dt Z^{1}_n=-(R_n,v_1\chi_0)+\Theta_{3n},\,\,\, \dx Z^{1}_n=( R_n,\chi_0)+\nu_0 U^{1}_{3n},\\
 &\dt Z^{2}_n= \dx \O_1 Z^{3}_n-(R_n,\bar{v}\chi_0)-( I_{3n},\bar{v}\chi_0),\\
 & \dt Z^{3}_n= -\dx \O_1Z^{2}_n. 
  \eln\right.
\ee
Taking inner product of $\eqref{er1}_1$ and $R_nw$ gives
\bma
&\frac12\Dt\intra \|R_n\|^2wdx+\frac12 \eps \theta \intra \|R_n\|^2wdx-\frac12\intra (v_1R_n,R_n)\dx wdx\nnm\\
&-\intra (L_1R_n,R_n)wdx-\intra   Z^{4}_n (R_n,v\chi_0) wdx\nnm\\
=&\intra (K_1I_{3n},R_n)wdx+\intra U^{4}_{3n}(R_n,v\chi_0)wdx. \label{en2}
\ema
Since
\bmas
\intra (v_1R_n,R_n)\dx wdx
=&2\intra (v_1P_dR_n,P_rR_n)\dx wdx+\intra (v_1P_rR_n,P_rR_n)\dx wdx\nnm\\
\le&  \eps\intra  (P_dR_n,P_dR_n) wdx+C\eps \intra(|v|P_rR_n,P_rR_n)  wdx, 
\emas
and
\bmas
-\intra   Z^{4}_n (R_n,v\chi_0) wdx
=&\intra  Z^{1}_n (\dt Z^{1}_n-\Theta_{3n})wdx\nnm\\
& +\intra Z^{2}_n [\dt Z^{2}_n-\dx \O_1Z^{3}_n+(I_{3n},\bar{v}\chi_0)]  wdx\nnm\\
=&\frac12\Dt\intra |(Z^{4}_n,Z^{3}_n)|^2  wdx+\frac12 \eps \theta\intra  |(Z^{4}_n,Z^{3}_n)|^2 wdx\nnm\\
&-\intra [Z^{1}_n\Theta_{3n}- Z^{2}_n(I_{3n},\bar{v}\chi_0)]wdx+\intra Z^{2}_n \O_1Z^{3}_n \dx wdx , 
\emas
 it follows from \eqref{en2} that for   $0<\eps\ll1$, $\theta=3/2$,
\bmas
&\frac12\Dt\intra \|Z_n\|^2wdx+\frac{\eps}{4}  \intra \|Z_n\|^2wdx+\frac\mu2\intra (\nu P_1R_n,P_1R_n)wdx \\
\le &\frac{C}{\eps}\intra (|\Theta_{3n}|^2 +|U^{4}_{3n}|^2+\|I_{3n}\|^2)wdx ,
\emas
which gives rise to
\be
 \intra \|Z_n\|^2wdx \le e^{-\frac{\eps t}{2}}\intra \|Z_n(0)\|^2wdx  +C\intt\intra e^{-\frac{\eps(t-s)}{2}}(|\Theta_{3n}|^2+\|\U_{3n}\|^2 )wdx ds. \label{kkk2}
\ee

We claim that for $0\le b\le \nu_0$ and $n\ge \alpha+1$,
\be
\intra e^{b|x|}(|\dxa\Theta_{3n}|^2+\|\dxa \U_{3n}\|^2 )dx\le Ce^{b t}(1+t)^{4n-2}. \label{kkk}
\ee
Assume that $\hat V(\xi)$ is analytic for $\xi\in D_{\delta}$ with $\delta>0$ and satisfies that $|\xi|^{\alpha}|\hat V(\xi)|\to 0$ as $|\xi|\to \infty$. Since
$$
e^{b x}\dxa V(x)=C\int^{\infty}_{-\infty} e^{ix\xi+b x}\xi^{\alpha}\hat{V}(\xi)d\xi
=C\int^{\infty}_{-\infty} e^{ixu}(u-ib)^{\alpha}\hat{V}(u-ib)du
$$
for $x\in \R$ and $0<b<\delta$, it follows that
$$
e^{b |x|}\dxa V(x)=\left\{\bln
&\mathcal{F}^{-1}\((u-ib)^{\alpha}\hat{V}(u-ib)\),\quad x\ge 0,\\
&\mathcal{F}^{-1}\((u+ib)^{\alpha}\hat{V}(u+ib)\),\quad x<0.
\eln\right.
$$
By Parseval's equality, we  obtain
\be
\int^{\infty}_{-\infty} e^{2b|x|}\|\dxa V(x)\|^2dx\le C\int^{\infty}_{-\infty} |u\pm ib|^{2\alpha}\|\hat{V}(u\pm ib)\|^2du, \label{u1}
\ee for any $0<b<\delta$. Thus, it follows from \eqref{u1}, \eqref{lll} and \eqref{phi1} that for $0\le 2b\le \nu_0$,
\bmas
&\quad\int^{\infty}_{-\infty} e^{2b|x|}(|\Theta_{3n}|^2+\|\U_{3n}\|^2 )(t,x)dx\\
 &\le C\int^{\infty}_{-\infty} |u\pm ib |^{2\alpha} (|\hat{\Theta}_{3n}|^2+\|\hat{\U}_{3n}\|^2 )(t,u\pm ib)du\\
&\le C\int^{\infty}_{0} e^{2bt}(1+t)^{4n-2}(1+u)^{2\alpha-2n}\ln^{2n}(2+u)du\\
&\le Ce^{2bt}(1+t)^{4n-2},
\emas
which prove \eqref{kkk}. Similarly,  by \eqref{er} we can obtain that for $0\le b\le \nu_0$ and $3n\ge \alpha$,
\be \intra e^{b|x|}\|\dxa Z_n(0)\|^2 dx\le C . \label{kkk1}\ee
Combining \eqref{kkk2}, \eqref{kkk} and \eqref{kkk1}, we have
\be
\intra \|Z_n\|^2 wdx\le C . \label{r2}
\ee
Similarly,
\be
\intra \|\dx Z_n\|^2 wdx \le  C . \label{r3}
\ee
Thus by \eqref{r2}, \eqref{r3} and Sobolev's embedding theorem, we have
$$
w(t,x)\|Z_n(t,x)\|^2\le C .
$$
This completes the proof of  the lemma.
\end{proof}


\begin{lem}\label{l-high2} Given $n\ge 2 $. There exist constants $\eta_0>0$ and $C>0$   such that
\bq
\|  G_{H,1}(t,x)-Y_{n,1}(t,x) \| \le Ce^{-\eta_0t},
\eq where $ G_{H,1}$ and $Y_{n,1}$ are defined by \eqref{GL0} and \eqref{Yn1} respectively.
\end{lem}

\begin{proof}
By \eqref{L-R} and  \eqref{YZ}, we have
$$
G(t,x)=G_L(t,x)+G_M(t,x)+G_{H,0}(t,x)+G_{H,1}(t,x),
$$
and
$$
G(t,x)=Y_{n,1}(t,x)+Y_{n,2}(t,x)+Z_n(t,x).
$$
It follows that
\be
G_{H,1}-Y_{n,1}=Z_n+(Y_{n,2}-G_{H,0})-(G_L+G_M).\label{GH1}
\ee
From Lemmas \ref{gh0}, \ref{W_1} and \ref{W_2}, it holds that for $n\ge 2$ and $\xi\in \R$,
\bma
\|(\hat{Y}_{n,2}-\hat{G}_{H,0})(t,\xi)\|&\le  \|(\hat{Y}_{n,2}-\hat{G}_{2})(t,\xi)\|+\|(\hat{G}_2 -\hat{G}_{H,0})(t,\xi)\|\nnm\\
&\le C (1+t)^{2n-1}(1+|\xi|)^{-\frac32}\ln(2+|\xi|),
\\
\|\hat{Z}_n(t,\xi)\|&\le Ce^{\delta t}(1+|\xi|)^{-2}\ln^2(2+|\xi|).\label{Z_1}
\ema
By \eqref{GH1}--\eqref{Z_1}, we obtain
\be
\|\hat{G}_{H,1}(t,\xi)-\hat{Y}_{n,1}(t,\xi)\|\le Ce^{\delta t}(1+|\xi|)^{-\frac32}\ln(2+|\xi|). \label{x3}
\ee
Also, by  Lemmas \ref{l-1} and \ref{mix3}, we have
\be
\|\hat{G}_{H,1}(t,\xi)\|+\|\hat{Y}_{n,1}(t,\xi)\|\le Ce^{-\kappa_0t}. \label{x5}
\ee
Thus, it follows from \eqref{x3} and \eqref{x5} that
$$
\|\hat{G}_{H,1}(t,\xi)-\hat{Y}_{n,1}(t,\xi)\|\le Ce^{-\frac{\kappa_0}{4}t}(1+|\xi|)^{-\frac{9}{8}}\ln^{\frac{3}{4}}(2+|\xi|),
$$
which gives
\bmas
\| (G_{H,1} -Y_{n,1})(t,x) \|=&C\left\|\int^\infty_{-\infty}e^{i x\xi}  [\hat{G}_{H,1}(t,\xi)-\hat{Y}_{n,1}(t,\xi)]d \xi\right\|\\
\le& C\int^\infty_{-\infty}e^{-\frac{\kappa_0}{4}t} (1+|\xi|)^{-\frac{9}{8}}\ln^{\frac{3}{4}}(2+|\xi|)d \xi \le Ce^{-\eta_0t},
\emas
where $\eta_0=\kappa_0/4$. The proof of the lemma is completed.
\end{proof}

Combining the decompositions \eqref{L-R} and \eqref{YZ} and the pointwise estimates in sections \ref{fluid} and \ref{kinetic},
we are ready to  prove Theorem \ref{green1} as follows.

\begin{proof}[\underline{\textbf{Proof of Theorem \ref{green1}}}]
By \eqref{L-R},  \eqref{YZ} and \eqref{Yn2}, we have
\be
G (t,x) =G_{L,0}  +G_{L,1} +G_{H,0}  +G_{H,1} +G_M , \label{L-R2}
\ee
and
\be
G(t,x)=Y_{2,1} +Y_{2,2} + Z_2 . \label{W-R1}
\ee
Let
$$G_4(t,x)=G(t,x) -G_{2}(t,x)-Y_{2,1}(t,x).$$
For $|x|\le 2t$, by \eqref{L-R2} we decompose
$$
G_4=  G_{L,0} +(G_{H,0}-G_{2,2}) +(G_{H,1}-Y_{2,1}) +(G_{2,1}+G_M+ G_{L,1}).
$$
For $|x|\ge 2t$, by \eqref{W-R1} we decompose
$$
G_4= (Y_{2,2}-G_{2})(t,x)+Z_2(t,x).
$$
Thus
\be
G= W_0+G_{F,0}+G_{F,1} + G_{R}, \label{GV}
\ee
where
$$
\left\{\bln
W_0&=Y_{2,1}+G_2,\\
 G_{F,0} &=G_{L,0} 1_{\{|x|\le 2t\}} ,\quad
  G_{F,1} =(G_{H,0}-G_{2,2})1_{\{|x|\le 2t\}} ,\\
  G_{R} &= (G_{H,1}-Y_{2,1}) 1_{\{|x|\le 2t\}}+(G_{2,1}+G_M+ G_{L,1})1_{\{|x|\le 2t\}}\\
  &\quad+(Y_{2,2}-G_{2}+Z_2 )1_{\{|x|\ge 2t\}}.
  \eln\right.
$$
 By Lemmas \ref{gl0}--\ref{gh0}, we can obtain \eqref{in1}--\eqref{in2}. By Lemmas \ref{W_1}, \ref{W_2} and \ref{l-high2}, we can obtain \eqref{in3}. The proof of the theorem is completed.
\end{proof}

\section{Green's function of Boltzmann equation}\setcounter{equation}{0}
\label{sect3}

In this section, we establish the pointwise space-time estimates of the Green's function to the linear Boltzmann equation \eqref{LB} as stated in Theorem \ref{green2}. As stated in the introduction,
the authors  in  \cite{Liu1,Liu2} study the pointwise estimates of the solution to the linearized Boltzmann equation with the initial data $f_0$ being compactly supported. We rewrite the results in \cite{Liu1,Liu2} to establish the pointwise estimates on the Green's function for the linearized Boltzmann equation~\eqref{LB} with the initial data $f_0=\delta(x)I_v$. Different from \cite{Liu1,Liu2},  we study the kinetic waves through Fourier transform of Boltzmann equation \eqref{be} in order to overcome the difficulty caused by the singularity arising from  $\delta(x)$. For this, we introduce an improved mixture lemma in frequency space  (Lemma  \ref{mix1a}), and construct the singular kinetic waves $\hat{W}_1$ of the Green's function $\hat{G}_b$  stated in \eqref{W-1}. 

\subsection{Spectrum of Boltzmann equation}
\begin{thm}[\cite{Ellis,Ukai1,Liu2}]\label{spect2}
Let $\sigma(\BB_0(\xi))$ denotes the spectral set
of the operator $\BB_0(\xi)$. We have

(1) For any $r_1>0$  there
exists $\alpha=\alpha(r_1)>0$ such that  for $|\xi|\geq r_1$ that
\be
\sigma(\BB_0(\xi))\subset\{\lambda\in\mathbb{C}\,|\,{\rm Re}\lambda\le -\alpha\}. \label{sg2a}
\ee

(2)
There exists a constant $r_0>0$ such that  $\sigma(\BB_0(\xi))$ for   $|\xi|\leq r_0$
consists of five points $\{\eta_j(\xi),j=-1,0,1,2,3\}$ in the domain $\mathrm{Re}\lambda>-\mu/2$,  which are  analytic functions of $\xi$ for $|\xi|\leq r_0$:
\be                                   \label{specr0c}
 \left\{\bln
 &\eta_{\pm1}(\xi)=\pm i\mathbf{c}\xi -A_{\pm1}\xi^2 +O(\xi^3), \\
& \eta_{0}(\xi) =-A_0\xi^2 +O(\xi^3) ,\\
 &\eta_{2}(\xi) = \eta_{3}(\xi) =-A_2\xi^2 +O(\xi^3),
 \eln\right.
 \ee
 where $A_{i}>0$, $i=-1,0,1,2,3$ are defined by
\bq
\left\{\bln
&A_{j}=- (L^{-1} P_1v_1E_j,v_1E_j)>0,
\\
&E_{\pm1}=\sqrt{\frac3{10}}\chi_0\pm\frac{\sqrt2}2v_1\sqrt{M}+\sqrt{\frac15}\chi_4,\\
&E_0=\sqrt{\frac25}\chi_0-\sqrt{\frac35}\chi_4,\\
&E_j= v_j\sqrt{M},\quad j=2,3.
\eln\right.
\eq

The eigenfunctions $ \varphi_j(\xi) $ are  pairwise orthogonal and satisfy \be
 \left\{\bln
 &(\varphi_j(\xi),\overline{\varphi_k(\xi)})=\delta_{jk},
  \quad  j, k=-1,0,1,2,3,                                  \label{eigfr0-1}
 \\
&\varphi_j(\xi) =E_j +\varphi_{j,1}\xi+O(\xi^2), \quad |\xi|\leq r_0,
 \eln\right.
 \ee
 where the coefficients $\varphi_{j,1}$  have the following expressions
\bq
  \left\{\bln                      \label{eigf1-1}
 &\varphi_{l,1}=\sum^1_{k=-1}b^l_{k}E_k+ i L^{-1} P_1v_1E_l,\,\,\  l=-1,0,1, \\
 &\varphi_{k,1}=  i L^{-1} P_1v_1E_k, \quad k=2,3,
  \eln\right.
  \eq  with $b^j_{k}$, $j,k=-1,0,1$ given by
  \be
   \left\{\bln
  &b^j_{j}=0,\quad
  b^j_{k}=\frac{(L^{-1}P_1v_1E_{j},v_1E_{k})}{i(u_k-u_j)},\,\,\ j\ne k,\\
  &u_{\pm1}=\mp\sqrt{\frac53}, \quad u_0=0.
   \eln\right.
  \ee
\end{thm}

\begin{thm}[\cite{Ellis,Ukai1}]\label{E_3c}The semigroup $S(t,\xi)=e^{t\BB_0(\xi)}$   satisfies
 \be
 S(t,\xi)f=S_1(t,\xi)f+S_2(t,\xi)f,
     \quad f\in L^2(\R^3_v), \ \ t>0, \label{E_3d}
 \ee
 where
 \bq
 S_1(t,\xi)f=\sum^3_{j=-1}e^{t\eta_j(\xi)} \(f,\overline{\varphi_j(\xi)}\) \varphi_j(\xi)
               1_{\{|\xi|\leq r_0\}},           \label{E_5}
 \eq
with $(\beta_j(|\xi|),\varphi_j(\xi))$ being the eigenvalue and eigenfunction of the operator $\BB_0(\xi)$ given by Theorem~\ref{spect2} for $|\xi|\le r_0$,
and $S_2(t,\xi)f =: S(t,\xi)f-S_1(t,\xi)f$ satisfies for a constant $\sigma_0>0$ independent of $\xi$ that
 \bq
 \|S_2(t,\xi)f\| \leq Ce^{-\sigma_0t}\|f\| ,\quad t>0.\label{B_3}
 \eq
\end{thm}

\subsection{Green's function}
By Theorems \ref{spect2} and \ref{E_3c}, we decompose the operator $G_b(t,x)$ into low-frequency part and high-frequency part:
\be \label{L-R1}
\left\{\bln
&G_b(t,x)=G_l(t,x)+G_h(t,x),\\
&G_l(t,x)=\frac1{\sqrt{2\pi}}\int_{|\xi|\le r_0/2} e^{ i x \xi +t\BB_0(\xi)}d\xi,
\\
&G_h(t,x)=\frac1{\sqrt{2\pi}}\int_{|\xi|\ge r_0/2} e^{ i x \xi +t\BB_0(\xi)}d\xi,
\eln\right.
\ee
and the operator $G_l(t,x)$ can be further divided into the fluid part and the non-fluid part:
\be G_l(t,x)=G_{l,0}(t,x)+G_{l,1}(t,x), \label{GL-0a}\ee
where
\bma
\hat G_{l,0}(t,\xi)=&\sum_{j=-1}^3e^{\eta_j(\xi)t} \varphi_j(\xi)\otimes \langle \varphi_j(\xi) |,\label{GL0a}\\
\hat G_{l,1}(t,\xi)=&\hat G_{l}(t,\xi)-\hat G_{l,0}(t,\xi). \label{GL1a}
\ema

First, we have the following estimates on each part of $\hat{G}_{b}(t,\xi)$ defined  by \eqref{L-R1}.

\begin{lem}[\cite{Ukai1,Ukai3}]\label{l-1a}
For any $g_0\in L^2(\R^3_v)$, there exist positive constants $C$ and $\kappa_0$ such that
\bma \|\hat{G}_l(t,\xi)g_0\| &\le  C\|g_0\| , \label{l1a}\\
\|\hat{G}_{l,1}(t,\xi)g_0\| &\le Cr_0e^{-\kappa_0t} \|g_0\| ,\label{low2a}\\
\|\hat{G}_{h}(t,\xi)g_0\|&\le Ce^{-\kappa_0t} \|g_0\| , \label{high1a}
\ema
where $\hat{G}_l(t,\xi)$,  $\hat{G}_{h}(t,\xi)$ and $\hat{G}_{l,1}(t,\xi)$ are  defined by \eqref{L-R1} and \eqref{GL-0a} respectively.
\end{lem}

From \cite{Liu1}, we have the pointwise estimate on the fluid part $G_{l,0}(t,x)$ as follows.

\begin{thm}[\cite{Liu1}]\label{green4}
For  any given constant $C_1>1$, there exist $C,D>0$ such that for $|x|\le C_1t$,
\bma
\|\dxa G_{l,0}(t,x)\|
&\le C\bigg(\sum^1_{i=-1}(1+t)^{-\frac{1+\alpha}2}e^{-\frac{(x-i\mathbf{c}t)^2}{Dt}} + e^{-\frac{t}{D}}\bigg), \label{in1a} \\
\|\dxa G_{l,0}(t,x)P_1\|
&\le C\bigg(\sum^1_{i=-1}(1+t)^{-\frac{2+\alpha}2}e^{-\frac{(x-i\mathbf{c}t)^2}{Dt}} + e^{-\frac{t}{D}}\bigg),
\\
\|\dxa P_1G_{l,0}(t,x)P_1\| &\le C\bigg(\sum^1_{i=-1}(1+t)^{-\frac{3+\alpha}2}e^{-\frac{(x-i\mathbf{c}t)^2}{Dt}} + e^{-\frac{t}{D}}\bigg). \label{in4a}
     \ema
\end{thm}

Next, we construct the kinetic waves based on Fourier analysis and the  mixture lemma \ref{mix1a}. Note that $\hat G_b(t,\xi)$ satisfies
\be \label{be}
\left\{\bln
&\dt \hat G_b+iv_1\xi\hat G_b-L\hat G_b=0,\\
&\hat{G}_b(0,\xi)=1(\xi)I_v.
\eln\right.
\ee
By \eqref{be}, we define the approximate solution sequence $I_k$ for $\hat G_b$ as follow
\be \label{c-wa}
\left\{\bln
&\dt \hat{I}_0+i v_1 \xi \hat{I}_0+\nu(v)\hat{I}_0=0,\\
&\hat{I}_0(\xi,0)=1(\xi)I_v,
\eln\right.
\ee
and
\be \label{c-w1a}
\left\{\bln
&\dt \hat{I}_k+i v_1 \xi \hat{I}_k+\nu(v)\hat{I}_k=K\hat{I}_{k-1}, \\
&\hat{I}_k(\xi,0)=0, \quad k\ge 1.
\eln\right.
\ee
By \eqref{c-wa} and \eqref{c-w1a}, we define 
\be
\W_k(t,x)=\sum_{j=0}^{3k}I_j(t,x), \quad \mathcal{R}_k(t,x)=G_b(t,x)-\mathcal{W}_k(t,x). \label{W-R}\ee

It is straightforward to verify that $\hat{\W}_k(t,\xi)$  satisfies
\be \label{ew1}
\left\{\bln
&\dt \hat{\W}_k+i v_1 \xi \hat{\W}_k-L\hat{\W}_k=-K\hat{I}_{3k},\\
&\hat{\W}_k(\xi,0)=1(\xi)I_v,
\eln\right.
\ee
and  $\hat{\RR}_k(t,\xi)$ satisfies
\be \label{er1a}
\left\{\bln
&\dt \hat \RR_k+i v_1 \xi \hat \RR_k-L\hat \RR_k=K\hat{I}_{3k},\\
&\hat \RR_k(\xi,0)=0.
\eln\right.
\ee


\begin{lem}\label{W_1a}
For each $k\ge 1$,  $\hat{I}_{k}(t,\xi) $ is  analytic for $\xi\in D_{\nu_0}$
and satisfies
\be
\|\hat{I}_{3k}(t,\xi)\|\le C_k (1+|\xi|)^{-k}e^{-\frac{\nu_0t}{2}},\label{w1b}
\ee for $C_k>0$ a positive constant.
In  particular, for $0\le b\le \nu_0$ and  $k\ge \alpha+1$,
\be
\intra e^{b|x|}\|\dxa I_{3k}(t)\|^2dx\le C_k e^{-\frac{\nu_0t}{2}}.\label{w1a}
\ee
\end{lem}
\begin{proof}Since
\be
\hat{I}_k(t,\xi)=\hat{\M}^t_{b,k}(\xi)=:\intt\int^{s_1}_0\cdots\int^{s_{k-1}}_0 \hat{S}^{t-s_1}K \hat{S}^{s_1-s_2}\cdots  \hat{S}^{s_{k-1}-s_{k}}K \hat{S}^{s_{k}}ds_{k}\cdots ds_1, \label{mix4}
\ee
we obtain \eqref{w1b} from Lemma \ref{mix1a}.
Thus, it follows from \eqref{u1} that for $0\le 2b\le \nu_0$ and $\alpha\le j-1$,
\bmas
\int^{\infty}_{-\infty} e^{2b|x|}\|\dxa I_{3k}(t,x)\|^2dx
 &\le C_k\int^{\infty}_{-\infty}  |u\pm ib |^{2\alpha} \|\hat{I}_{3k} (t,u\pm ib ) \|^2du\\
&\le C_k\int^{\infty}_{0} e^{-\frac{\nu_0t}{2}}(1+u)^{2\alpha-2k} du\le C_ke^{-\frac{\nu_0t}{2}},
\emas
which gives \eqref{w1a}.
The proof of the lemma is completed.
\end{proof}

Then, we can show the pointwise estimate on the remaining terms $\mathcal{R}_k(t,x)$ as follows.

\begin{lem}\label{green5}Given $k\ge 2$. There exist constants $\delta>0$ small and $C_k>0$   such that
\be
\| \RR_k(t,x)\| \le C_ke^{-\delta(|x|-\frac32\mathbf{c}t)}.
\ee
\end{lem}

\begin{proof}Set
$$w(t,x)=e^{\eps (|x|-\theta t)}\quad {\rm with}\,\,\, \eps>0,\,\, \theta>1.$$

By \eqref{er1a}, we have
 \be
 \dt \RR_k +v_1\dx \RR_k -L\RR_k = KI_{3k} . \label{er3}
 \ee
Taking the inner product of \eqref{er3} and $\RR_kw $ to get
\bma
&\frac12\Dt\intra \|\RR_k\|^2wdx+\frac12\eps\theta\intra \|\RR_k\|^2wdx-\frac12\intra (v_1\RR_k,\RR_k)\dx wdx\nnm\\
&-\intra (L P_1\RR_k, P_1\RR_k)wdx=\intra (KI_{3k},\RR_k)wdx  . \label{r1}
\ema
Since (cf.  Lemma 5.3 in \cite{Liu1})
$$|(v_1P_0g,P_0g)|\le \mathbf{c}(P_0g,P_0g), \quad \forall g\in L^2_v,$$
it follows that
\bma
\intra (v_1\RR_k,\RR_k)\dx wdx
=&\intra (v_1P_0\RR_k,P_0\RR_k)\dx wdx+2\intra (v_1P_0\RR_k,P_1\RR_k)\dx wdx\nnm\\
&+\intra (v_1P_1\RR_k,P_1\RR_k)\dx wdx\nnm\\
\le&   \frac54\mathbf{c}\eps\intra (P_0\RR_k,P_0\RR_k)wdx+C\eps \intra (|v| P_1\RR_k,P_1\RR_k)  wdx. \label{r1a}
\ema
Thus, it follows from \eqref{r1} and \eqref{r1a} that for $0<\eps \ll1$, $\theta=3\mathbf{c}/2$,
\bma
& \Dt\intra \|\RR_k\|^2wdx+\frac14\eps \mathbf{c}\intra \|\RR_k\|^2wdx+\mu\intra (\nu P_1\RR_k, P_1\RR_k)wdx\nnm\\
&\le \frac{C}{\eps}\intra \|KI_{3k}\|^2wdx. \label{H_1}
\ema
Applying Gronwall's inequality to \eqref{H_1} and using Lemma \ref{W_1a}, we have
\be
\intra \|\RR_k\|^2wdx\le C\intt\intra e^{- \frac{\eps \mathbf{c}(t-s)}{4}} \|I_{3k}\|^2w dx ds\le C . \label{r6}
\ee
Similarly,
\be
\intra \|\dx \RR_k\|^2wdx\le  C .  \label{r6a}
\ee
By \eqref{r6} and \eqref{r6a}, we have
$$w(t,x)\|\RR_k(t,x)\|^2\le C.$$
This completes the proof of  the lemma.
\end{proof}

\begin{thm}\label{green6}
Let $k\ge 2$.   There exist constants $\eta_0>0$ small and $C_k>0$   such that
\bq
\|G_h(t,x)-\W_k(t,x)\|\le C_k e^{-\eta_0t}.\label{x5a}
\eq
\end{thm}

\begin{proof}
By \eqref{L-R1} and  \eqref{W-R}, we have
$$
G_b(t,x)=G_l(t,x)+G_h(t,x)=\W_k(t,x)+\RR_k(t,x).
$$
It follows that
$$
G_h(t,x)-\W_k(t,x)=\RR_k(t,x)-G_l(t,x).
$$
From \eqref{er1a}, Lemma \ref{l-1a} and Lemma \ref{W_1a},
\bmas
\|\hat{\RR}_k(t,\xi)\|
&\le \intt \|\hat{G}_b(t-s)K\hat{I}_{3k}(s)\| ds\le C\intt \|\hat{I}_{3k}(s)\|ds\nnm\\
&\le C_k(1+|\xi|)^{-k},
\emas
which gives
\be
\|\hat{G}_h(t,\xi)-\hat{\W}_k(t,\xi)\|\le \|\hat{\RR}_k(t,\xi)\|+\|\hat{G}_l(t,\xi)\|\le C(1+|\xi|)^{-k}. \label{x3a}
\ee
From Lemma \ref{l-1a} and Lemma \ref{W_1a},
\be
\|\hat{G}_h(t,\xi)\|+\|\hat{\W}_k(t,\xi)\|\le Ce^{-\kappa_0t}. \label{x5b}
\ee
Thus, it follows from \eqref{x3a} and \eqref{x5b} that there exists a constant $\eta_0>0$ such that
$$
\|\hat{G}_h(t,\xi)-\hat{\W}_k(t,\xi)\|\le Ce^{-\frac{\kappa_0}{4}t}(1+|\xi|)^{-\frac{3k}{4}},
$$
which gives
\bmas
\|G_h(t,x)-\W_k(t,x)\|=&C\left\|\intra e^{i x \xi}[\hat{G}_h(t,\xi)-\hat{\W}_k(t,\xi)]d\xi\right\|\\
\le& C\intra e^{-\frac{\kappa_0}{4}t}(1+|\xi|)^{-\frac{3k}{4}}d\xi \le Ce^{-\eta_0t},
\emas
for $k\ge 2$. The proof of the lemma is completed.
\end{proof}

\begin{proof}[\underline{\textbf{Proof of Theorem \ref{green2}}}]
By \eqref{L-R1} and  \eqref{W-R}, we have
$$
G_b(t,x)=G_{l,0}(t,x)+G_{l,1}(t,x)+G_h(t,x),
$$
and
$$
G_b(t,x)=\W_2(t,x)+ \RR_2(t,x).
$$

Thus
\be G_b(t,x)=G_{b,0}(t,x)+G_{b,1}(t,x)+W_1(t,x), \label{G_bb}\ee
where
\be \label{G_bc}
\left\{\bln
W_1(t,x)&=\W_2, \quad G_{b,0}(t,x) =G_{l,0} 1_{\{|x|\le 2\mathbf{c}t\}},\\
G_{b,1}(t,x)&=(G_{l,1}+G_h-\W_2)1_{\{|x|\le 2\mathbf{c}t\}}+\RR_2(t,x)1_{\{|x|\ge 2\mathbf{c}t\}}.
\eln\right.
\ee
By Lemmas \ref{green4}--\ref{green6}, the statements in the theorem hold.
\end{proof}

\section{The  nonlinear system}
\label{behavior-nonlinear}
 \setcounter{equation}{0}
In this section, we show the pointwise estimate of the solution to the Cauchy problem for Vlasov-Maxwell-Boltzmann systems based on the estimates for the linearized problem obtained in Sections~\ref{fluid}--\ref{sect3}.

\subsection{Energy estimate}
We first obtain some energy estimates.
Let $N$ be a positive integer and $U=(f_1,f_2,E,B)$. Set
\bma
\mathbb{E}_{N,k}(U)&=\sum_{|\alpha|+|\beta|\le N}\|w^k\dxa\dvb (f_1,f_2)\|^2_{L^2_{x,v}}+\sum_{|\alpha|\le N}\|\dxa(E,B)\|^2_{L^2_x},\label{energy3}\\
\mathbb{H}_{N,k}(U)&= \sum_{|\alpha|+|\beta|\le N}\|w^k\dxa\dvb
(P_1f_1,P_rf_2)\|^2_{L^2_{x,v}}+\sum_{1\le|\alpha|\le N}\|\dxa (E,B)\|^2_{L^2_x}\nnm\\
&\quad +\sum_{|\alpha|\le N-1}\|\dx^{\alpha+1}  (P_0f_1,P_{d}f_2)\|^2_{L^2_{x,v}}+\|E\|^2_{L^2_x} ,\label{energy3a}\\
\mathbb{D}_{N,k}(U)&=\sum_{|\alpha|+|\beta|\le N}\|w^{\frac12+k}\dxa\dvb  (P_1f_1,P_rf_2)\|^2_{L^2_{x,v}}+\sum_{1\le |\alpha|\le N-1}\|\dxa  (E,B)\|^2_{L^2_x}\nnm\\
&\quad+\sum_{|\alpha|\le N-1}\|\dx^{\alpha+1}  (P_0f_1,P_{d}f_2)\|^2_{L^2_{x,v}}+\|E\|^2_{L^2_x} ,\label{energy3b}
\ema
for $k\ge 0$. For brevity, we denote $\mathbb{E}_N(U)=\mathbb{E}_{N,0}(U)$, $\mathbb{H}_N(U)=\mathbb{H}_{N,0}(U)$ and $\mathbb{D}_N(U)=\mathbb{D}_{N,0}(U)$.


The energy estimates of the VMB system near a global Maxwellian have been well  established when
 $x\in \R^3$, cf.  \cite{Duan4,Li1,Strain}. Note that a key feature of 3-D in the energy estimate used  in  \cite{Duan4,Li1,Strain} comes from the Sobolev's inequality
$$\|u\|_{L^6_x}\le C\|\Tdx u\|_{L^2_x},\quad \forall u=u(x),\,\, x\in \R^3.$$
However, this Sobolev's inequality does not hold  when $x\in \R$. In fact, the 1-D Sobolev's inequality takes the form
$$\|u\|_{L^\infty_x}\le C\|u\|^{1/2}_{L^2_x}\|\dx u\|^{1/2}_{L^2_x},\quad \forall u=u(x),\,\, x\in \R.$$
Thus,  the nonlinear coupling for the 1-D VMB system is bounded by
$$C\sqrt{ \mathbb{E}_N(U)\mathbb{D}_N(U)}\| P_0f_1\|_{L^2_{v}(L^\infty_x)}+C\sqrt{\mathbb{E}_N(U)}\mathbb{D}_N(U),$$
instead of $C\sqrt{\mathbb{E}_N(U)}\mathbb{D}_N(U)$. By using this  and a similar  argument as in \cite{Duan4,Li1,Strain}, we can obtain
 the following energy estimates in 1-D setting.

\begin{lem}\label{energy1}
For  $N\ge 2$, there are two equivalent energy functionals
$\mathcal{E}_{N}(\cdot)\sim \mathbb{E}_N(\cdot)$, $\mathcal{H}_{N}(\cdot)\sim \mathbb{H}_N(\cdot)$
such that the following holds. If
$\mathbb{E}_N(U)$ is sufficiently small, then the solution $U=(f_1,f_2,E,B)(t,x,v)$ to the two-species
 VMB system \eqref{VMB3a}--\eqref{VMB3d} satisfies
\bma
\Dt \mathcal{E}_{N}(U) + \mu \mathbb{D}_N(U) &\le  C \mathbb{E}_N(U)\| P_0f_1\|^2_{L^2_{v}(L^\infty_x)},  \label{G_ba}\\
\Dt \mathcal{H}_{N}(U)+\mu \mathbb{D}_N(U)&\le C\|\dx P_0f_1\|^2_{L^2_{x,v}} +C\|\dx B\|^2_{L^2_{x}} +C \mathbb{E}_N(U)\| P_0f_1\|^2_{L^2_{v}(L^\infty_x)}.\label{G_4}
\ema
\end{lem}

\begin{lem}\label{energy2}
For $N\ge 2$ and $k\ge 1$, there are the equivalent energy functionals $\mathcal{E}_{N,k}(\cdot)\sim \mathbb{E}_{N,k}(\cdot)$, $\mathcal{H}_{N,k}(\cdot)\sim \mathbb{H}_{N,k}(\cdot)$
such that if $\mathbb{E}_{N,k}(U)$ is sufficiently small, then the solution $U=(f_1,f_2,E,B)(t,x,v)$ to the two-species
 VMB system \eqref{VMB3a}--\eqref{VMB3d} satisfies
\bma \Dt \mathcal{E}_{N,k}(U)+\mu \mathbb{D}_{N,k}(U)&\le C\mathbb{E}_N(U)\| P_0f_1\|^2_{L^2_{v}(L^\infty_x)}, \label{G_4b}\\
\Dt \mathcal{H}_{N,k}(U)+\mu \mathbb{D}_{N,k}(U)&\le C\|\dx P_0f_1\|^2_{L^2_{x,v}}+C\|\dx B\|^2_{L^2_{x}} +C \mathbb{E}_N(U)\| P_0f_1\|^2_{L^2_{v}(L^\infty_x)}.\label{G_4a}
\ema
\end{lem}

\subsection{The pointwise estimate}

The following Lemmas \ref{S_1}--\ref{wc-1} are  about estimates on the wave coupling of linear interactions and nonlinear interactions.

\begin{lem} \label{S_1}
For any given $\beta\ge 0$, there exists a constant $C>0$ such that
\be
\| S^tg_0(x)\|_{L^\infty_{v,\beta}}\le Ce^{-\frac{2\nu_0t}3}\max_{y\in \R}e^{-\frac{\nu_0|x-y|}3}\|g_0(y)\|_{L^\infty_{v,\beta}}, \label{s3a}
\ee
where $S^t$  and $\nu_0$ are defined by \eqref{S-t} and \eqref{nuv}. In particular, if $g_0(x,v)$  satisfies
$$
 \| g_0(x)\|_{L^\infty_{v,\beta}}\le C(1+|x|^2)^{-\frac{\gamma}2},\quad \forall \gamma\ge 0,
$$
then we have
\bma
\| S^tg_0(x)\|_{L^\infty_{v,\beta}}&\le Ce^{-\frac{2\nu_0t}3}(1+|x|^2)^{-\frac{\gamma}2},\label{wt1}\\
\|W_{1}(t)\ast g_0(x)\|_{L^\infty_{v,\beta}}&\le Ce^{-\frac{\nu_0t}3}(1+|x|^2)^{-\frac{\gamma}2}, \label{wt}
\ema
where  $W_{1}(t,x)$ is defined by \eqref{W-1}.
\end{lem}

\begin{proof}
We first prove \eqref{s3a}. Recall 
$$ S^tg_0(x,v)=e^{-\nu(v)t}g_0(x-v_1t,v).$$
Let $y=x-v_1t$. By \eqref{nuv}, we have
$$\nu_0|x-y|=\nu_0|v_1|t\le \nu(v)t.$$
Thus
\bma
| S^tg_0(x,v)|&\le e^{-\frac{2\nu(v)t}3}e^{-\frac{\nu_0|x-y|}3} |g_0(y,v)| \nnm\\
&\le e^{-\frac{2\nu_0t}3}(1+|v|)^{-\beta}\max_{y\in \R}e^{-\frac{\nu_0|x-y|}3}\|g_0(y)\|_{L^\infty_{v,\beta}}, \label{bb2}
\ema
which gives \eqref{s3a}.

To prove \eqref{wt}, by \eqref{W-1}, we have
$$W_{1}(t)\ast g_0(x)=\sum^{6}_{k=0} \M^t_{b,k} g_0(x),$$
where $\hat{\M}_{b,k}^t(\xi)$ is defined by \eqref{mix4}.
By \eqref{s3a}, we have
\be \|\M^t_{b,0} g_0(x)\|_{L^\infty_{v,\beta}}=\|S^tg_0(x)\|_{L^\infty_{v,\beta}}\le Ce^{-\frac{2\nu_0t}3}(1+|x|^2)^{-\frac{\gamma}2}. \label{J_2d}\ee
Since
$$\|Kg\|_{L^\infty_{v,\beta}}\le C\|g\|_{L^\infty_{v,\beta}},\quad \forall g\in L^2_v,$$
it follows from \eqref{J_2d} that
\bmas
\|\M^t_{b,1}  g_0(x)\|_{L^\infty_{v,\beta}}
&\le C\intt  e^{-\frac{2\nu_0(t-s)}3}\max_{y\in \R}e^{-\frac{\nu_0|x-y|}3}\|S^sg_0\|_{L^\infty_{v,\beta}} ds\\
&\le Cte^{-\frac{2\nu_0t}3}(1+|x|^2)^{-\frac{\gamma}2}.
\emas
By induction, we  obtain
\be
\| \M^t_{b,k} g_0(x)\|_{L^\infty_{v,\beta}}\le C_kt^ke^{-\frac{2\nu_0t}3}(1+|x|^2)^{-\frac{\gamma}2},\quad \forall\ k\ge 1, \label{J_k}
\ee
 which gives \eqref{wt}. The proof of the lemma is completed.
\end{proof}

\begin{lem}\label{green3a} Given  $\alpha,\beta,\gamma\ge 0$, if the function $F(t,x,v)$  satisfies
$$
 \| F(t,x)\|_{L^\infty_{v,\beta-1}}\le C(1+t)^{-\frac{\alpha}2}B_{\frac{\gamma}2}(t,x-\lambda t), \quad \lambda \in \R,
$$
then we have
\bma
\bigg\|\intt  S^{t-s}F(s,x)ds\bigg\|_{L^\infty_{v,\beta}}&\le C(1+t)^{-\frac{\alpha}2}B_{\frac{\gamma}2}(t,x-\lambda t), \label{s4}\\
\bigg\|\intt  W_{1}(t-s)\ast F(s,x)ds\bigg\|_{L^\infty_{v,\beta}}&\le C(1+t)^{-\frac{\alpha}2}B_{\frac{\gamma}2}(t,x-\lambda t), \label{s4a}
\ema
where $C>0$  is a constant, and $W_{1}(t,x)$ is defined by \eqref{W-1}.
\end{lem}

\begin{proof} By \eqref{bb2}, we have
\bmas
&\quad\nu(v)^{\beta}\intt | S^{t-s}F(s,x,v)|ds\\
&\le C\intt e^{-\frac{2\nu(v)(t-s)}3}\nu(v) e^{-\frac{\nu_0|x-y|}3}\|F(s,y)\|_{L^\infty_{v,\beta-1}} ds\\
&\le C\intt e^{-\frac{2\nu(v)(t-s)}3}\nu(v)(1+s)^{-\frac{\alpha}2} e^{-\frac{\nu_0|x-y|}3}B_{\frac{\gamma}2}(s,y-\lambda s) ds.
\emas
Since
\bmas
 e^{-\frac{\nu_0(t-s)}3}e^{-\frac{\nu_0|x-y|}{3}}B_{\frac{\gamma}2}(s,y-\lambda s)&\le e^{-\frac{|x-y-\lambda(t- s)|}{3D}}B_{\frac{\gamma}2}(s,y-\lambda s)\\
 &\le e^{-\frac{|x- \lambda t|}{6D}}+CB_{\frac{\gamma}2}(t,x-\lambda t),
 \emas
where $D=\nu_0^{-1}\max\{1,|\lambda|\}$, we have
\bmas
&\quad\nu(v)^{\beta}\intt | S^{t-s}F_1(s,x,v)|ds\nnm\\
&\le C\intt e^{-\frac{ \nu(v)(t-s)}3}\nu(v)(1+s)^{-\frac{\alpha}2}ds \(e^{-\frac{|x-\lambda t|}{6D}}+B_{\frac{\gamma}2}(t,x-\lambda t)\) \\
&\le C(1+t)^{-\frac{\alpha}2} \(e^{-\frac{|x-\lambda t|}{6D}}+B_{\frac{\gamma}2}(t,x-\lambda t)\),
\emas
which gives \eqref{s4}. Similarly,
\bmas
&\quad\nu(v)^{\beta}\intt |\M^{t-s}_{b,1}  F(s,x,v)|ds\\
&\le C\intt e^{-\frac{2\nu(v)(t-s)}3} e^{-\frac{\nu_0|x-y|}3}\int^s_0\|KS^{s-s_1}F(s_1,y)\|_{L^\infty_{v,\beta}}ds_1 ds\\
&\le C\intt e^{-\frac{2\nu(v)(t-s)}3}(1+s)^{-\frac{\alpha}2} e^{-\frac{\nu_0|x-y|}3}B_{\frac{\gamma}2}(s,y-\lambda s) ds\\
&\le C(1+t)^{-\frac{\alpha}2} B_{\frac{\gamma}2}(t,x-\lambda t).
\emas
By induction, we  obtain
$$
\bigg\|\intt \M^{t-s}_{b,k}  F(s,x)ds\bigg\|_{L^\infty_{v,\beta}}\le C_k(1+t)^{-\frac{\alpha}2} B_{\frac{\gamma}2}(t,x-\lambda t),\quad \forall\ k\ge 1,
$$
 which gives \eqref{s4a}.
\end{proof}

\begin{lem} \label{S_2}
Let $\beta,\gamma\ge 0$. If $V_0(x)=(g_0(x,v),Y^2_0(x),Y^3_0(x))\in \mathcal{X}_1$  satisfies
$$
 \| g_0(x)\|_{L^\infty_{v,\beta}}+|Y^2_0(x)|+|Y^3_0(x)|\le C(1+|x|^2)^{-\frac{\gamma}2},
$$
then
\be
\| W_{0,1}(t)\ast V_0(x)\|_{Y^\infty_{\beta}}\le Ce^{-\frac{\nu_0t}3}(1+|x|^2)^{-\frac{\gamma}2} , \label{s3b}
\ee
where $C>0$ is a constant, and $W_{0,1}(t,x)$ is defined by \eqref{W_0}.
\end{lem}

\begin{proof} For $n\ge 1$, define $V_n(t,x)=(F_n(t,x,v),Z^2_n(t,x),Z^3_n(t,x))\in L^2_v\times \R^3\times \R^3$ as
$$V_n(t,x)= \U_{n-1,1}(t)\ast V_0(x),$$
where $\U_{k,1}(t,x)$, $k\ge 0$ is defined by \eqref{w6}. 

From \eqref{U_5}, \eqref{U-1} and \eqref{w6}, we have
\be\label{V_1}
\left\{\bln
&F_n = J_n+H_{n},\quad Z^2_n=(Z^1_n,\tilde{Z}^2_n),\quad Z^3_n=(0,\tilde{Z}^3_n),\\
&(J_n,\tilde{Z}^2_n,\tilde{Z}^3_n)^T= \Q^t_{n-1,1}\ast (g_0,\tilde{Y}^2_0,\tilde{Y}^3_0)^T,\\
&H_1\equiv0, \quad  Z^1_1  =(\nu_0+\dx)^{-1} (F_1,\chi_0),\\
&H_{n+1}= \sum^{n-1}_{k=0}\intt \M^{t-s}_{k} Z^1_{n-k}v_1\chi_0ds,
\\
&Z^1_{n+1} =(\nu_0+\dx)^{-1}\[(F_{n+1},\chi_0)+\nu_0 Z^1_{n}\] ,
\eln\right.
\ee
where $n\ge 1$, $(\tilde{Y}^2_0,\tilde{Y}^3_0), (\tilde{Z}^2_n,\tilde{Z}^3_n)\in \R^2\times \R^2$, and  $\Q^t_{k,1}(t,x)$, $k\ge 0$ is defined by \eqref{Q_3}.

Let $\tilde{Z}_n(t,x)=(\tilde{Z}^2_n,\tilde{Z}^3_n)(t,x) $ for $n\ge 1$. By \eqref{Q_0} and \eqref{V_1}, we obtain
\be
F_{1}(t,x)=S^tg_0,
\quad  \tilde{Z}_1(t,x)\equiv0. \label{V_3}
\ee
From Lemma \ref{S_1}, we have
\be \|F_{1}(t,x)\|_{L^\infty_{v,\beta}}=\|S^tg_0(x)\|_{L^\infty_{v,\beta}}  \le Ce^{-\frac{2\nu_0t}3}(1+|x|^2)^{-\frac{\gamma}2}. \label{V_3a}\ee
Since
\be \mathcal{F}^{-1}\(\frac{1}{\nu_0+i\xi}\)=Ce^{-\nu_0x}1_{\{x>0\}}, \label{V_3d}\ee
it follows that
\bma
|Z^1_1(t,x)|&=\left|(\nu_0+\dx)^{-1}(F_1(t,x),\chi_0)\right| \nnm\\
&\le C e^{-\frac{2\nu_0t}3}\int^{\infty}_0 e^{-\nu_0y}(1+|x-y|^2)^{-\frac{\gamma}2}dy\le Ce^{-\frac{2\nu_0t}3}(1+|x|^2)^{-\frac{\gamma}2}.\label{V_3b}
\ema
Combining   \eqref{V_3}--\eqref{V_3b}, we obtain
\be
\|V_1(t,x)\|_{Y^\infty_{\beta}}\le \|F_1(t,x)\|_{L^\infty_{v,\beta}}+ |Z^1_1(t,x) |+ |\tilde{Z}_1(t,x) |
\le Ce^{-\frac{2\nu_0t}3}(1+|x|^2)^{-\frac{\gamma}2}.
\ee

By \eqref{Q_1} and \eqref{V_1}, we have
$$
\left\{\bln
&F_2(t,x)=\M^t_1g_0+\sum^4_{j=1}S^tR^0_j(\tilde{Y}_0\cdot\tilde{\X}_j)+\intt S^{t-s}v_1\chi_0Z^1_1 ds,
\\
&\tilde{Z}_2(t,x)=\sum^4_{j=1}\tilde{\X}_j \intr S^tR^0_jg_0dv,
\eln\right.
$$
where $\tilde{Y}_0(x)=(\tilde{Y}^2_0,\tilde{Y}^3_0)(x)\in \R^2\times \R^2$, 
and $R^0_j$ is an operator in $L^2_v$ given by
$$
\hat{R}^0_j=u^0_j=(\nu(v)+iv_1\xi+\alpha_j)^{-1}r_j,\quad r_j=(v\cdot \X^2_j)\chi_0.
$$
For any $h_0=h_0(x,v)$,  $R^0_jh_0$ has the expression
\be
R^0_jh_0(x,v)=\int^{\infty}_0e^{-\nu(v)t}h_0(x-(v_1+d_j)t,v)r_jdt, \label{R-0}
\ee
where $d_1=d_2=1,\ d_3=d_4=-1$. Let $y=x-(v_1+d_j)t$. By \eqref{nuv}, we have
$$\nu_0|x-y|\le\nu_0(|v_1|+1)t\le \nu(v)t.$$
Thus, for any $\beta\ge 0$,
$$
|R^0_jh_0(x,v) |\le C\int^{\infty}_0e^{-\frac{2\nu_0t}3} (1+|v|)^{-\beta}e^{-\frac{\nu_0|x-y|}3}\|h_0(y,v)|v|\chi_0\|_{L^\infty_{v,\beta}}dt,
$$
which gives
\be
\|R^0_jh_0(x) \|_{L^\infty_{v,\beta}}\le C \max_{y\in\R} e^{-\frac{\nu_0|x-y|}3}\|h_0(y,v)|v|\chi_0\|_{L^\infty_{v,\beta}} . \label{V_4}
\ee
Thus
$$
\|S^tR^0_j(\tilde{Y}_0\cdot\tilde{\X}_j)\|_{L^\infty_{v,\beta}}+\|S^tR^0_jg_0\|_{L^\infty_{v,\beta}}\le Ce^{-\frac{2\nu_0t}3}(1+|x|^2)^{-\frac{\gamma}2},
$$
which together with \eqref{J_k} gives
\be
\|F_2(t,x)\|_{L^\infty_{v,\beta}}+|\tilde{Z}_2(t,x)|\le C(1+t)e^{-\frac{2\nu_0t}3}(1+|x|^2)^{-\frac{\gamma}2}. \label{V_4a}
\ee
Moreover, it follows from \eqref{V_4a}, \eqref{V_3b} and \eqref{V_3d} that
\be
|Z^1_2(t,x)|=\left|(\nu_0+\dx)^{-1}[(F_2,\chi_0)+Z^1_1]\right| \le C(1+t)e^{-\frac{2\nu_0t}3}(1+|x|^2)^{-\frac{\gamma}2}.\label{V_3c}
\ee
Thus, by  \eqref{V_4a} and \eqref{V_3c}, we obtain
\be
\|V_2(t,x)\|_{Y^\infty_{\beta}} \le C(1+t)e^{-\frac{2\nu_0t}3}(1+|x|^2)^{-\frac{\gamma}2}. \label{V_2b}
\ee
Furthermore, by \eqref{Q_2} and \eqref{V_1}, we have
$$
\left\{\bln
&F_3(t,x)=\M^t_2g_0+\sum^4_{j=1}S^tR^0_j\intr R^0_jg_0dv+\sum^4_{j=1}(\M^t_1R^0_j+S^tR^1_j)(\tilde{Y}_0\cdot\tilde{\X}_j)\\
&\qquad\qquad+\sum^4_{j=1}\intt S^{t-s}r_j \intr S^sR^0_jg_0dv ds+\sum^1_{k=0}\intt  \M^{t-s}_kv_1\chi_0Z^1_{2-k} ds,
\\
&\tilde{Z}_3(t,x)=\sum^4_{j=1}\tilde{\X}_j\bigg\{\intr (\M^t_1R^0_j+S^tR^1_j)g_0dv +\intr S^tR^0_j R^0_{l_j} (\tilde{Y}_0\cdot\tilde{\X}_{l_j})dv\bigg\},
\eln\right.
$$
where $R^1_j$ is an operator in $L^2_v$ given by
$$
\hat{R}^1_j=u^1_j=(\nu(v)+iv_1\xi+\alpha_j)^{-1}K_1u^0_j.
$$
By using a similar argument as for proving \eqref{V_2b}, we obtain
\be
\|V_3(t,x)\|_{Y^\infty_{\beta}}\le C(1+t)^2e^{-\frac{2\nu_0t}3}(1+|x|^2)^{-\frac{\gamma}2}.
\ee
Similarly, by \eqref{Q_3a} and \eqref{Q_4b}, one can show that
$$
\|V_n(t,x)\|_{Y^\infty_{\beta}} \le C_n(1+t)^{n-1}e^{-\frac{2\nu_0t}3}(1+|x|^2)^{-\frac{\gamma}2},\quad n\ge 4.
$$
Thus  \eqref{s3b} holds and this completes the proof of the lemma.
\end{proof}

\begin{lem}\label{green3b} Given $\alpha,\beta,\gamma\ge 0$, if $V(t,x)=(F(t,x,v),0,0)\in \mathcal{X}_1$  satisfies $(F(t,x),\chi_0)=0$ and
$$
 \|V(t,x)\|_{Y^\infty_{\beta-1}}\le (1+t)^{-\frac{\alpha}2} B_{\frac{\gamma}2}(t,x-\lambda t),\\
$$
then we have
\be
\bigg\|\intt  W_{0,1}(t-s)\ast V(s,x) ds\bigg\|_{Y^\infty_{\beta}}\le C(1+t)^{-\frac{\alpha}2} B_{\frac{\gamma}2}(t,x-\lambda t),
\ee
where $C>0$ is a constant, and $W_{0,1}(t,x)$ is defined by \eqref{W_0}.
\end{lem}

\begin{proof}For $n\ge 1$, denote $V_n(t,x)=(F_n(t,x,v),Z^2_n(t,x),Z^3_n(t,x))\in L^2_v\times \R^3\times \R^3$ as
$$V_n(t,x) =\intt  \U_{n-1,1}(t-s)\ast V (s,x)ds.$$

By  \eqref{V_1}, we have
$$
F_{1}(t,x)=\intt S^{t-s}F(s)ds,
\quad  \tilde{Z}_1(t,x)\equiv0,
$$
$$
\left\{\bln
&F_2(t,x)=\intt \M^{t-s}_1F(s)ds +\intt S^{t-s}v_1\chi_0Z^1_1 ds,
\\
&\tilde{Z}_2(t,x)=\sum^4_{j=1}\tilde{\X}_j\intt \intr S^{t-s}R^0_j F(s)dvds,
\eln\right.
$$
and
$$
\left\{\bln
&F_3(t,x)=\intt \M^{t-s}_2F(s)ds+\sum^4_{j=1}\intt S^{t-s}R^0_j\intr R^0_j F(s)dv ds \\
&\qquad\qquad+\sum^4_{j=1}\intt\int^s_0 S^{t-s}r_j \intr S^{s-s_1}R^0_j F(s_1)dvds_1ds\\
&\qquad\qquad+\sum^1_{j=0}\intt  \M^{t-s}_jv_1\chi_0Z^1_{2-j} ds,
\\
&\tilde{Z}_3(t,x)=\sum^4_{j=1}\tilde{\X}_j \intt\intr ( \M^{t-s}_1R^0_j+S^{t-s}R^1_j)F(s)dvds ,
\eln\right.
$$ where $\tilde{Z}_n(t,x)=(\tilde{Z}^2_n,\tilde{Z}^3_n)(t,x)$ $(n\ge 1)$.
By using a similar argument as for Lemma \ref{S_2}, we can prove the lemma. Here we omit the detail
for brevity.
\end{proof}

Let $\Pi_{\alpha}(t,x;\lambda)\ge 0$ be a function supported in $|x-\lambda  t|\le 1$ and satisfies
\be \intra \Pi_{\alpha}(t,x;\lambda)dx \le C(1+t)^{-\frac{\alpha}2}. \label{H-a}\ee

\begin{lem}\label{p-1}For any constants $D>0$, $\lambda\in \R$ and $\gamma>1$, we have
\bma
 \intra\frac{e^{-\frac{|x-y-\lambda t|^2}{D(1+t)}}}{(1+t)^{\frac{\alpha}2}}(1+|y|^2)^{-\frac{\gamma}2} dy
  \le& C(1+t)^{-\frac{\alpha}2} B_{\frac{\gamma}2}(t,x-\lambda t),\label{wave3a}\\
   \intra\frac{e^{-\frac{\nu_0|x-y-\lambda t|}{2}}}{ (1+t)^{\frac{\alpha}2}}(1+|y|^2)^{-\frac{\gamma}2} dy
  \le& C(1+t)^{-\frac{\alpha}2} 
  B_{\frac{\gamma}2}(t,x-\lambda t),\label{wave3b}\\
   \intra \Pi_{\alpha}(t,x-y;\lambda) (1+|y|^2)^{-\frac{\gamma}2} dy
  \le& C(1+t)^{-\frac{\alpha}2} 
  B_{\frac{\gamma}2}(t,x-\lambda t),\label{wave3c}
  \ema
where  $C>$ is a constant.
\end{lem}

\begin{proof}
For $|x-\lambda t|^2\le 1+t$,
\bmas
\intra \frac{e^{-\frac{|x-y-\lambda t|^2}{D(1+t)}}}{(1+t)^{\frac{\alpha}2}}(1+|y|^2)^{-\frac{\gamma}2} dy
\le C(1+t)^{-\frac{\alpha}2}.
\emas
For $|x-\lambda t|^2\ge 1+t$,
\bmas
 \intra \frac{e^{-\frac{|x-y-\lambda t|^2}{D(1+t)}}}{(1+t)^{\frac{\alpha}2}}(1+|y|^2)^{-\frac{\gamma}2} dy
 \le&\(\int_{|y|\le \frac{|x-\lambda t|}{2}}+ \int_{|y|\ge \frac{|x-\lambda t|}{2}}\) \frac{e^{-\frac{|x-y-\lambda t|^2}{D(1+t)}}}{(1+t)^{\frac{\alpha}2}}(1+|y|^2)^{-\frac{\gamma}2}dy \\
 \le& C(1+t)^{-\frac{\alpha}2} \(e^{-\frac{|x-\lambda t|^2}{4D(1+t)}}+B_{\frac{\gamma}2}(t,x-\lambda t)\).
  \emas
Thus, we obtain \eqref{wave3a}.  \eqref{wave3b} and \eqref{wave3c} can be proved similarly.
\end{proof}

We now consider the following integrals for estimating the nonlinear interactions:
\bma
&I^{\alpha,\beta,\gamma}(t,x;t_1,t_2;\lambda,\mu,D) \nnm\\
=&\int^{t_2}_{t_1} \intra (1+t-s)^{-\frac{\alpha}2}e^{-\frac{|x-y-\lambda(t-s)|^2}{D(1+t-s)}}(1+s)^{-\frac{\beta}2} B_{\frac{\gamma}2}(s,y-\mu s)dyds, \label{Iab}\\
&J^{\alpha,\beta,\gamma}(t,x;t_1,t_2;\lambda,\mu)\nnm\\
=&\int^{t_2}_{t_1} \intra (1+t-s)^{-\frac{\alpha}2}e^{-\frac{\nu_0|x-y-\lambda(t-s)|}{2}}(1+s)^{-\frac{\beta}2} B_{\frac{\gamma}2}(s,y-\mu s) dyds,\label{Jab}\\
&K^{\alpha, \beta,\gamma}(t,x;t_1,t_2;\lambda,\mu)\nnm\\
=&\int^{t_2}_{t_1} \intra  \Pi_{\alpha}(t-s,x-y;\lambda)(1+s)^{-\frac{\beta}2} B_{\frac{\gamma}2}(s,y-\mu s) dyds, \label{Kab}\\
&L^{ \beta,\gamma}(t,x;t_1,t_2;\lambda, D)\nnm\\
=&\int^{t_2}_{t_1} \intra  e^{-\frac{|x-y|+t-s}{D}}(1+s)^{-\frac{\beta}2} B_{\frac{\gamma}2}(s,y-\lambda s) dyds. \label{Lab}
\ema

Set
\be
\Gamma_{\alpha}(t)=\intt (1+s)^{-\frac{\alpha}2}ds
=O(1)\left\{\bal (1+t)^{\frac{2-\alpha}2} & \alpha< 2,\\
\ln(1+t) & \alpha=2,\\
1 & \alpha> 2.
\ea\right.
\ee

\begin{lem}\label{green3c}
For $s\in [0,t]$, $A^2\ge 1+t$ and $\alpha\ge 0$,
\bma
   e^{-\frac{A^2}{D(1+s)}}&\le C_\alpha\(\frac{1+s}{1+t}\)^{-\alpha} e^{-\frac{A^2}{D(1+t)}}, \label{a}\\
   \(1+\frac{A^2}{1+s}\)^{-\alpha}&\le 2^\alpha\(\frac{1+t}{1+s}\)^{-\alpha} \(1+\frac{A^2}{1+t}\)^{-\alpha},\label{b}
  \ema
where $D,C_\alpha>0$ are two constants.
\end{lem}
\begin{proof}
Since
$$
e^{\frac{(t-s)A^2}{D(1+t)(1+s)}}\ge C_\alpha\(1+\frac{(t-s) A^2}{(1+t)(1+s)}\)^\alpha\ge C_\alpha\(1+\frac{t-s}{1+s}\)^\alpha=C_\alpha\(\frac{1+t}{1+s}\)^\alpha,
$$
it follows that
$$
\(\frac{1+s}{1+t}\)^\alpha e^{\frac{A^2}{D(1+s)}-\frac{A^2}{D(1+t)}}=\(\frac{1+s}{1+t}\)^\alpha e^{\frac{(t-s)A^2}{D(1+t)(1+s)}}\ge C_\alpha,
$$
which gives \eqref{a}.

Since
$$
1+\frac{A^2}{1+s}\ge \frac{1+t}{1+s}\frac{A^2}{1+t}\ge \frac12\(\frac{1+t}{1+s}\)\(1+\frac{A^2}{1+t}\),
$$
it follows that
$$
\(1+\frac{A^2}{1+s}\)^{-\alpha}\le 2^\alpha\(\frac{1+t}{1+s}\)^{-\alpha}\(1+\frac{A^2}{1+t}\)^{-\alpha},
$$
which gives \eqref{b}.
\end{proof}

\begin{lem}\label{couple1a} Assume $\alpha,\beta\ge 1,\gamma>1$, and $\lambda< \mu $. There exists a constant $C>0$ such that
\bma
 &\quad I^{\alpha,\beta,\gamma}(t,x;0,t;\lambda,\lambda;D)\nnm\\
  &\le C  \[(1+t)^{-\frac{\alpha}2}\Gamma_{\beta-1}(t)+(1+t)^{-\frac{\beta}2}\Gamma_{\alpha-1}(t)\]B_{\frac{\gamma}2}(t,x-\lambda t), \label{C}\\
   &\quad I^{\alpha,\beta,\gamma}(t,x;0,t;\lambda,\mu;D)\nnm\\
  &\le C\[(1+t)^{-\frac{\alpha}2}\Gamma_{\beta-1}(t)+(1+t)^{-\frac{\beta}2}\Gamma_{\alpha-1}(t)\]\nnm\\
  &\quad\times\(B_{\frac{\gamma}2}(t,x-\lambda t)+B_{\frac{\gamma}2}(t,x-\mu t)\)\nnm\\
  &\quad +C(1+t)^{-\frac{\alpha-1}2} (1+x-\lambda t)^{-\frac{\beta-1}2}1_{\{\lambda t+\sqrt{1+t}\le x\le \mu t-\sqrt{1+t}\}}\nnm\\
   &\quad +C(1+t)^{-\frac{\beta-1}2}(1+\mu t-x)^{-\frac{\alpha-1}2}1_{\{\lambda t+\sqrt{1+t}\le x\le \mu t-\sqrt{1+t}\}}, \label{C_2}
  \ema
where $I^{\alpha,\beta,\gamma}(t,x;t_1,t_2;\lambda,\mu;D)$ is defined by \eqref{Iab}.
\end{lem}

\begin{proof}
For $ |x-\lambda t|\le \sqrt{1+t}$, we have
\bmas
 &\quad I^{\alpha,\beta,\gamma}(t,x;0,t;\lambda,\lambda;D)\nnm\\
&=\intt\intra (1+t-s)^{-\frac{\alpha}2}e^{-\frac{|x-y-\lambda(t-s)|^2}{D(1+t-s)}}(1+s)^{-\frac{\beta}2}B_{\frac{\gamma}2}(s,y-\lambda s)dyds \\
&\le  C(1+t)^{-\frac{\alpha}2}\int^{t/2}_0  (1+s)^{-\frac{\beta-1}2}ds +C(1+t)^{-\frac{\beta}2}\int^t_{t/2} (1+t-s)^{-\frac{\alpha-1}2}ds  \\
&\le C  \[(1+t)^{-\frac{\alpha}2}\Gamma_{\beta-1}(t)+(1+t)^{-\frac{\beta}2}\Gamma_{\alpha-1}(t)\].
\emas
For $ |x-\lambda t|\ge \sqrt{1+t}$, we have
\bmas
&e^{-\frac{|x-y-\lambda(t-s)|^2}{D(1+t-s)}} B_{\frac{\gamma}2}(s,y-\lambda s)\\
\le& e^{-\frac{|x-\lambda t|^2}{4D(1+t-s)}}B_{\frac{\gamma}2}(s,y-\lambda s) + Ce^{-\frac{|x-y-\lambda(t-s)|^2}{D(1+t-s)}}B_{\frac{\gamma}2}(s,x-\lambda t) ,
\emas
which together with Lemma \ref{green3c} gives
\bmas
 &\quad I^{\alpha,\beta,\gamma}(t,x;0,t;\lambda,\lambda;D)\nnm\\
 &\le C(1+t)^{-\frac{1}2}e^{-\frac{|x-\lambda t|^2}{4D(1+t)}}\intt (1+t-s)^{-\frac{\alpha-1}2}(1+s)^{-\frac{\beta-1}2}ds\\
 &\quad + C(1+t)^{-\frac{1}2}B_{\frac{\gamma}2}(t,x-\lambda t)\intt (1+t-s)^{-\frac{\alpha-1}2}(1+s)^{-\frac{\beta-1}2} ds\\
&\le C  \[(1+t)^{-\frac{\alpha}2}\Gamma_{\beta-1}(t)+(1+t)^{-\frac{\beta}2}\Gamma_{\alpha-1}(t)\]B_{\frac{\gamma}2}(t,x-\lambda t).
\emas
Thus, we obtain \eqref{C}.

For $|x-\lambda t|\le \sqrt{1+t}$ or $|x-\mu t|\le \sqrt{1+t}$, we have
\bmas
&I^{\alpha,\beta,\gamma}(t,x;0,t;\lambda,\mu;D) \\
=&\intt\intra (1+t-s)^{-\frac{\alpha}2}e^{-\frac{|x-y-\lambda(t-s)|^2}{D(1+t-s)}}(1+s)^{-\frac{\beta}2}B_{\frac{\gamma}2}(s,y-\mu s)dyds \\
\le&  C(1+t)^{-\frac{\alpha}2}\int^{t/2}_0  (1+s)^{-\frac{\beta-1}2}ds +C(1+t)^{-\frac{\beta}2}\int^t_{t/2} (1+t-s)^{-\frac{\alpha-1}2}ds  \\
\le& C \[(1+t)^{-\frac{\alpha}2}\Gamma_{\beta-1}(t)+(1+t)^{-\frac{\beta}2}\Gamma_{\alpha-1}(t)\].
\emas

For $\lambda t+\sqrt{1+t}\le x\le \mu t-\sqrt{1+t}$, we write
\bmas
&\quad I^{\alpha,\beta,\gamma}(t,x;0,t;\lambda,\mu;D) \\
&=  \( \int^{\frac{|x-\lambda t|}{A}}_{0}+\int^{t- \frac{|x-\mu t|}{A}}_{\frac{|x-\lambda t|}{A}}+\int^t_{t- \frac{|x-\mu t|}{A}}\) \intra \\
&\qquad(1+t-s)^{-\frac{\alpha}2}e^{-\frac{|x-y-\lambda(t-s)|^2}{D(1+t-s)}}(1+s)^{-\frac{\beta}2} B_{\frac{\gamma}2}(s,y-\mu s)dyds\\
&=:I_1+I_2+I_3,
\emas
where $A= 4(\mu-\lambda) $ such that $ |x-\lambda t|/A\le t/4$ and $ |x-\mu t|/A\le t/4$.

By splitting $y$ into two case:
$$ \{|y-\mu s|\le \frac12|x-\lambda t+(\lambda-\mu)s|\}, \quad \{|y-\mu s|\ge  \frac12|x-\lambda t+(\lambda-\mu)s|\}, $$
we have
\bma
&e^{-\frac{|x-y-\lambda(t-s)|^2}{D(1+t-s)}} B_{\frac{\gamma}2}(s,y-\mu s)\nnm\\
\le& e^{-\frac{|x-\lambda t+(\lambda-\mu)s|^2}{4D(1+t-s)}}B_{\frac{\gamma}2}(s,y-\mu s)+Ce^{-\frac{|x-y-\lambda(t-s)|^2}{D(1+t-s)}}B_{\frac{\gamma}2}(s,x-\lambda t+(\lambda-\mu)s). \label{aa}
\ema
This together with Lemma \ref{green3c} implies that
\bmas
I_1&\le C(1+t)^{-\frac{1}2} e^{-\frac{|x-\lambda t|^2}{8D(1+t)}}\int^{\frac{x-\lambda t}{A}}_{0} (1+t-s)^{-\frac{\alpha-1}2}(1+s)^{-\frac{\beta-1}2}  ds\\
&\quad+C(1+t)^{-\frac{1}2}B_{\frac{\gamma}2}(t,x-\lambda t) \int^{\frac{x-\lambda t}{A}}_{0} (1+t-s)^{-\frac{\alpha-1}2}(1+s)^{-\frac{\beta-1}2} ds\\
&\le C(1+t)^{-\frac{\alpha}2}\Gamma_{\beta-1}(t)B_{\frac{\gamma}2}(t,x-\lambda t),\\
I_3&\le C(1+t)^{-\frac{1}2} e^{-\frac{|x-\mu t|^2}{ 8D(1+t)}}\int^t_{t-\frac{\mu t-x}{A}} (1+t-s)^{-\frac{\alpha-1}2}(1+s)^{-\frac{\beta-1}2} ds\\
&\quad+C(1+t)^{-\frac{1}2} B_{\frac{\gamma}2}(t,x-\mu t) \int^t_{t-\frac{\mu t-x}{A}} (1+t-s)^{-\frac{\alpha-1}2}(1+s)^{-\frac{\beta-1}2} ds\\
&\le C(1+t)^{-\frac{\beta}2}\Gamma_{\alpha-1}(t)B_{\frac{\gamma}2}(t,x-\mu t).
\emas
We decompose $I_2$ into
\bmas
I_2&\le  C\int^{t/2}_{\frac{x-\lambda t}{A}}  (1+t-s)^{-\frac{\alpha}2}  (1+s)^{-\frac{\beta-1}2}e^{-\frac{|x-\lambda t+(\lambda-\mu)s|^2}{4D(1+t-s)}}ds\\
 &\quad +C\int^{t/2}_{\frac{x-\lambda t}{A}}(1+t-s)^{-\frac{\alpha-1}2}  (1+s)^{-\frac{\beta}2}B_{\frac{\gamma}2}(s,x-\lambda t+(\lambda-\mu)s) ds\\
  &\quad +C\int^{t- \frac{\mu t-x}{A}}_{t/2}  (1+t-s)^{-\frac{\alpha}2}  (1+s)^{-\frac{\beta-1}2}e^{-\frac{|x-\lambda t+(\lambda-\mu)s|^2}{4D(1+t-s)}}ds\\
&\quad +C\int^{t- \frac{\mu t-x}{A}}_{t/2} (1+t-s)^{-\frac{\alpha-1}2}  (1+s)^{-\frac{\beta}2}B_{\frac{\gamma}2}(s,x-\lambda t+(\lambda-\mu)s) ds\\
&=:I_{2,1}+I_{2,2}+I_{2,3}+I_{2,4}.
 \emas
It holds that
\bmas
I_{2,1}&\le  C(1+t)^{-\frac{\alpha}2} (1+x-\lambda t)^{-\frac{\beta-1}2}\int^{t/2}_{\frac{x-\lambda t}{A}}  e^{-\frac{|x-\lambda t+(\lambda-\mu)s|^2}{4Dt}}ds\\
&\le C(1+t)^{-\frac{\alpha-1}2}(1+x-\lambda t)^{-\frac{\beta-1}2},\\
I_{2,4}&\le C(1+t)^{-\frac{\beta}2}(1+\mu t-x)^{-\frac{\alpha-1}2}\int^{t- \frac{\mu t-x}{A}}_{t/2}  B_{\frac{\gamma}2}(t,x-\lambda t+(\lambda-\mu)s) ds\\
&\le C(1+t)^{-\frac{\beta-1}2}(1+\mu t-x)^{-\frac{\alpha-1}2},
\emas
where we have used
\bma
\intt B_{\frac{\gamma}2}(t,x-\lambda t+(\lambda-\mu)s)ds&=\frac{ \sqrt{1+t}}{\lambda-\mu}\int^{\frac{x-\lambda t}{\sqrt{1+t}}+(\lambda-\mu)\sqrt{1+t}}_{\frac{x-\lambda t}{\sqrt{1+t}}}(1+z^2)^{-\frac{\gamma}2}dz \nnm\\
&\le C\sqrt{1+t}\Gamma_{2\gamma}(\sqrt{1+t}). \label{bg}
\ema

Suppose that $x-\lambda t\le \mu t-x$. For $t/2\le s\le t-(\mu t-x)/A$ ,
it holds that $ t-s \ge (\mu t-x)/A$. Hence
\bmas
I_{2,3}\le&C(1+t)^{-\frac{\beta-1}2}(1+\mu t-x)^{-\frac{\alpha}2}\int^{t- \frac{\mu t-x}{A}}_{t/2}  e^{-\frac{A|x-\lambda t+(\lambda-\mu)s|^2}{4D(1+\mu t-x)}}ds\\
\le &C(1+t)^{-\frac{\beta-1}2}(1+\mu t-x)^{-\frac{\alpha-1}2}.
\emas
Similarly,
\bmas
I_{2,2}\le&C(1+t)^{-\frac{\alpha-1}2}  (1+x-\lambda t)^{-\frac{\beta}2}\int^{t/2}_{\frac{x-\lambda t}{A}}B_{\frac{\gamma}2}(x-\lambda t,x-\lambda t+(\lambda-\mu)s) ds\\
\le &C(1+t)^{-\frac{\alpha-1}2} (1+x-\lambda t)^{-\frac{\beta-1}2}.
\emas
Thus
\be
I_2\le C(1+t)^{-\frac{\alpha-1}2} (1+x-\lambda t)^{-\frac{\beta-1}2}+C(1+t)^{-\frac{\beta-1}2}(1+\mu t-x)^{-\frac{\alpha-1}2}.
 \ee

For  $x\le \lambda t-\sqrt{1+t}$ or $x\ge \mu t+\sqrt{1+t}$, it holds that
\be \label{aa1}
|x-\lambda t+(\lambda-\mu)s|\ge
\left\{\bln
&|x-\lambda t| ,\quad x\le \lambda t-\sqrt{1+t},\\
&|x-\mu t|,  \quad x\ge \mu t+\sqrt{1+t}.
\eln\right.
\ee
Thus, it follows from \eqref{aa} that
\bmas
&e^{-\frac{|x-y-\lambda(t-s)|^2}{D(1+t-s)}} B_{\frac{\gamma}2}(s,y-\mu s)\\
\le& \(e^{-\frac{|x-\lambda t|^2}{4D(1+t-s)}}+e^{-\frac{|\mu t-x|^2}{4D(1+t-s)}}\)B_{\frac{\gamma}2}(s,y-\mu s)\\
&+Ce^{-\frac{|x-y-\lambda(t-s)|^2}{D(1+t-s)}}\(B_{\frac{\gamma}2}(s,x-\lambda t)+B_{\frac{\gamma}2}(s,x-\mu t)\).
\emas
This together with Lemma \ref{green3c} gives
\bmas
&I^{\alpha,\beta,\gamma}(t,x;0,t;\lambda,\mu;D) \\
\le& C \[(1+t)^{-\frac{\alpha}2}\Gamma_{\beta-1}(t)+(1+t)^{-\frac{\beta}2}\Gamma_{\alpha-1}(t)\]\(B_{\frac{\gamma}2}(t,x-\lambda t)+B_{\frac{\gamma}2}(t,x-\mu t)\).
\emas
This completes the proof of the lemma.
\end{proof}

 The following corollary is a direct consequence of the proof of the above lemma.

\begin{cor}\label{couple1b} Assume $\alpha,\beta\ge 1,\gamma>1$, and $\lambda< \mu $. There exists a constant $C>0$ such that
\bma
 &\quad I^{\alpha,\beta,\gamma}(t,x;0,t/2;\lambda,\lambda;D)
  \le C  (1+t)^{-\frac{\alpha}2}\Gamma_{\beta-1}(t) B_{\frac{\gamma}2}(t,x-\lambda t), \label{C1}\\
    &\quad I^{\alpha,\beta,\gamma}(t,x;t/2,t;\lambda,\lambda;D)
  \le C (1+t)^{-\frac{\beta}2}\Gamma_{\alpha-1}(t) B_{\frac{\gamma}2}(t,x-\lambda t), \label{C2}\\
   &\quad I^{\alpha,\beta,\gamma}(t,x;0,t/2;\lambda,\mu;D)\nnm\\
  &\le C (1+t)^{-\frac{\alpha}2}\Gamma_{\beta-1}(t) \(B_{\frac{\gamma}2}(t,x-\lambda t)+B_{\frac{\gamma}2}(t,x-\mu t)\)\nnm\\
  &\quad +C(1+t)^{-\frac{\alpha-1}2}(1+x-\lambda t)^{-\frac{\beta-1}2}1_{\{\lambda t+\sqrt{1+t}\le x\le \mu t-\sqrt{1+t}\}},  \\
   &\quad I^{\alpha,\beta,\gamma}(t,x;t/2,t;\lambda,\mu;D)\nnm\\
  &\le C (1+t)^{-\frac{\beta}2}\Gamma_{\alpha-1}(t) \(B_{\frac{\gamma}2}(t,x-\lambda t)+B_{\frac{\gamma}2}(t,x-\mu t)\)\nnm\\
  &\quad +C(1+t)^{-\frac{\beta-1}2}(1+\mu t-x)^{-\frac{\alpha-1}2}1_{\{\lambda t+\sqrt{1+t}\le x\le \mu t-\sqrt{1+t}\}}. \label{C_2b}
  \ema
\end{cor}

\begin{lem}\label{couple1} Assume $\alpha,\beta\ge 1$, $\gamma\ge 1$, and $\lambda<\mu $. There exists a constant $C>0$ such that
\bma
 &\quad J^{\alpha,\beta,\gamma}(t,x;0,t;\lambda,\lambda)\nnm\\
  &\le C \[(1+t)^{-\frac{\alpha}2}\Gamma_{\beta}(t)+(1+t)^{-\frac{\beta}2}\Gamma_{\alpha}(t)\]B_{\frac{\gamma}2}(t,x-\lambda t), \label{C_3}\\
   &\quad J^{\alpha,\beta,\gamma}(t,x;0,t;\lambda,\mu)\nnm\\
  &\le C \[(1+t)^{-\frac{\alpha}2}\Gamma_{\beta}(t)+(1+t)^{-\frac{\beta}2}\Gamma_{\alpha}(t)\]
\(B_{\frac{\gamma}2}(t,x-\lambda t)+B_{\frac{\gamma}2}(t,x-\mu t)\)\nnm\\
  &\quad +C(1+t)^{-\frac{\alpha-1}2}\Gamma_{2\gamma}(\sqrt{1+t}) (1+x-\lambda t)^{-\frac{\beta}2}1_{\{\lambda t+\sqrt{1+t}\le x\le \mu t-\sqrt{1+t}\}}\nnm\\
   &\quad +C(1+t)^{-\frac{\beta-1}2}\Gamma_{2\gamma}(\sqrt{1+t})(1+\mu t-x)^{-\frac{\alpha}2}1_{\{\lambda t+\sqrt{1+t}\le x\le \mu t-\sqrt{1+t}\}}, \label{C_4}
  \ema
where $J^{\alpha,\beta,\gamma}(t,x;t_1,t_2;\lambda,\mu)$ is defined by \eqref{Jab}.
\end{lem}

\begin{proof}Since
$$e^{-\frac{\nu_0|x-y-\lambda(t-s)|}{2}} B_{\frac{\gamma}2}(s,y-\lambda s)\le e^{-\frac{\nu_0|x-y-\lambda(t-s)|}{4}}\(e^{-\frac{\nu_0|x-\lambda t|}{8}}+  B_{\frac{\gamma}2}(t,x-\lambda t) \),$$
it follows that
\bmas
&J^{\alpha,\beta,\gamma}(t,x;0,t;\lambda,\lambda) \\
=&\intt\intra (1+t-s)^{-\frac{\alpha}2}e^{-\frac{\nu_0|x-y-\lambda(t-s)|}{2}}(1+s)^{-\frac{\beta}2}B_{\frac{\gamma}2}(s,y-\lambda s) dyds \\
\le& C \(e^{-\frac{\nu_0|x-\lambda t|}{8}}+  B_{\frac{\gamma}2}(t,x-\lambda t) \)\(\int^{t/2}_0+ \int^t_{t/2}\)(1+t-s)^{-\frac{\alpha}2}(1+s)^{-\frac{\beta}2}  ds\\
\le& C \[(1+t)^{-\frac{\alpha}2}\Gamma_{\beta}(t)+(1+t)^{-\frac{\beta}2}\Gamma_{\alpha}(t)\]\( B_{\frac{\gamma}2}(t,x-\lambda t) +e^{-\frac{|x|+\lambda_1 t}{C_1}}\),
\emas
where $\lambda_1=\max\{1,|\lambda|\}$ and $C_1=24/\nu_0$. This gives \eqref{C_3}.

For $|x-\lambda t|\le \sqrt{1+t}$ or $|x-\mu t|\le \sqrt{1+t}$, we have
\bmas
&J^{\alpha,\beta,\gamma}(t,x;0,t;\lambda,\mu) \\
=&\intt\intra (1+t-s)^{-\frac{\alpha}2}e^{-\frac{\nu_0|x-y-\lambda(t-s)|}{2}}(1+s)^{-\frac{\beta}2}B_{\frac{\gamma}2}(s,y-\mu s)dyds \\
\le&  C\(\int^{t/2}_0+ \int^t_{t/2}\) (1+t-s)^{-\frac{\alpha}2}(1+s)^{-\frac{\beta}2}  ds\\
\le& C \[(1+t)^{-\frac{\alpha}2}\Gamma_{\beta}(t)+(1+t)^{-\frac{\beta}2}\Gamma_{\alpha}(t)\].
\emas
For $\lambda t+\sqrt{1+t}\le x\le \mu t-\sqrt{1+t}$, we write
\bmas
&\quad J^{\alpha,\beta,\gamma}(t,x;0,t;\lambda,\mu;C,D) \\
& = \( \int^{\frac{x-\lambda t}{A}}_{0}+\int^{t- \frac{\mu t-x}{A}}_{\frac{x-\lambda t}{A}}+\int^t_{t- \frac{\mu t-x}{A}}\) \intra \\
&\qquad(1+t-s)^{-\frac{\alpha}2}e^{-\frac{\nu_0|x-y-\lambda(t-s)|}{2}}(1+s)^{-\frac{\beta}2} B_{\frac{\gamma}2}(s,y-\mu s)dyds\\
&=:I_1+I_2+I_3,
\emas
where $A= 3(\mu-\lambda) $, $ (x-\lambda t)/A\le t/3$ and $ (\mu t-x)/A\le t/3$.

Since
\bma
& e^{-\frac{\nu_0|x-y-\lambda(t-s)|}{2}} B_{\frac{\gamma}2}(s,y-\mu s)\nnm\\
\le &e^{-\frac{\nu_0|x-y-\lambda(t-s)|}{4}}\(e^{-\frac{\nu_0|x-\lambda t+(\lambda-\mu)s|}{8}}+ CB_{\frac{\gamma}2}(t,x-\lambda t+(\lambda-\mu)s) \),\label{bb}
\ema
 it follows that
\bmas
I_1&\le C(1+t)^{-\frac{\alpha}2}\(e^{-\frac{\nu_0(x-\lambda t)}{12}}+ B_{\frac{\gamma}2}(t,x-\lambda t) \)\int^{\frac{x-\lambda t}{A}}_{0}  (1+s)^{-\frac{\beta}2}  ds\\
&\le C(1+t)^{-\frac{\alpha}2}\Gamma_{\beta}(t)\(e^{-\frac{|x|+\lambda_1 t}{C_2}}+ B_{\frac{\gamma}2}(t,x-\lambda t) \),\\
I_2&\le  C\(\int^{t/2}_{\frac{x-\lambda t}{A}}+\int^{t- \frac{\mu t-x}{A}}_{t/2}\)  (1+t-s)^{-\frac{\alpha}2}  (1+s)^{-\frac{\beta}2} \\
&\qquad  \(e^{-\frac{\nu_0|x-\lambda t+(\lambda-\mu)s|}{8}}+ B_{\frac{\gamma}2}(t,x-\lambda t+(\lambda-\mu)s)  \)ds\\
&\le  C(1+t)^{-\frac{\alpha-1}2}\Gamma_{2\gamma}(\sqrt{1+t})(1+x-\lambda t)^{-\frac{\beta}2}\\
&\quad+ C(1+t)^{-\frac{\beta-1}2}\Gamma_{2\gamma}(\sqrt{1+t})(1+\mu t-x)^{-\frac{\alpha}2} ,\\
I_3&\le C(1+t)^{-\frac{\beta}2}\(e^{-\frac{\nu_0(\mu t-x)}{12}}+ B_{\frac{\gamma}2}(t,x-\mu t) \) \int^t_{t- \frac{\mu t-x}{A}}  (1+t-s)^{-\frac{\alpha}2} ds\\
&\le C(1+t)^{-\frac{\beta}2}\Gamma_{\alpha}(t)\(e^{-\frac{|x|+\lambda_1 t}{C_2}}+ B_{\frac{\gamma}2}(t,x-\mu t) \),
\emas
where $\lambda_1=\max\{1,|\lambda|\}$ and $C_2=36/\nu_0$, and we use \eqref{bg} in the estimate of $I_2$.

For $x\le \lambda t-\sqrt{1+t}$ or $x\ge \mu t+\sqrt{1+t}$,  it follow from \eqref{aa1} and \eqref{bb} that
\bmas
&e^{-\frac{\nu_0|x-y-\lambda(t-s)|}{2}} B_{\frac{\gamma}2}(s,y-\mu s) \\
\le&
e^{-\frac{\nu_0|x-y-\lambda(t-s)|}{4}}\(e^{-\frac{|x|+\lambda_1 t}{C_2}}+ B_{\frac{\gamma}2}(t,x-\lambda t)+ B_{\frac{\gamma}2}(t,x-\mu t) \),
\emas
which gives
\bmas
J^{\alpha,\beta,\gamma}(t,x;0,t;\lambda,\mu)
\le& C \[(1+t)^{-\frac{\alpha}2}\Gamma_{\beta}(t)+(1+t)^{-\frac{\beta}2}\Gamma_{\alpha}(t)\]\\
&\times \(B_{\frac{\gamma}2}(t,x-\lambda t)+ B_{\frac{\gamma}2}(t,x-\mu t)\)+Ce^{-\frac{|x|+\lambda_1 t}{C_2}}.
\emas
And this completes the proof of the lemma.
\end{proof}

\begin{lem}\label{couple2} Assume $\alpha,\beta\ge 1$, $\gamma\ge 1$, and $\lambda< \mu $. There exists a constant $C>0$ such that
\bma
 &\quad K^{\alpha,\beta,\gamma}(t,x;0,t;\lambda,\lambda)\nnm\\
  &\le C \[(1+t)^{-\frac{\alpha}2}\Gamma_{\beta}(t)+(1+t)^{-\frac{\beta}2}\Gamma_{\alpha}(t)\]B_{\frac{\gamma}2}(t,x-\lambda t) , \label{C_3b}\\
   &\quad K^{\alpha,\beta,\gamma}(t,x;0,t;\lambda,\mu)\nnm\\
  &\le C \[(1+t)^{-\frac{\alpha}2}\Gamma_{\beta}(t)+(1+t)^{-\frac{\beta}2}\Gamma_{\alpha}(t)\]\(B_{\frac{\gamma}2}(t,x-\lambda t)+B_{\frac{\gamma}2}(t,x-\mu t)\) \nnm\\
  &\quad +C(1+t)^{-\frac{\alpha-1}2}\Gamma_{2\gamma}(\sqrt{1+t}) (1+x-\lambda t)^{-\frac{\beta}2}1_{\{\lambda t+\sqrt{1+t}\le x\le \mu t-\sqrt{1+t}\}}\nnm\\
   &\quad +C(1+t)^{-\frac{\beta-1}2}\Gamma_{2\gamma}(\sqrt{1+t})(1+\mu t-x)^{-\frac{\alpha}2}1_{\{\lambda t+\sqrt{1+t}\le x\le \mu t-\sqrt{1+t}\}}, \label{C_4a}
  \ema
where $K^{\alpha,\beta,\gamma}(t,x;t_1,t_2;\lambda,\mu)$ is defined by \eqref{Kab}.
\end{lem}

\begin{proof}Since
$$ B_{\frac{\gamma}2}(s,y-\lambda s)\le C  B_{\frac{\gamma}2}(t,x-\lambda t) ,$$
for $|x-y-\lambda(t-s)|\le 1$, it follows that
\bmas
&K^{\alpha,\beta,\gamma}(t,x;0,t;\lambda,\lambda) \\
=&\intt\int_{|x-y-\lambda(t-s)|\le 1} \Pi_{\alpha}(t-s,x-y;\lambda) (1+s)^{-\frac{\beta}2}B_{\frac{\gamma}2}(s,y-\lambda s) dyds \\
\le& C B_{\frac{\gamma}2}(t,x-\lambda t)\(\int^{t/2}_0+ \int^t_{t/2}\)(1+t-s)^{-\frac{\alpha}2}(1+s)^{-\frac{\beta}2}  ds\\
\le& C \[(1+t)^{-\frac{\alpha}2}\Gamma_{\beta}(t)+(1+t)^{-\frac{\beta}2}\Gamma_{\alpha}(t)\] B_{\frac{\gamma}2}(t,x-\lambda t) .
\emas
This gives \eqref{C_3b}.

For $|x-\lambda t|\le \sqrt{1+t}$ or $|x-\mu t|\le \sqrt{1+t}$, we have
\bmas
&K^{\alpha,\beta,\gamma}(t,x;0,t;\lambda,\mu) \\
=&\intt\int_{|x-y-\lambda(t-s)|\le 1} \Pi_{\alpha}(t-s,x-y;\lambda) (1+s)^{-\frac{\beta}2}B_{\frac{\gamma}2}(s,y-\mu s)dyds \\
\le&  C\(\int^{t/2}_0+ \int^t_{t/2}\) (1+t-s)^{-\frac{\alpha}2}(1+s)^{-\frac{\beta}2}  ds\\
\le& C \[(1+t)^{-\frac{\alpha}2}\Gamma_{\beta}(t)+(1+t)^{-\frac{\beta}2}\Gamma_{\alpha}(t)\].
\emas

For $\lambda t+\sqrt{1+t}\le x\le \mu t-\sqrt{1+t}$, we write
\bmas
K^{\alpha,\beta,\gamma}(t,x;0,t;\lambda,\mu)
& = \( \int^{\frac{x-\lambda t}{A}}_{0}+\int^{t- \frac{\mu t-x}{A}}_{\frac{x-\lambda t}{A}}+\int^t_{t- \frac{\mu t-x}{A}}\) \int_{|x-y-\lambda(t-s)|\le 1} \\
&\qquad  \Pi_{\alpha}(t-s,x-y;\lambda) (1+s)^{-\frac{\beta}2} B_{\frac{\gamma}2}(s,y-\mu s)dyds\\
&=:I_1+I_2+I_3,
\emas
where $A= 3(\mu-\lambda) $, $ (x-\lambda t)/A\le t/3$ and $ (\mu t-x)/A\le t/3$.

Since
\be
  B_{\frac{\gamma}2}(s,y-\mu s) \le C B_{\frac{\gamma}2}(t,x-\lambda t+(\lambda -\mu)s), \label{bb1}
\ee
for $|x-y-\lambda(t-s)|\le 1$, it follows that
\bmas
I_1&\le C(1+t)^{-\frac{\alpha}2}B_{\frac{\gamma}2}(t,x-\lambda t) \int^{\frac{x-\lambda t}{A}}_{0}  (1+s)^{-\frac{\beta}2}  ds\\
&\le C(1+t)^{-\frac{\alpha}2}\Gamma_{\beta}(t)B_{\frac{\gamma}2}(t,x-\lambda t),\\
I_2&\le  C\(\int^{t/2}_{\frac{x-\lambda t}{A}}+\int^{t- \frac{\mu t-x}{A}}_{t/2}\)  (1+t-s)^{-\frac{\alpha}2}  (1+s)^{-\frac{\beta}2} \\
&\qquad  \times B_{\frac{\gamma}2}(t,x-\lambda t+(\lambda-\mu)s) ds\\
&\le  C(1+t)^{-\frac{\alpha-1}2}\Gamma_{2\gamma}(\sqrt{1+t})(1+x-\lambda t)^{-\frac{\beta}2}\\
&\quad+ C(1+t)^{-\frac{\beta-1}2}\Gamma_{2\gamma}(\sqrt{1+t})(1+\mu t-x)^{-\frac{\alpha}2} ,\\
I_3&\le C(1+t)^{-\frac{\beta}2}B_{\frac{\gamma}2}(t,x-\mu t) \int^t_{t- \frac{\mu t-x}{A}}  (1+t-s)^{-\frac{\alpha}2} ds\\
&\le C(1+t)^{-\frac{\beta}2}\Gamma_{\alpha}(t)B_{\frac{\gamma}2}(t,x-\mu t).
\emas

For $x\le \lambda t-\sqrt{1+t}$ or $x\ge \mu t+\sqrt{1+t}$, it follow from \eqref{bb1} that
$$
B_{\frac{\gamma}2}(s,y-\mu s)
\le
C\( B_{\frac{\gamma}2}(t,x-\lambda t)+ B_{\frac{\gamma}2}(t,x-\mu t) \),
$$
if $|x-y-\lambda(t-s)|\le 1$. Hence,
\bmas
K^{\alpha,\beta,\gamma}(t,x;0,t;\lambda,\mu)
\le& C \[(1+t)^{-\frac{\alpha}2}\Gamma_{\beta}(t)+(1+t)^{-\frac{\beta}2}\Gamma_{\alpha}(t)\]\\
&\times \(B_{\frac{\gamma}2}(t,x-\lambda t)+ B_{\frac{\gamma}2}(t,x-\mu t)\) .
\emas
This completes the proof of  the lemma.
\end{proof}

\begin{lem}\label{wc-1} For any $\beta,\gamma\ge 0$, there exists a constant $C>0$ such that
\be
  L^{ \beta,\gamma}(t,x;0,t;\lambda, D)  \le C(1+t)^{-\frac{\beta}2}B_{\frac{\gamma}2}(t,x-\lambda t),\label{wave1}
  \ee
  where $L^{ \beta,\gamma}(t,x;t_1,t_2;\lambda,D)$ is defined by \eqref{Lab}.
\end{lem}

\begin{proof}
Since
\bmas
e^{-\frac{|x-y|+t-s}D} B_{\frac{\gamma}2}(s,y-\lambda s)&\le e^{-\frac{|x-y|+t-s}{2D}}e^{-\frac{|x-y-\lambda (t-s)|}{D_1}} B_{\frac{\gamma}2}(s,y-\lambda s)\\
&\le Ce^{-\frac{|x-y|+t-s}{2D}}\(e^{-\frac{|x-\lambda t|}{2D_1}}  +  B_{\frac{\gamma}2}(t,x-\lambda t)\),
\emas
where $D_1=\max\{2D,2\lambda_1 D\}$ and $\lambda_1=\max\{1,|\lambda|\}$, we have
\bmas
&L^{ \beta,\gamma}(t,x;t_1,t_2;\lambda, D)\nnm\\
 =&\intt\intra e^{-\frac{|x-y|+t-s}D}(1+s)^{-\frac{\beta}2} B_{\frac{\gamma}2}(s,y-\lambda s)dyds\\
 \le& C\(e^{-\frac{|x-\lambda t|}{2D_1}}  +  B_{\frac{\gamma}2}(t,x-\lambda t)\)\intt\intra e^{-\frac{|x-y|+t-s}{2D}}(1+s)^{-\frac{\alpha}2}  dyds\\
 \le& C(1+t)^{-\frac{\alpha}2} \(B_{\frac{\gamma}2}(t,x-\lambda t)+e^{-\frac{|x|+\lambda_1 t}{6D_1}}\).
  \emas
 This gives \eqref{wave1} and completes the proof of the lemma.
\end{proof}

With Theorem \ref{green1} and Lemmas \ref{S_1}--\ref{wc-1}, we are ready  to prove Theorem \ref{thm1} as follows.

\begin{proof}[\underline{\textbf{Proof of Theorem \ref{thm1}}}]
Let $f_1(t)$ and $V(t)=(V^1,V^2,V^3)=(f_2,E,B)$ be a solution to the IVP problem \eqref{VMB3a}--\eqref{VMB3d} for $t>0$. We can represent this solution by
\bmas
f_1(t,x)&=G_b(t)\ast f_{1,0}+\intt G_b(t-s)\ast (H_1+H_2)(s) ds \nnm\\
&=:J_1+J_2,\\
V^{i}(t,x)&=\sum^3_{j=1}G^{ij}(t)\ast V^j_0+ \intt G^{i1}(t-s)\ast H_3(s) ds\nnm\\
&=:J^{i}_3+J^{i}_4,\quad i=1,2,3,
\emas
where $V_0=(V^1_0,V^2_0,V^3_0)=(f_{2,0},E_0,B_0)$, and $H_i$, $i=1,2,3$ are the nonlinear terms defined by
$$
\left\{\bln
H_1&=\frac12(v\cdot E)f_2-(E+v\times B)\cdot \Tdv f_2,\quad H_2=\Gamma(f_1,f_1),\\
H_3&=\frac12(v\cdot E)f_1-(E+v\times B)\cdot \Tdv f_1+\Gamma(f_2,f_1).
\eln\right.
$$
Let
$$
\Phi^{\alpha,\gamma}(t,x )=\sum_{-2\le i\le 2} (1+t)^{-\frac{\alpha}2}B_{\frac{\gamma}2}(t,x-\sigma_it) .
$$
Define
\bmas
 \mathcal{Q}(t)=\sup_{0\le s\le t, x\in \R }
 &\Big\{
    \(\| f_1\|_{L^\infty_{v,3}}+ \|\Tdv f_1\|_{L^\infty_{v,3}}+  | B|\)\Phi^{1 ,1}(t,x)^{-1}
   \nnm\\
&+\(\| f_2\|_{L^\infty_{v,3}}+\|\Tdv f_2\|_{L^\infty_{v,3}}+| E|+\|\dx f_1\|_{L^\infty_{v,3}}+|\dx B|\) \Phi^{2,1}(t,x)^{-1}\nnm\\
  & +\(\|\dx f_2\|_{L^\infty_{v,3}}+|\dx E|\)\Phi^{\frac52,1}(t,x)^{-1} +(1+s)^{-\frac34}  \mathbb{H}_{5,3}(U)^{\frac12}   \Big\}.
 \emas

From \cite{Ukai1,Ukai3}, it holds that for $\beta\in \N^3$ and $k\ge 0$,
\be
\|\dvb \Gamma(f,g)\|_{L^\infty_{v,k-1}}\le C\sum_{\beta_1+\beta_2\le  \beta}\|\dv^{\beta_1}f\|_{L^\infty_{v,k}}\|\dv^{\beta_2}g\|_{L^\infty_{v,k}}. \label{c1}
\ee
Thus
\bma
\|\dxa H_1(s,x)\|_{L^\infty_{v,2}}&\le C\sum_{\alpha'\le\alpha}\Big(|\dx^{\alpha'} E|\|\dx^{\alpha-\alpha'}f_2\|_{L^\infty_{v,3}}+|\dx^{\alpha'} E|\|\Tdv \dx^{\alpha-\alpha'}f_2\|_{L^\infty_{v,2}}\nnm\\
&\quad+|\dx^{\alpha'} B|\|\Tdv \dx^{\alpha-\alpha'}f_2\|_{L^\infty_{v,3}}\Big)\nnm\\
&\le C\mathcal{Q}^2(t)\Phi^{3+\frac{1}{10}\alpha,2-\frac{2}5\alpha}(s,x) ,  \label{h1}\\
\|\dxa H_2(s,x)\|_{L^\infty_{v,2}}&\le C\sum_{\alpha'\le\alpha}\|\dx^{\alpha'}f_1\|_{L^\infty_{v,3}}\|\dx^{\alpha-\alpha'}f_1\|_{L^\infty_{v,3}}
\le C\mathcal{Q}^2(t)\Phi^{2+\alpha,2}(s,x),  \label{h2}\\
\|\dxa H_3(s,x)\|_{L^\infty_{v,2}}&\le C\sum_{\alpha'\le\alpha}\Big(|\dx^{\alpha'} E|\|\dx^{\alpha-\alpha'}f_1\|_{L^\infty_{v,3}}+|\dx^{\alpha'} E|\|\Tdv \dx^{\alpha-\alpha'}f_1\|_{L^\infty_{v,2}}\nnm\\
&\quad+|\dx^{\alpha'} B|\|\Tdv \dx^{\alpha-\alpha'}f_1\|_{L^\infty_{v,3}}+\|\dx^{\alpha'}f_2\|_{L^\infty_{v,3}}\|\dx^{\alpha-\alpha'}f_1\|_{L^\infty_{v,3}}\Big)\nnm\\
&\le C\mathcal{Q}^2(t) \Phi^{2+\frac{4}5\alpha,2-\frac25\alpha}(s,x) ,  \label{h3}
\ema
where $\alpha=0,1$, $0\le s\le t$ and we have used Gagliardo-Nirenberg interpolation inequality to have
\bmas
\|\dx\Tdv f_1(s,x) \|_{L^\infty_{v,3}}&\le C\|\dx\Tdv^3 (w^3 f_1) \|^{2/5}_{L^2_{v}}\|\dx (w^3f_1) \|^{3/5}_{L^\infty_{v}}+C\|\dx (w^2f_1) \|_{L^\infty_{v}}
\\
&\le C\mathcal{Q}(t)\Phi^{\frac95,\frac35}(s,x),
\\
\|\dx\Tdv f_2(s,x) \|_{L^\infty_{v,3}}&\le C\|\dx\Tdv^3 (w^3 f_2) \|^{2/5}_{L^2_{v}}\|\dx (w^3f_2) \|^{3/5}_{L^\infty_{v}}+C\|\dx (w^2f_2) \|_{L^\infty_{v}}
\\
&\le C\mathcal{Q}(t) \Phi^{\frac{21}{10},\frac35}(s,x) .
\emas
By Theorems \ref{green1} and \ref{green2}, we decompose
\bma
\dxa J_1&=\dxa  G_{b,0}(t)\ast f_{1,0}+G_{b,1}(t)\ast \dxa  f_{1,0}+W_1(t)\ast\dxa f_{1,0}\nnm\\
&=:J^{\alpha}_{1,1}+J^{\alpha}_{1,2}+J^{\alpha}_{1,3}, \label{gt1}\\
\dxa J_2&=\sum^2_{i=1} \intt\dxa G_{b,0}(t-s)\ast H_i(s)ds+\sum^2_{i=1} \intt G_{b,1}(t-s)\ast \dxa H_i(s)ds\nnm\\
&\quad+\sum^2_{i=1}\intt W_{1}(t-s)\ast\dxa H_i(s)ds\nnm\\
&=:J^{\alpha}_{2,1}+J^{\alpha}_{2,2}+J^{\alpha}_{2,3},
\\
\dxa J^{i}_3&=\sum_{j=1,2,3} \dxa G^{ij}_{F,0}(t)\ast V^j_0+ \sum_{j=1,2,3}G^{ij}_R(t)\ast\dxa V^j_0\nnm\\
&\quad +\sum_{j=1,2,3}(G^{ij}_{F,1}+G^{ij}_2)(t)\ast \dxa V^j_0+\sum_{j=1,2,3}W^{ij}_{0,1}(t)\ast\dxa V^j_0\nnm\\
&=:J^{i,\alpha}_{3,1}+J^{i,\alpha}_{3,2}+J^{i,\alpha}_{3,3}+J^{i,\alpha}_{3,4}, \label{gt3}\\
\dxa J^{i}_4&=\intt \dxa G^{i1}_{F,0}(t-s)\ast H_3(s)ds+\intt G^{i1}_{R}(t-s)\ast \dxa H_3(s)ds\nnm\\
&\quad+\intt (G^{i1}_{F,1}+G^{i1}_2)(t-s)\ast \dxa H_3(s)ds+\intt W^{i1}_{0,1}(t-s)\ast\dxa H_3(s)ds\nnm\\
&=:J^{i,\alpha}_{4,1}+J^{i,\alpha}_{4,2}+J^{i,\alpha}_{4,3}+J^{i,\alpha}_{4,4}. \label{gt4}
\ema

First, we estimate the linear terms $\dxa J_1$ and $\dxa J^i_3$ for $\alpha=0,1$. By  Theorems \ref{green1}, \ref{green2}  and Lemma \ref{p-1}, we obtain
\bmas
\|J^{\alpha}_{1,k}\|
&\le C\delta_0 \sum_{l=-1,0,1}(1+t)^{-\frac{1+\alpha}2}B_{\frac{\gamma}2}(t,x-\mathbf{c} lt) ,  \\
\|J^{i,\alpha}_{3,k}\|
&\le C\delta_0  (1+t)^{-\frac{2+\alpha}2}B_{\frac{\gamma}2}(t,x)  ,\quad i=1,2, \\
\|J^{3,\alpha}_{3,k}\|
&\le C\delta_0 (1+t)^{-\frac{1+\alpha}2}B_{\frac{\gamma}2}(t,x) ,
\emas where $k=1,2$.
By Lemma \ref{gh1} and Theorem \ref{green1}, we have
$$
J^{i,\alpha}_{3,3}=\sum^3_{j=1}F^{ij}_{2}(t)\ast \dx^{2}\dxa V^j_0+\sum^3_{j=1}[\Lambda^{ij}_{2}(t)+\Omega^{ij}_{2}(t)]\ast (\nu_0+\dx)^{2}\dxa V^j_0,
$$
which together with Lemma \ref{p-1} implies that
\bmas
\|J^{1,\alpha}_{3,3}\|&\le C\sum_{l=\pm 1}(1+t)^{-\frac52}\ln^2(2+t)B_{\frac{\gamma}2}(t,x-l t),\\
\|J^{i,\alpha}_{3,3}\| &\le C\sum_{l=\pm 1}(1+t)^{-2}\ln^2(2+t)B_{\frac{\gamma}2}(t,x-l t),\quad i=2,3.
\emas
Moreover, by Lemmas  \ref{S_1} and \ref{S_2}, we obtain
$$
\|J^{\alpha}_{1,3}\|,\|J^{i,\alpha}_{3,4}\|\le C\delta_0e^{-\frac{\nu_0 t}{3}}(1+|x|^2)^{-\frac{\gamma}2}.
$$
Thus
\be \label{J1}
\left\{\bln
\|\dxa J_{1}\|&\le C\delta_0 \sum_{j=-2,0,2}(1+t)^{-\frac{1+\alpha}2}B_{\frac{\gamma}2}(t,x-\sigma_j t) ,\\
\|\dxa J^i_{3}\|&\le C\delta_0 \sum_{j=-1,0,1} (1+t)^{-\frac{2+\alpha}2}B_{\frac{\gamma}2}(t,x-\sigma_j t),\quad i=2,3, \\
\|\dxa J^3_{3}\|&\le C\delta_0 \sum_{j=-1,0,1} (1+t)^{-\frac{1+\alpha}2}B_{\frac{\gamma}2}(t,x-\sigma_j t) .
\eln\right.
\ee

Next, we estimate the nonlinear terms $\dxa J_2$ and $\dxa J^i_4$ for $\alpha=0,1$. Denote
\bmas
\mathcal{I}^{\alpha,\beta,\gamma} (t_1,t_2;\lambda,\theta)&=\int^{t_2}_{t_1} \intra (1+t-s)^{-\frac{\alpha}2}e^{-\frac{|x-y-\lambda(t-s)|^2}{\theta(1+t-s)}}\Phi^{\beta,\gamma}(s,y)dyds,\\
\mathcal{J}^{\alpha,\beta,\gamma} (t_1,t_2;\lambda)&=\int^{t_2}_{t_1} \intra (1+t-s)^{-\frac{\alpha}2}e^{-\frac{\nu_0|x-y-\lambda(t-s)|}{2}}\Phi^{\beta,\gamma}(s,y)dyds,\\
\mathcal{K}^{\alpha,\beta,\gamma} (t_1,t_2;\lambda)&=\int^{t_2}_{t_1} \intra \Pi_{\alpha}(t-s,x-y;\lambda)\Phi^{\beta,\gamma}(s,y)dyds,\\
\mathcal{L}^{ \beta,\gamma} (t_1,t_2;\theta)&=\int^{t_2}_{t_1} \intra  e^{-\frac{ |x-y|+t-s}{\theta}}\Phi^{\beta,\gamma}(s,y)dyds,
\emas
where $\alpha,\beta,\gamma\ge 0$, $\theta> 0$, $\lambda\in \R$, and $\Pi_{\alpha}(t,x;\lambda)$ is defined by \eqref{H-a}.

By Theorems  \ref{green1}--\ref{green2} and Lemma \ref{wc-1}, we have
\bma
\|J^{\alpha}_{2,2}\|\le& C\mathcal{Q}^2(t)  \[\mathcal{L}^{3,\frac85}(0,t;D)+ \mathcal{L}^{2+\alpha,2}(0,t;D)\]\le C\mathcal{Q}^2(t)\Phi^{2+\alpha,\frac85}(t,x),  \label{J22}\\
\|J^{i,\alpha}_{4,2}\|\le& C\mathcal{Q}^2(t)  \mathcal{L}^{2+\frac{4}5\alpha,\frac85}(0,t;D)\le C\mathcal{Q}^2(t)\Phi^{2+\frac{4}5\alpha,\frac85}(t,x),\quad i=1,2,3.\label{J42}
\ema
By Lemmas \ref{green3a} and \ref{green3b}, we have
\bma
\|J^{\alpha}_{2,3}\| &\le C\mathcal{Q}^2(t)\Phi^{2+ \alpha,\frac85}(t,x),\label{J23}\\
\|J^{i,\alpha}_{4,4}\|&\le C\mathcal{Q}^2(t)\Phi^{2+\frac{4}5\alpha,\frac85}(t,x), \quad i=1,2,3. \label{J44}
\ema

For $J^{\alpha}_{2,1}$, we first decompose it into
\bmas
J^{\alpha}_{2,1}&=\intt \dxa G_{b,0}(t-s)\ast P_0H_1ds+\intt \dxa G_{b,0}(t-s)\ast (P_1H_1+H_2)ds\\
&=:I^{\alpha}_1+I^{\alpha}_2.
\emas
By Theorem \ref{green2}, Corollary \ref{couple1b} and \eqref{h1}--\eqref{h3}, we have
\bmas
\|I^{\alpha}_2\|\le & \int^{t/2}_0\|\dxa G_{b,0}(t-s)P_1\|\ast \|(P_1H_1+H_2)\|ds\\
&+\int^{t}_{t/2}\|G_{b,0}(t-s)P_1\|\ast \|\dxa (P_1H_1+H_2)\|ds\\
\le &C\mathcal{Q}^2(t)\sum_{j=-2,0,2}\[\mathcal{I}^{2+\alpha,2,2}(0,t/2;\sigma_j, D)+\mathcal{I}^{2,2+\alpha,\frac85}(t/2,t;\sigma_j, D)\]\\
\le& C\mathcal{Q}^2(t)\Phi^{1+\alpha,1}(t,x), \quad \alpha=0,1.
\emas
Since
$$
 P_0H_1=(n_2 E+m_2\times B)\cdot v\chi_0+\sqrt{\frac23} E\cdot m_2 \chi_4,
$$
where $ n_2=(f_2,\chi_0)$ and $m_2=(f_2,v\chi_0),$ we write
\bmas
I^{\alpha}_1&=\intt \dxa G_{b,0}(t-s)\ast \bigg(n_2 E \cdot v\chi_0+\sqrt{\frac23} E\cdot m_2 \chi_4\bigg)ds\\
&\quad+\intt \dxa G_{b,0}(t-s)\ast (m_2\times B )\cdot v\chi_0 ds=:I^{\alpha}_{1,1}+I^{\alpha}_{1,2}.
\emas
For $I^{\alpha}_{1,1}$, it holds that for $\alpha=0,1$,
$$
\|I^{\alpha}_{1,1}\|\le C\mathcal{Q}^2(t)\sum_{j=-2,0,2 } \mathcal{I}^{1+\alpha,4,2}(0,t;\sigma_j, D)\le C\mathcal{Q}^2(t) \Phi^{1+\alpha,1}(t,x).
$$

To estimate $I^{\alpha}_{1,2}$, by \eqref{VMB3b} we have
\bmas
m_2\times B&=\dx \O B\times B-\dt E\times B\\
&=\frac12(\dx |B|^2,0,0)-\dt (E\times B)-E\times \dx\O E,
\emas
which gives
\bmas
I^{\alpha}_{1,2}&=\frac12 \intt \dx^{\alpha}G_{b,0}(t-s)\ast \dx |B|^2 v_1\chi_0 ds-\intt \dxa G_{b,0}(t-s)\ast \dt(E\times B)\cdot v\chi_0 ds\\
&\quad- \intt \dxa G_{b,0}(t-s)\ast (E\times \dx\O E)\cdot v\chi_0 ds\\
&=:I^{\alpha}_{3,1}+I^{\alpha}_{3,2}+I^{\alpha}_{3,3}.
\emas
By Theorem \ref{green2} and Lemmas \ref{couple1a}, \ref{couple1b}, we obtain that for $\alpha=0,1$,
\bmas
\|I^{\alpha}_{3,1}\|\le & \int^{t/2}_0 \|\dx^{\alpha+1}G_{b,0}(t-s)\|\ast |B|^2   ds + \int^{t}_{t/2} \|\dx^{\alpha}G_{b,0}(t-s)\|\ast |\dx |B|^2| ds\\
\le  &C\mathcal{Q}^2(t)\sum_{j=-2,0,2}\[\mathcal{I}^{2+\alpha,2,2}(0,t/2; \sigma_j, D)+\mathcal{I}^{1+\alpha,3,2}(t/2,t; \sigma_j, D)\]\\
\le& C\mathcal{Q}^2(t) \Phi^{1+\alpha,1}(t,x),\\
\|I^{\alpha}_{3,3}\|\le &C \intt \|\dx^{\alpha}G_{b,0}(t-s)\|\ast ( |E||\dx E|)  ds\\
\le &C\mathcal{Q}^2(t)\sum_{j=-2,0,2 } \mathcal{I}^{1+\alpha,4,2}(0,t;\sigma_j, D)\le C\mathcal{Q}^2(t)\Phi^{1+\alpha,1}(t,x).
\emas

For $I^{\alpha}_{3,2}$, by integration by part and noting that $ G_{b,0}(0,x)=0$ because $G_{b,0}=G_{l,0}1_{\{|x|\le 2\mathbf{c}t\}}$ is supported in $|x|\le 2\mathbf{c}t$, we have
$$
I^{\alpha}_{3,2}=-\dxa G_{b,0}(t)\ast (E_0\times B_0)\cdot v\chi_0+\intt \dt\dxa G_{b,0}(t-s)\ast(E\times B)\cdot v\chi_0 ds.
$$
By \eqref{GL0a}, we have
$$
\dt \hat{G}_{l,0}(t,\xi)=\sum^1_{j=-1}\eta_j(\xi)e^{\eta_j(\xi)t}\varphi_j(\xi)\otimes \langle \varphi_j(\xi)|.
$$
Thus, by using  a similar argument as for Lemma \ref{gl0}, we obtain
\be
 \|\dt\dxa G_{b,0}(t,x)\|\le C\bigg(\sum^1_{l=-1}(1+t)^{-\frac{2+\alpha}2}e^{-\frac{|x-l\mathbf{c}t|^2}{D(1+t)}}+e^{-\frac{|x|+t}D}\bigg) . \label{in5}
 \ee
 Then, it follows from Lemma \ref{couple1a} and \eqref{in5} that
\bmas
\|I^{\alpha}_{3,2}\|&\le C\sum_{j=-2,0,2 }\(\mathcal{Q}^2(t) \mathcal{I}^{2+\alpha,3,2}(0,t;\sigma_j, D) +\delta^2_0 (1+t)^{-\frac{1+\alpha}2}B_{\gamma}(t,x-\sigma_j t)\)\\
&\le C(\delta_0^2+\mathcal{Q}^2(t))\Phi^{1+\alpha,1}(t,x), \quad \alpha=0,1.
 \emas
 Thus
 \be
 \|J^{\alpha}_{2,1}\|\le C(\delta_0^2+\mathcal{Q}^2(t))\Phi^{1+\alpha,1}(t,x), \quad \alpha=0,1. \label{J21}
 \ee
By combining \eqref{J22}, \eqref{J23} and \eqref{J21}, we have
\be
 \|\dxa J_{2}\|\le C(\delta_0^2+\mathcal{Q}^2(t))\Phi^{1+\alpha,1}(t,x), \quad \alpha=0,1. \label{J21a}
 \ee

We now turn to estimate  $J^{i,\alpha}_{4,1}$, $i=1,2,3$.
By \eqref{G_L0} and the fact that $ \|\Tdv h_j(\xi)\|=O(|\xi|),$
 we obtain from Lemma \ref{gl0} that
$$
\left\{\bln
\|\dxa G^{i1}_{F,0}(t,x)\Tdv f_0\|&\le C\( (1+t)^{-\frac{3+\alpha}2}e^{-\frac{|x|^2}{D(1+t)}} +e^{-\frac{|x|+t}D}\)\|f_0\|,\quad i=1,2, \\
\|\dxa G^{31}_{F,0}(t,x)\Tdv f_0\|&\le C\( (1+t)^{-\frac{2+\alpha}2}e^{-\frac{|x|^2}{D(1+t)}} +e^{-\frac{|x|+t}D}\)\|f_0\| ,
\eln\right.
$$
for $\alpha\ge -1$, where $\dx^{-1}$ is defined by \eqref{dxa}.
Thus
\bma
J^{i,\alpha}_{4,1}=&\int^{t/2}_0 \dxa G^{i1}_{F,0}(t-s) \ast  H_3(s) ds+\int^t_{t/2}  \dx^{\alpha-1}G^{i1}_{F,0}(t-s) \ast  \dx H_3(s) ds\nnm\\
\le&C\int^{t/2}_0 \|\dxa G^{i1}_{F,0}(t-s)\|\ast \|H_4\|ds+C\int^t_{t/2}  \| \dx^{\alpha-1} G^{i1}_{F,0}(t-s)\|\ast\| \dx H_4\|ds, \label{J41}
\ema
where
$$H_4=\frac12(v\cdot E)f_1-(E+v\times B) f_1+\Gamma(f_2,f_1).$$
Since
$$
\|\dxa H_4(s,x)\|_{L^\infty_{v,2}}\le C\mathcal{Q}^2(t) \Phi^{2+\alpha,2}(t,x),\quad \alpha=0,1,
$$
it follows from \eqref{J41} that
 \bma
\|J^{i,\alpha}_{4,1}\|\le& C\mathcal{Q}^2(t)\[\mathcal{I}^{3+\alpha ,2,2}(0,t/2;0,D)+\mathcal{I}^{2+\alpha,3,2}(t/2,t;0,D)\]\nnm\\
\le& C\mathcal{Q}^2(t)\Phi^{2+\alpha,1}(t,x)\ln^{\alpha}(2+t),\quad i=1,2, \label{J41a}
\\
\|J^{3,\alpha}_{4,1}\|\le& C\mathcal{Q}^2(t)\[\mathcal{I}^{2+\alpha ,2,2}(0,t/2;0,D)+\mathcal{I}^{1+\alpha,3,2}(t/2,t;0,D)\]\nnm\\
\le& C\mathcal{Q}^2(t)\Phi^{1+\alpha,1}(t,x).\label{J41b}
 \ema

For $J^{i,\alpha}_{4,3}$, by Theorem \ref{green1}, Lemma \ref{gh1} and Lemmas \ref{couple1}, \ref{couple2} we have for a fixed constant $0<\eps\le 1/2$ and $\alpha=0,1$,
\bma
\|J^{i,\alpha}_{4,3}\|
&\le \intt (\|F^{i1}_1 \|+\|\Lambda^{i1}_1\|+\|\Omega^{i1}_1 \|)(t-s)\ast  (\|\dxa  H_3\|+\|\dx^{\alpha+1}  H_3\|)(s)ds\nnm\\
&\le C\mathcal{Q}^2(t)\sum_{j=-1,1}\[\mathcal{J}^{3-\eps,2+\frac{\alpha}2,1}(0,t;\sigma_j)+\mathcal{K}^{3-\eps,2+\frac{\alpha}2,1}(0,t;\sigma_j)\] \nnm\\
&\le C\mathcal{Q}^2(t) \Phi^{2+\frac{\alpha}2,1}(t,x),\quad i=1,2,3, \label{J43}
\ema
where we have used
\bma
\|\dx^2 H_3(s,x)\|_{L^2_{v}}&\le C\sum_{\alpha\le2}\Big(|\dx^{\alpha} E|\|w\dx^{2-\alpha}f_1\|_{L^2_{v}}+|\dx^{\alpha} E|\|\Tdv\dx^{2-\alpha}f_1\|_{L^2_{v}}\nnm\\
&\quad+|\dx^{\alpha} B|\| w\Tdv\dx^{2-\alpha}f_1\|_{L^2_{v}}
+\|w\dx^{\alpha}f_2\|_{L^2_{v}}\|w\dx^{2-\alpha}f_1\|_{L^2_{v}}\Big)\nnm\\
&\le C\mathcal{Q}^2(t)\Phi^{\frac52,1}(t,x). \label{H22}
\ema
In \eqref{H22}, we have used the Sobolev's embedding theorem for obtaining
$$\|w\dx^2 \dvb f_i\|_{L^2_v}+|\dx^2 E|+|\dx^2 B|\le C\mathbb{H}_{5,3}(U)^{\frac12}\le C(1+t)^{-\frac34}\mathcal{Q}(t) $$
for $i=1,2,\, |\beta|=0,1.$
By combining \eqref{J42}, \eqref{J44}, \eqref{J41a}, \eqref{J41b} and \eqref{J43}, we have
\be \label{J4}
\left\{\bln
\|\dxa J^i_{4}\|&\le C\mathcal{Q}^2(t) \Phi^{2+\frac{\alpha}2,1}(t,x),\quad i=1,2,\\
\|\dxa J^3_{4}\|&\le C\mathcal{Q}^2(t) \Phi^{1+\alpha,1}(t,x).
\eln\right.
\ee
Thus, it follows from \eqref{J1}, \eqref{J21a} and \eqref{J4} that for $\alpha=0,1$,
\be \label{density1}
\left\{\bln
\|\dxa f_1(t,x)\|_{L^2_{v}}+|\dxa B(t,x)|
&\le C(\delta_0+ \mathcal{Q}^2(t))\Phi^{1+\alpha,1} (t,x), \\
\|\dxa f_2(t,x)\|_{L^2_{v}}+|\dxa E(t,x)|
&\le C(\delta_0+ \mathcal{Q}^2(t))\Phi^{2+\frac{\alpha}{2},1} (t,x).
\eln\right.
\ee

By \eqref{VMB3a} and \eqref{VMB4a}, we have
\bmas
\dt f_1+v_1\dx f_1+\nu(v)f_1&=Kf_1 +H_1 +H_2,\\
\dt f_2+v_1\dx f_2+\nu(v)f_2&=K_1f_2+v\cdot E\chi_0+H_3.
\emas
Thus we can represent $f_1$ and $f_2$ by
\bma
f_1(t,x)&= S^tf_{1,0}+\intt  S^{t-s}(Kf_1+H_1 +H_2)ds,\label{s0}\\
f_2(t,x)&= S^tf_{2,0}+\intt  S^{t-s}(K_1f_2+v\cdot E \chi_0+H_3)ds.
\ema
By Lemma \ref{S_1}, it holds that
\be
\| S^tf_{j,0}(x)\|_{L^\infty_{v,3}} \le C\delta_0e^{-\frac{2\nu_0t}3}(1+|x|^2)^{-\frac{\gamma}2},\quad j=1,2.\label{s1}
\ee
By \eqref{density1}, we have for $\alpha=0,1,$
\bmas
\|\dxa Kf_1 \|_{L^\infty_{v,0}}&\le C\|\dxa f_1 \|_{L^2_{v}}\le C(\delta_0+\mathcal{Q}^2(t))\Phi^{1+\alpha,1}(t,x) ,\\
\|\dxa K_1f_2 \|_{L^\infty_{v,0}}&\le C\|\dxa f_2 \|_{L^2_{v}}\le C(\delta_0+\mathcal{Q}^2(t))\Phi^{2+\frac{\alpha}{2},1}(t,x),
\emas
which together with  \eqref{s4}, \eqref{h1}, \eqref{h2} and \eqref{h3} imply that
\bma
\bigg\|\intt  S^{t-s}\dxa(Kf_1+H_1 +H_2)ds\bigg\|_{L^\infty_{v,1}}
&\le C(\delta_0+\mathcal{Q}^2(t))\Phi^{1+\alpha,1}(t,x),
\\
\bigg\|\intt  S^{t-s}\dxa(K_1f_2+v\cdot E \chi_0+H_3)ds\bigg\|_{L^\infty_{v,1}}
&\le C(\delta_0+\mathcal{Q}^2(t))\Phi^{2+\frac{\alpha}{2},1}(t,x). \label{s2b}
\ema
Thus, it follows from \eqref{s0}--\eqref{s2b} that
\bmas
\|\dxa f_1(t,x)\|_{L^\infty_{v,1}}&\le  C(\delta_0+\mathcal{Q}^2(t))\Phi^{1+\alpha,1}(t,x),\\
\|\dxa f_2(t,x)\|_{L^\infty_{v,1}}&\le  C(\delta_0+\mathcal{Q}^2(t))\Phi^{2+\frac{\alpha}{2},1}(t,x).
\emas
By induction  and using
$$
\|Kf_i\|_{L^\infty_{v,k}},\|K_1f_i\|_{L^\infty_{v,k}}\le C_k\|f_i\|_{L^\infty_{v,k-1}},\,\,\,k\ge 1,\,\, i=1,2,
$$ we have
\be  \label{s1a}
\left\{\bln
\|\dxa f_1(t,x)\|_{L^\infty_{v,3}}&\le  C(\delta_0+\mathcal{Q}^2(t))\Phi^{1+\alpha,1}(t,x),\\
\|\dxa f_2(t,x)\|_{L^\infty_{v,3}}&\le  C(\delta_0+\mathcal{Q}^2(t))\Phi^{2+\frac{\alpha}{2},1}(t,x).
\eln\right.
\ee

Taking the derivative $\dvb$ of \eqref{VMB3a} and \eqref{VMB4a} with $|\beta|= 1$, we have
$$
\dt \dvb f_j+v_1\dx \dvb f_j+\nu(v)\dvb f_j=H^{\beta}_j,\quad j=1,2,\\
$$
where
$$
\left\{\bln
H^{\beta}_1=&-\dx^\beta f_1- \dv^{\beta}\nu(v)f_1+\dvb (Kf_1)
+\dvb H_1+\dvb H_2,\\
H^{\beta}_2=&-\dx^\beta f_2- \dv^{\beta}\nu(v)f_2+\dvb (K_1f_2)+\dvb(v \chi_0)\cdot E+\dvb H_3.
\eln\right.
$$
Thus, $\dvb f_j$, $j=1,2$ can be represented by
\be
 \dvb f_j(t,x)= S^t\dvb f_{j,0}+\intt  S^{t-s}H^{\beta}_j ds,\quad j=1,2. \label{s2}
\ee
It follows from  \eqref{s1a} that
\bmas
\|H^{\beta}_1(s,x)\|_{L^\infty_{v,2}}&\le C(\|\dx^\beta f_1\|_{L^\infty_{v,2}}+\|f_1\|_{L^\infty_{v,2}}) +C\sum_{k\le 2}(| E|+| B|) \|\Tdv^k f_2\|_{L^\infty_{v,3}}\\
&\quad+C\sum_{k \le 1} \|\Tdv^k f_1\|_{L^\infty_{v,3}} \| f_1\|_{L^\infty_{v,3}}\\
&\le C(\delta_0+\mathcal{Q}^2(t))\Phi^{1,1}(t,x),
\\
\| H^\beta_2(s,x)\|_{L^\infty_{v,2}}&\le C(\|\dx^\beta f_2\|_{L^\infty_{v,2}}+\|f_2\|_{L^\infty_{v,2}}+| E|)+C\sum_{k\le 2}(| E|+| B|) \|\Tdv^k f_1\|_{L^\infty_{v,3}}\\
&\quad+C\sum_{k+l\le 1} \|\Tdv^k f_1\|_{L^\infty_{v,3}} \|\Tdv^{l}f_2\|_{L^\infty_{v,3}}\\
& \le C(\delta_0+\mathcal{Q}^2(t))\Phi^{2,1}(t,x),
\emas
where we have used
\bmas
\|\Tdv^2 f_1 \|_{L^\infty_{v,3}}&\le C(\|P_0 f_1 \|_{L^2_v}+\| w^3\Tdv^2 P_1 f_1 \|_{H^2_v})\le C(1+t)^{-\frac12}(\delta_0+\mathcal{Q}^2(t)), \\
\|\Tdv^2 f_2 \|_{L^\infty_{v,3}}&\le C(\|P_d f_2 \|_{L^2_v}+\| w^3\Tdv^2 P_r f_2 \|_{H^2_v})\le C(1+t)^{-\frac34}(\delta_0+\mathcal{Q}^2(t)).
\emas
The above estimates  and \eqref{s2}, \eqref{wt1}, \eqref{s4} give
\be \label{v1}
\left\{\bln
\|\Tdv f_1(t,x)\|_{L^\infty_{v,3}}\le  C\Phi^{1,1}(t,x)(\delta_0+\mathcal{Q}^2(t)),\\
\|\Tdv f_2(t,x)\|_{L^\infty_{v,3}}\le  C\Phi^{2,1}(t,x)(\delta_0+\mathcal{Q}^2(t)).
\eln\right.
\ee

  Let $0 <l< 1$ and $n\ge 4$. Assume that $\mathcal{Q}(t)\le C\delta_0$. Multiplying \eqref{G_4b} by $(1 + t)^l$ and
then taking time integration over $[0, t]$ gives (cf. Theorem 1.3 in \cite{Duan4})
\bmas
&(1+t)^l\mathbb{E}_{n,k}(U)(t)+\mu \intt (1+s)^l\mathbb{D}_{n,k}(U)(s)ds
\nnm\\
\le &C\mathbb{E}_{n,k}(U_0)+Cl\intt (1+s)^{l-1}\mathbb{E}_{n,k}(U)(s)ds\\
&+C\intt (1+s)^l\mathbb{E}_{n}(U)\|P_0f_1(s)\|^2_{L^2_v(L^\infty_x)}ds\\
\le &C\mathbb{E}_{n,k}(U_0)+C(l+\delta_0^2)\intt (1+s)^{l-1}\mathbb{E}_{n,k}(U)(s)ds,
\emas
where we had used
$$\|P_0f_1(t)\|_{L^2_v(L^\infty_x)}\le C(\delta_0+\mathcal{Q}^2(t))(1+t)^{-\frac12}.$$
It follows from \eqref{energy3} and \eqref{energy3b} that
 \bma
&\quad (1+t)^l\mathbb{E}_{n,k}(U)(t)+\mu \intt (1+s)^l\mathbb{D}_{n,k}(U)(s)ds
\nnm\\
&\le C\mathbb{E}_{n,k}(U_0)+C(l+\delta_0^2)\intt (1+s)^{l-1}(\|P_0f_1(s)\|^2_{L^2_{x,v}}+\|B(s)\|^2_{L^2_x})ds\nnm\\
&\quad+C(l+\delta_0^2)\intt (1+s)^{l-1}\mathbb{D}_{n+1,k}(U) ds . \label{l5}
\ema
We use similar estimates from \eqref{G_4b} for $n + 1$  as above to obtain
\bmas
&\quad (1+t)^{l-1}\mathbb{E}_{n+1,k}(U)(t)+(1-l)\intt (1+s)^{l-2}\mathbb{E}_{n+1,k}(U)(s)ds\nnm\\
&\quad+\mu \intt (1+s)^{l-1}\mathbb{D}_{n+1,k}(U)(s)ds \nnm\\
&\le C\mathbb{E}_{n+1,k}(U_0)+C\intt (1+s)^{l-1}\mathbb{E}_{n+1}(U)\|P_0f_1 \|^2_{L^2_v(L^\infty_x)}ds\nnm\\
&\le C\mathbb{E}_{n+1,k}(U_0)+C\delta_0^2\intt (1+s)^{l-2}\mathbb{E}_{n+1}(U) ds. 
\emas
This implies that for $1-l\ge C\delta_0^2$,
\bma
(1+t)^{l-1}\mathbb{E}_{n+1,k}(U) +\mu \intt (1+s)^{l-1}\mathbb{D}_{n+1,k}(U) ds\le C\mathbb{E}_{n+1,k}(U_0). \label{l6}
\ema
By \eqref{l5} and \eqref{l6}, we obtain
 \bmas
&\quad (1+t)^l\mathbb{E}_{n,k}(U)(t)+\mu \intt (1+s)^l\mathbb{D}_{n,k}(U)(s)ds
\nnm\\
&\le C\mathbb{E}_{n+1,k}(U_0)+C\intt (1+s)^{l-1}(\|P_0f_1(s)\|^2_{L^2_{x,v}}+\|B(s)\|^2_{L^2_x})ds
\emas
for $0<l<1$ and $n\ge 4$. Taking $l=1/2+\epsilon$ for a fixed constant $0<\epsilon<1/2$  yields
\bmas
&(1+t)^{\frac12+\epsilon}\mathbb{E}_{n,k}(U)(t)+\mu \intt (1+s)^{\frac12+\epsilon}\mathbb{D}_{n,k}(U)(s)ds
\nnm\\
\le& C\mathbb{E}_{n+1,k}(U_0)+C\intt (1+s)^{-\frac12+\epsilon}(1+s)^{-\frac12}(\delta_0+\mathcal{Q}^2(t))^2 ds\nnm\\
\le& C\mathbb{E}_{n+1,k}(U_0)+C(1+t)^{\epsilon}(\delta_0+\mathcal{Q}^2(t))^2.
\emas  This gives
\be
\mathbb{E}_{7,3}(U)(t)\le C(1+t)^{-\frac12}(\delta_0+\mathcal{Q}^2(t))^2.\label{J_6z}
\ee

Then
\bma
\|\dx^5(E,B)(t)\|_{L^2_x}\le& C(1+t)^{-\frac34}(\|\dx^5V_{0}\|_{Z^2}+\|V_{0}\|_{Z^1}+\|\dx^{6}V_{0}\|_{Z^2})\nnm\\
&+C\intt (1+t-s)^{-\frac54}(\|\dx^5 H_3(s)\|_{L^2_{x,v}}+\|\dx H_3(s)\|_{L^{2,1}})ds\nnm\\
&+C\intt (1+t-s)^{-\frac32}\ln^2(1+t-s)\|\dx^{6}H_3(s)\|_{L^2_{x,v}}ds\nnm\\
\le &C\delta_0(1+t)^{-\frac34}+C(1+t)^{-\frac34}(\delta_0+\mathcal{Q}^2(t))^2.\label{mag_3z}
\ema
In \eqref{mag_3z}, we have used  the following estimates 
$$
\left\{\bln
\|\dx H_3\|_{L^{2,1} }& \le C \mathbb{E}_{7,3}(U)^{\frac12}\mathbb{H}_{5,3}(U)^{\frac12}\le  C(1+t)^{-1}(\delta_0+\mathcal{Q}^2(t))^2,
\\
\|\dx^k H_3\|_{L^2_{x,v} }& \le C \mathbb{E}_{7,3}(U)^{\frac12}(\mathbb{H}_{5,3}(U)^{\frac12}+\sup_x \|U\| )\le  C(1+t)^{-\frac34}(\delta_0+\mathcal{Q}^2(t))^2,
\eln\right.
$$ where $k=5,6$.

Then, by \eqref{G_4a} and \eqref{energy3b} there exists a constant $d_1>0$ so that
\bma
&\Dt \mathcal{H}_{5,3}(U) + d_1 \mathbb{H}_{5,3} (U) \nnm\\
\le& C(\|\dx P_0f_1 \|^2_{L^2_{x,v}}+\|\dx B\|^2_{L^2_{x}}+ \|\dx^5(E,B) \|^2_{L^2_{x}})+C\mathbb{E}_{5}(U)\|P_0f_1\|^2_{L^2_v(L^\infty_x)}.
\ema
This and \eqref{mag_3z} give
\bma
\mathcal{H}_{5,3}(U)(t)\le& e^{-d_1 t}\mathcal{H}_{5,3}(U_0)+C\intt e^{-d_1(t-s)} \|\dx^5(E,B)(s)\|^2_{L^2_{x}}ds\nnm\\
&+C\intt e^{-d_1(t-s)}(\|\dx P_0f_1(s)\|^2_{L^2_{x,v}}+\mathbb{E}_{5}(U)\|P_0f_1\|^2_{L^2_v(L^\infty_x)})ds\nnm\\
\le& C(1+t)^{-\frac32} (\delta_0+\mathcal{Q}^2(t))^2 .\label{other1}
\ema
Combining   \eqref{density1},  \eqref{v1}  and \eqref{other1}, we have
$$
\mathcal{Q}(t)\le C\delta_0+C\mathcal{Q}^2(t),
$$
which implies that  $\mathcal{Q}(t)\le C\delta_0$  if $\delta_0>0$ is small enough. The global existence of the solution can be proved by combining the local existence and the uniform energy estimate, the details are omitted. And this completes the proof of the theorem.
\end{proof}

\bigskip
\noindent {\bf Acknowledgements:}
 The first author was supported by the National Natural Science Foundation of China (Nos.
11931010, 12226326, 12226327), by the key research project of Academy for Multidisciplinary Studies,
Capital Normal University, and by the Capacity Building for Sci-Tech Innovation-Fundamental Scientific
Research Funds (No. 007/20530290068). 
The second author was supported by
a fellowship award from the Research Grants Council of the Hong Kong Special Administrative Region, China (Project no. SRF2021-1S01). And
the third author was supported by the National Science Fund for Excellent Young Scholars (No. 11922107), the National Natural Science Foundation of China  grants (No.  12171104),  Guangxi Natural Science Foundation (No. 2019JJG110010), and the special foundation for Guangxi Ba Gui Scholars.

\bigskip
\noindent {\bf Conflict of Interest Statement:}
The authors declared that they have no conflicts of interest to this work.

\bigskip
\noindent {\bf Data Availability Statement:}
The authors confirm that the data supporting the findings of this study are available within the article and its supplementary materials.


\end{document}